\documentclass[11pt]{amsart}
\usepackage[margin=1.45in]{geometry}

\pdfoutput=1

\usepackage[english]{babel}
\usepackage[utf8]{inputenc}
\usepackage[T1]{fontenc}
\usepackage{microtype}
\usepackage{amsmath}
\usepackage{amsfonts}
\usepackage{amssymb}
\usepackage{amsthm}
\usepackage{eucal}
\usepackage{comment}
\usepackage{hyperref}
\usepackage{tikz-cd}
\tikzcdset{column sep/normal=1.75em}
\tikzcdset{row sep/normal=1.75em}
\usetikzlibrary{babel}
\usepackage{relsize}

\newtheorem{proposition}{Proposition}[section]
\newtheorem{lemma}[proposition]{Lemma}
\newtheorem{theorem}[proposition]{Theorem}
\newtheorem{corollary}[proposition]{Corollary}

\DeclareMathOperator{\Projectives}{\mathsf{P}} %Projective objects of a category
\DeclareMathOperator{\proj}{\mathsf{proj}} %Finitely generated projectives
\DeclareMathOperator{\Ker}{\mathsf{Ker}} %Kernel
\DeclareMathOperator{\Coker}{\mathsf{Coker}} %Cokernel
\DeclareMathOperator{\Image}{\mathsf{Im}} %Image
\DeclareMathOperator{\Hom}{\mathsf{Hom}} %Collection of morphisms
\DeclareMathOperator{\End}{\mathsf{End}} %Endomorphisms
\DeclareMathOperator{\add}{\mathsf{add}} %Category addX
\DeclareMathOperator{\gldim}{\mathsf{gl.dim}} %Global dimension
\DeclareMathOperator{\domdim}{\mathsf{dom.dim}} %Dominant dimension
\DeclareMathOperator{\rdomdim}{\mathsf{r.dom.dim}} %Right dominant dimension
\DeclareMathOperator{\ldomdim}{\mathsf{l.dom.dim}} %Left dominant dimension
\DeclareMathOperator{\Ext}{\mathsf{Ext}} %Ext functor
\DeclareMathOperator{\Tor}{\mathsf{Tor}} %Tor functor
\DeclareMathOperator{\op}{\mathsf{op}} %Opposite
\DeclareMathOperator{\Ab}{\mathsf{Ab}} %Category of abelian groups
\DeclareMathOperator{\pdim}{\mathsf{pd}} %Projective dimension
\DeclareMathOperator{\grade}{\mathsf{grade}} %Projective dimension
\DeclareMathOperator{\Tr}{\mathsf{Tr}} %Transpose
\DeclareMathOperator{\mr}{\mathsf{m}} %Cokernel of C(-,f)
\DeclareMathOperator{\ml}{\mathsf{m}^{\op}} %Cokernel of C(f,-)
\DeclareMathOperator{\SplitEpi}{\mathsf{Split \text{ } Epi}} %Category of split epimorphisms
\DeclareMathOperator{\SplitMono}{\mathsf{Split \text{ } Mono}} %Category of split monomorphisms
\DeclareMathOperator{\E}{\mathsf{E}} %For the functors E^{i}(F)
\DeclareMathOperator{\Yoneda}{\mathsf{Y}} %Yoneda embedding
\let\mod\relax
\DeclareMathOperator{\mod}{\mathsf{mod}} %Category of finitely presented modules
\DeclareMathOperator{\Mod}{\mathsf{Mod}} %Category of modules
\DeclareMathOperator{\Mor}{\mathsf{Mor}} %Category of morphisms
\DeclareMathOperator{\undermod}{\underline{\mathsf{mod}}} %Stable category
\DeclareMathOperator{\underHom}{\underline{\mathsf{Hom}}} %Stable Hom
\newcommand{\eval}{\mathtt{e}} %Evaluation functor
\newcommand{\prov}{\mathtt{p}} %Projectivization functor
\newcommand{\A}{\mathcal{A}}
\newcommand{\B}{\mathcal{B}}
\newcommand{\C}{\mathcal{C}}
\newcommand{\D}{\mathcal{D}}
\newcommand{\I}{\mathcal{I}}

\begin{document}

\title{A functorial approach to $n$-abelian categories}
\author{Vitor Gulisz}
\address{Mathematics Department, Northeastern University, Boston, MA 02115, USA}
\email{gulisz.v@northeastern.edu}
\date{August 13, 2025}

\begin{abstract}
We develop a functorial approach to the study of $n$-abelian categories by reformulating their axioms in terms of their categories of finitely presented functors. Such an approach allows the use of classical homological algebra and representation theory techniques to understand higher homological algebra. As an application, we present two possible generalizations of the axioms ``every monomorphism is a kernel'' and ``every epimorphism is a cokernel'' of an abelian category to $n$-abelian categories. We also specialize our results to modules over rings, thereby describing when the category of finitely generated projective modules over a ring is $n$-abelian. Moreover, we establish a correspondence for $n$-abelian categories with additive generators, which extends the higher Auslander correspondence.
\end{abstract}

\maketitle

\tableofcontents

\section{Introduction}

Higher homological algebra has its origin in \cite{MR2298819}, where Iyama started the study of $n$-cluster tilting subcategories of abelian categories. In order to capture the intrinsic properties of such subcategories, Jasso introduced the notion of an $n$-abelian category in \cite{MR3519980}, which was a breakthrough and crucial to settle higher homological algebra as a subject in its own right. In fact, Jasso proved in \cite[Theorem 3.16]{MR3519980} that an $n$-cluster tilting subcategory of an abelian category is an $n$-abelian category, while the converse was proved, independently, by Kvamme in \cite[Corollary 1.3]{MR4301013} and by Ebrahimi and Nasr-Isfahani in \cite[Theorem A]{MR4514466}. Thus, these concepts coincide, and one can use them interchangeably. There is, however, an advantage of considering $n$-abelian categories, which is the conceptual clarity they provide, given their analogy with abelian categories. Indeed, $n$-abelian categories can be regarded as ``longer'' or ``higher'' versions of abelian categories, with the latter being recovered from the former when $n = 1$. Hence the name ``higher homological algebra'', which concerns not only $n$-abelian categories, but also $n$-exact categories, defined by Jasso in \cite{MR3519980}, $(n+2)$-angulated categories, introduced by Geiss, Keller and Oppermann in \cite{MR3021448}, and $n$-exangulated categories, defined by Herschend, Liu and Nakaoka in \cite{MR4188310}.

In this paper, we propose to advance the conceptual clarity provided by $n$-abelian categories by investigating them through a functorial approach. We follow the philosophy that functors over a category are its ``representations'', and can be useful to understand its properties. This idea was made precise by Mitchell in \cite{MR0294454}, where it was supported that additive functors over a preadditive category are analogues of modules over a ring. The use of functor categories to understand a category itself goes back to the works \cite{MR0209333} of Freyd and \cite{MR0212070} of Auslander. In fact, much of the present paper was influenced by the functorial approach to representation theory employed by Auslander in \cite{MR0212070}, \cite{auslander1971representation}, \cite{MR349747}, \cite{artin.algebras.ii} and several other papers.

The use of functor categories to the study of $n$-abelian categories allows the use of classical homological algebra and representation theory techniques to investigate higher homological algebra. Indeed, throughout this paper, we describe properties of $n$-abelian categories in terms of their categories of finitely presented functors, which are abelian, thereby obtaining classical homological algebra descriptions (with a representation theory flavor) for higher homological algebra phenomena. Once this is done, one can work with functors and then obtain new insights on $n$-abelian categories by ``translating'' results for functors to the language of higher homological algebra. Similar approaches have also been considered in several papers, as in \cite{MR2027559}, \cite{MR3406183}, \cite{EbrahimiNasr-Isfahani}, \cite{MR3659323}, \cite{MR3836680}, \cite{MR4392222}, \cite{MR3638352}, \cite{Klapproth}, \cite{MR4301013}, \cite{MR2471947} and \cite{MR4093080}, where functor categories were used to study classical and higher homological algebra aspects of additive categories.

Let us outline the sections and main results of this paper.

Section \ref{section.2} provides the reader with the conventions we follow, and concise backgrounds on functor categories, ideals of categories, projectively stable categories and higher homological algebra in the context of $n$-abelian categories.

In Section \ref{section.3}, we explore the idea that $n$-kernels and $n$-cokernels are projective resolutions in functor categories. Then we reformulate the axioms of an $n$-abelian category in terms of its categories of finitely presented functors. These reformulations are summarized in Theorem \ref{theorem.2}, the main result of this paper.

Section \ref{section.4} is devoted to the study of von Neumann regular categories. Through the functorial approach, we give an alternative proof of a result of Jasso, which characterizes von Neumann regular categories as the categories that are $n$-abelian for more than one (equivalently, for every) positive integer $n$. This is Proposition \ref{proposition.6}.

In Section \ref{section.5}, we consider the double dual sequence of a finitely presented functor, and some of its properties, which are used to prove further results in this paper. As an application, we show in Corollary \ref{corollary.5} how the cases $n = 1$ and $n \geqslant 2$ for $n$-abelian categories can be distinguished in terms of the existence of finitely presented functors that are reflexive but not projective.

In Section \ref{section.6}, we use the functorial approach to obtain two possible generalizations of the axioms ``every monomorphism is a kernel'' and ``every epimorphism is a cokernel'' of an abelian category to $n$-abelian categories, which are registered in Theorems \ref{theorem.8} and \ref{theorem.11}. These generalizations are stated in terms of $m$-segments, $m$-cosegments, pre-$m$-segments and pre-$m$-cosegments, which are certain sequences consisting of $m$ morphisms, that recover the notions of monomorphisms and epimorphisms when $m = 1$.

Section \ref{section.7} contains results that aim to characterize $n$-abelian categories with enough injectives or enough projectives in terms of the global and dominant dimensions of their categories of finitely presented functors. Its main result is Theorem \ref{theorem.12}.

In Section \ref{section.8}, we specialize previous results to rings and modules over rings. By doing so, we describe in Theorem \ref{theorem.4} when the category of finitely generated projective modules over a ring is $n$-abelian. Such a description leads to a correspondence for $n$-abelian categories with additive generators, which extends the higher Auslander correspondence to rings that are not necessarily Artin algebras. This is Theorem \ref{theorem.6}.

In Appendix \ref{section.9}, we introduce the tensor product of finitely presented functors, and we prove in Theorem \ref{theorem.5} that the global dimensions of the categories of finitely presented contravariant and covariant functors over a coherent category coincide.

In Appendix \ref{section.10}, we state alternative ways of expressing part of the axioms given in Theorem \ref{theorem.2}, which are related to the grade of finitely presented functors. Consequently, we obtain another characterization of $n$-abelian categories in Theorem \ref{theorem.14}.

Appendix \ref{section.11} covers a few known results on $n$-abelian categories from the functorial perspective, as a further exposition of the functorial approach.

\section{Preliminaries}\label{section.2}

\subsection{Conventions}

Throughout this paper, $\C$ will be an additive and idempotent complete category (see the definitions below) and $n$ will be a positive integer. All subcategories are assumed to be full.

Let $\A$ be a category. Given two objects $X,Y \in \A$, we denote the collection of morphisms from $X$ to $Y$ in $\A$ by $\A(X,Y)$. We write $gf \in \A(X,Z)$ for the composition of two morphisms $f \in \A(X,Y)$ and $g \in \A(Y,Z)$ in $\A$, and we denote the identity morphism of an object $X \in \A$ by $1_{X}$ or $1$.

An object $P \in \A$ is said to be \textit{projective} if for every epimorphism $g \in \A(X,Y)$ and every morphism $f \in \A(P,Y)$, there is a morphism $h \in \A(P,X)$ for which $f = gh$. We say that $\A$ \textit{has enough projectives} if for every $X \in \A$, there is an epimorphism $P \to X$ in $\A$ with $P$ projective. Dually, an object $I \in \A$ is \textit{injective} if for every monomorphism $g \in \A(X,Y)$ and every morphism $f \in \A(X,I)$, there is a morphism $h \in \A(Y,I)$ such that $f = hg$. We say that $\A$ \textit{has enough injectives} if for every $X \in \A$, there is a monomorphism $X \to I$ in $\A$ with $I$ injective.

We say that a category $\A$ is \textit{preadditive} if $\A(X,Y)$ is an abelian group for every $X,Y \in \A$ and the composition of morphisms in $\A$ is bilinear. If $\A$ is a preadditive category that has a zero object and every finite collection of objects in $\A$ admits a direct sum, then $\A$ is called \textit{additive}. We say that $\A$ is \textit{idempotent complete} if every idempotent morphism $f$ in $\A$ (that is, $f^{2} = f$) can be written as $f = gh$, where $g$ and $h$ are morphisms in $\A$ such that $hg = 1$.

Let $\A$ be an additive and idempotent complete category. Given a subcategory $\B$ of $\A$, we denote by $\add \B$ the subcategory of $\A$ consisting of the direct summands of finite direct sums of objects in $\B$. We remark that $\add \B$ is an additive and idempotent complete category. Moreover, observe that $\add \B = \B$ if and only if $\B$ is closed under finite direct sums and direct summands in $\A$. If $\B$ has only one object, say $X$, then we denote $\add \B$ by $\add X$. We say that $\A$ \textit{has an additive generator} if there is some $X \in \A$ for which $\A = \add X$, and in this case, $X$ is called an \textit{additive generator} of $\A$. We will denote the endomorphism ring of $X \in \A$ by $\End(X)$.

Let $\A$ be an abelian category. For each $X \in \A$, we denote the \textit{projective dimension} of $X$ by $\pdim X$. We also denote the \textit{global dimension} of $\A$ by $\gldim (\A)$. Moreover, $\Ab$ will stand for the category of abelian groups.

Finally, by a \textit{ring}, we mean an associative ring with identity. For a ring $\Lambda$, a right $\Lambda$-module will be referred to as \textit{$\Lambda$-module}, and a left $\Lambda$-module will be regarded as a $\Lambda^{\op}$-module. We denote the category of $\Lambda$-modules by $\Mod \Lambda$, while the category of finitely presented $\Lambda$-modules will be denoted by $\mod \Lambda$. We also let $\proj \Lambda$ be the category of finitely generated projective $\Lambda$-modules.

\subsection{Functor categories}

In this subsection, we review some basic definitions and results concerning functor categories. The reader may find more details in \cite{MR0212070}, \cite{auslander1971representation}, \cite{MR349747}, \cite{MR166240}, \cite{MR0209333}, \cite{MR1487973}, \cite{MR4327095} and \cite{MR0294454}.

We call a contravariant additive functor from $\C$ to $\Ab$ a \textit{right $\C$-module}. We denote by $\Mod \C$ the category whose objects are the right $\C$-modules and whose morphisms are natural transformations. That is, for $F, G \in \Mod \C$, the collection of morphisms from $F$ to $G$ is given by the natural transformations from $F$ to $G$, and we denote it by $\Hom(F,G)$. In general, $\Hom(F,G)$ might not be a set, although it is indeed a set if $\C$ is skeletally small, in which case $\Mod \C$ is a locally small category. A \textit{left $\C$-module} is a covariant additive functor from $\C$ to $\Ab$, that is, a right $\C^{\op}$-module. We give preference to work with right $\C$-modules, with the case of left $\C$-modules being recovered by taking $\C^{\op}$ in place of $\C$. We will call a right $\C$-module a \textit{$\C$-module}, and a left $\C$-module will be referred to as a $\C^{\op}$-module.

A sequence \[ \begin{tikzcd}
F \arrow[r, "\alpha"] & G \arrow[r, "\beta"] & H
\end{tikzcd} \] in $\Mod \C$ is called \textit{exact} when \[ \begin{tikzcd}
F(X) \arrow[r, "\alpha_{X}"] &[0.2em] G(X) \arrow[r, "\beta_{X}"] &[0.2em] H(X)
\end{tikzcd} \] is exact in $\Ab$ for every $X \in \C$. In words, a sequence in $\Mod \C$ is exact if and only if each of its components is an exact sequence of abelian groups.

When $\C$ is skeletally small, the category $\Mod \C$ is abelian, and a sequence in $\Mod \C$ is exact in the above sense if and only if it is exact as a sequence in an abelian category. We remark that, in this case, kernels, cokernels and direct sums in $\Mod \C$ are computed componentwise. If $\C$ is not skeletally small, then $\Mod \C$ might not be locally small, which prevents us from saying that $\Mod \C$ is a preadditive category. However, $\Mod \C$ still has kernels, cokernels and direct sums, which are computed componentwise, and basic results that hold for abelian categories also hold for $\Mod \C$.

It follows from the Yoneda lemma that, for every $X \in \C$, the representable functor $\C(-,X)$ is a projective object in $\Mod \C$. We say that a $\C$-module $F$ is \textit{finitely generated} if there is an epimorphism $\C(-,X) \to F$ in $\Mod \C$ for some $X \in \C$. We denote by $\proj \C$ the subcategory of $\Mod \C$ consisting of the finitely generated projective $\C$-modules. We recall that the Yoneda embedding $\C \to \Mod \C$ induces an equivalence of categories $\C \to \proj \C$ since $\C$ is additive and idempotent complete. Therefore, every object in $\proj \C$ is isomorphic to $\C(-,X)$ for some $X \in \C$, and every morphism in $\proj \C$ is isomorphic to $\C(-,f)$ for some morphism $f$ in $\C$.

A $\C$-module $F$ is \textit{finitely presented} if there is an exact sequence  \[ \begin{tikzcd}
{\C(-,X)} \arrow[r, "{\C(-,f)}"] &[1.2em] {\C(-,Y)} \arrow[r] & F \arrow[r] & 0
\end{tikzcd} \] in $\Mod \C$ for some $f \in \C(X,Y)$. Such an exact sequence is called a \textit{projective presentation} of $F$. We denote by $\mod \C$ the subcategory of $\Mod \C$ consisting of the finitely presented $\C$-modules. The category $\mod \C$ is always locally small, regardless if $\C$ is skeletally small or not. We also remark that the category $\mod \C$ is additive, it has cokernels, its projective objects are given by $\proj \C$ and it has enough projectives.

The next result is folklore, see, for example, \cite[page 18]{MR0294454} and \cite[Proposition 2.3]{MR1487973}. It tells us, in particular, that the category $\mod \C$ determines $\C$ up to equivalence of categories. In fact, $\C$ is recovered from $\mod \C$ by taking its projective objects.

\begin{proposition}\label{proposition.16}
Let $\B$ and $\C$ be additive and idempotent complete categories. The following are equivalent:
\begin{enumerate}
    \item[(a)] The categories $\B$ and $\C$ are equivalent.
    \item[(b)] The categories $\mod \B$ and $\mod \C$ are equivalent.
\end{enumerate}
If $\B$ and $\C$ are skeletally small, then the above items are also equivalent to:
\begin{enumerate}
    \item[(c)] The categories $\Mod \B$ and $\Mod \C$ are equivalent.
\end{enumerate}
\end{proposition}

\begin{proof}
It is easy to see that if $\B$ and $\C$ are equivalent categories, then so are $\mod \B$ and $\mod \C$. Conversely, if there is an equivalence of categories $\mod \B \approx \mod \C$, then it induces an equivalence on the subcategories of projective objects, hence there is an equivalence $\proj \B \approx \proj \C$. Since the Yoneda embedding induces equivalences $\B \approx \proj \B$ and $\C \approx \proj \C$, we deduce that there is an equivalence $\B \approx \C$. Therefore, items (a) and (b) are equivalent.

Next, assume that $\B$ and $\C$ are skeletally small. Again, it is easy to see that if $\B$ and $\C$ are equivalent categories, then so are $\Mod \B$ and $\Mod \C$. Conversely, if there is an equivalence of categories $\Mod \B \approx \Mod \C$, then it induces an equivalence $\proj \B \approx \proj \C$, see \cite[Chapter 5, Exercise G]{MR166240} or \cite[Theorem 2.7.2]{VitorGulisz}. Hence, as above, we conclude that there is an equivalence $\B \approx \C$. Thus, items (a) and (c) are equivalent.
\end{proof}

As a final observation, we do not assume $\C$ to be skeletally small in this paper. This will not cause any issues, as we mainly work with the category $\mod \C$, and only componentwise constructions and arguments are used when dealing with $\Mod \C$.

\subsection{Ideals}\label{subsection.ideals}

As in \cite{MR116045}, a \textit{two-sided ideal} of $\C$, or for short, an \textit{ideal} of $\C$, is a collection $\I$ of morphisms in $\C$ with the property that $\I \cap \C(X,Y)$ is a subgroup of $\C(X,Y)$ for all $X,Y \in \C$, and such that $fgh \in \I$ for all morphisms $f,h \in \C$ and $g \in \I$, whenever the composition $fgh$ makes sense. For $X,Y \in \C$, we denote $\I \cap \C(X,Y) = \I(X,Y)$. If $\I$ is an ideal of $\C$, then we define the category $\C / \I$ as follows: the objects of $\C / \I$ are the same as the objects of $\C$, but given $X,Y \in \C$, we set \[ (\C / \I)(X,Y) = \C(X,Y) / \I(X,Y), \] where the above expression is a quotient of abelian groups.

For the rest of this subsection, let $\B$ be a subcategory of $\C$ such that $\add \B = \B$.

We say that a morphism $f \in \C(X,Y)$ \textit{factors through} $\B$ if there are $g \in \C(B,Y)$ and $h \in \C(X,B)$ with $B \in \B$ such that $f = gh$. We define the \textit{ideal generated by $\B$}, which we denote by $\langle \B \rangle$, to be the ideal of $\C$ consisting of the morphisms in $\C$ that factor through $\B$. The following proposition is well known. It says that taking the quotient of $\C$ by $\langle \B \rangle$ ``annihilates'' precisely the objects in $\B$.

\begin{proposition}\label{proposition.23}
The following are equivalent for an object $X \in \C$: \begin{enumerate}
    \item[(a)] $X \simeq 0$ in $\C / \langle \B \rangle$.
    \item[(b)] $1_{X} \in \langle \B \rangle$.
    \item[(c)] $X \in \B$.
\end{enumerate}
\end{proposition}

\begin{proof}
Left to the reader.
\end{proof}

The next result, due to Heller, tells us that isomorphic objects in $\C / \langle \B \rangle$ differ in $\C$ up to direct sums of objects in $\B$.

\begin{theorem}\label{theorem.1}
Given two objects $X,Y \in \C$, we have $X \simeq Y$ in $\C / \langle \B \rangle$ if and only if there are $A,B \in \B$ such that $X \oplus A \simeq Y \oplus B$ in $\C$.
\end{theorem}

\begin{proof}
See \cite[Theorem 2.2]{MR116045}. Observe that, because $\C$ is idempotent complete, $\C$ is a category ``with cancellation'', which is assumed in \cite[Theorem 2.2]{MR116045}.
\end{proof}

In view of Theorem \ref{theorem.1}, we adopt the following notation (whenever it is convenient): we denote an object $X \in \C$ by $X \oplus \B$ when it is considered to be in the category $\C / \langle \B \rangle$. In this notation, Theorem \ref{theorem.1} says that, for $X,Y \in \C$, we have $X \oplus \B \simeq Y \oplus \B$ if and only if there are $A,B \in \B$ with $X \oplus A \simeq Y \oplus B$.

\subsection{The projectively stable
category}

Let $\A$ be an additive and idempotent complete category, and let $\Projectives (\A)$ be the subcategory of projective objects of $\A$. The \textit{projectively stable
category} of $\A$ is the category $\A / \langle \Projectives (\A) \rangle$, which we denote by $\underline{\A}$. Note that, for $X,Y \in \A$, the collection of morphisms from $X$ to $Y$ in $\underline{\A}$ is given by the quotient of abelian groups \[ \underline{\A}(X,Y) = \A(X,Y) / \langle \Projectives (\A) \rangle (X,Y). \] When $\A = \mod \C$, we have $\Projectives (\A) = \proj \C$ and we write $\underline{\A} = \undermod \C$. Moreover, in this case, for $F, G \in \mod \C$, we denote $\underline{\A}(F,G)$ by $\underHom (F,G)$.

Now, suppose that $\A$ is an abelian category with enough projectives, and take some $X \in \A$. If \[ \begin{tikzcd}
0 \arrow[r] & K \arrow[r] & P \arrow[r] & X \arrow[r] & 0
\end{tikzcd} \] is an exact sequence in $\A$ with $P \in \Projectives (\A)$, then $K$ is called a \textit{syzygy} of $X$ and we denote it by $\Omega X$. Despite this notation, observe that syzygies are not unique. However, if \[ \begin{tikzcd}
0 \arrow[r] & J \arrow[r] & Q \arrow[r] & X \arrow[r] & 0
\end{tikzcd} \] is another exact sequence in $\A$ with $Q \in \Projectives (\A)$, then it follows from Schanuel's lemma that $K \oplus Q \simeq J \oplus P$. Therefore, we conclude from Theorem \ref{theorem.1} that $K$ and $J$ are isomorphic in $\underline{\A}$. In symbols, $K \oplus \Projectives (\A) \simeq J \oplus \Projectives (\A)$. Hence a syzygy of $X$ becomes unique up to isomorphism in the projectively stable category $\underline{\A}$. Given a positive integer $j$, we define a \textit{$j$\textsuperscript{th} syzygy} of $X$ recursively by letting it be a syzygy of a $(j-1)$\textsuperscript{th} syzygy of $X$, agreeing that a $0$\textsuperscript{th} syzygy of $X$ is $X$. We denote it by $\Omega^{j} X$, which is unique up to isomorphism in the category $\underline{\A}$. We remark that the syzygy defines a functor $\Omega : \underline{\A} \to \underline{\A}$ in the obvious way. Similarly, given a positive integer $j$, the $j$\textsuperscript{th} syzygy defines a functor $\Omega^{j} : \underline{\A} \to \underline{\A}$, which coincides with the composition of $j$ copies of $\Omega$. For more details, see \cite[Section 3]{MR116045} or \cite[Proposition 2.7]{MR0269685}.

\subsection{Higher homological algebra}

In this subsection, we recall some basic definitions from \cite{MR3519980}.

Let $f \in \C(X,Y)$ be a morphism in $\C$. An \textit{$n$-kernel} of $f$ is a sequence \[ \begin{tikzcd}
X_{n} \arrow[r, "f_{n}"] &[0.05em] \cdots \arrow[r, "f_{2}"] & X_{1} \arrow[r, "f_{1}"] & X
\end{tikzcd} \] of morphisms in $\C$ with the property that \[ \begin{tikzcd}
0 \arrow[r] & {\C(-,X_{n})} \arrow[r, "{\C(-,f_{n})}"] &[1.45em] \cdots \arrow[r, "{\C(-,f_{2})}"] &[1.4em] {\C(-,X_{1})} \arrow[r, "{\C(-,f_{1})}"] &[1.4em] {\C(-,X)} \arrow[r, "{\C(-,f)}"] &[1.2em] {\C(-,Y)}
\end{tikzcd} \] is exact in $\Mod \C$. Dually, an \textit{$n$-cokernel} of $f$ is a sequence \[ \begin{tikzcd}
Y \arrow[r, "g_{1}"] & Y_{1} \arrow[r, "g_{2}"] & \cdots \arrow[r, "g_{n}"] &[0.05em] Y_{n}
\end{tikzcd} \] of morphisms in $\C$ for which \[ \begin{tikzcd}
0 \arrow[r] & {\C(Y_{n},-)} \arrow[r, "{\C(g_{n},-)}"] &[1.45em] \cdots \arrow[r, "{\C(g_{2},-)}"] &[1.4em] {\C(Y_{1},-)} \arrow[r, "{\C(g_{1},-)}"] &[1.4em] {\C(Y,-)} \arrow[r, "{\C(f,-)}"] &[1.2em] {\C(X,-)}
\end{tikzcd} \] is exact in $\Mod \C^{\op}$.

We say that $\C$ \textit{has $n$-kernels} if every morphism in $\C$ has an $n$-kernel, and that $\C$ \textit{has $n$-cokernels} if every morphism in $\C$ has an $n$-cokernel.

An additive and idempotent complete category $\C$ is called \textit{$n$-abelian} if it satisfies the following axioms:

\begin{enumerate}
    \item[(A1)] $\C$ has $n$-kernels.
    \item[(A1$^{\op}$)] $\C$ has $n$-cokernels.
    \item[(A2)] For every monomorphism $f \in \C(X,Y)$ in $\C$ and for every $n$-cokernel \[ \begin{tikzcd}
Y \arrow[r, "g_{1}"] & Y_{1} \arrow[r, "g_{2}"] & \cdots \arrow[r, "g_{n}"] &[0.05em] Y_{n}
\end{tikzcd} \] of $f$, the sequence \[ \begin{tikzcd}
X \arrow[r, "f"] & Y \arrow[r, "g_{1}"] & \cdots \arrow[r, "g_{n-1}"] &[0.7em] Y_{n-1}
\end{tikzcd} \] is an $n$-kernel of $g_{n}$.
    \item[(A2$^{\op}$)] For every epimorphism $f \in \C(X,Y)$ in $\C$ and for every $n$-kernel \[ \begin{tikzcd}
X_{n} \arrow[r, "f_{n}"] &[0.05em] \cdots \arrow[r, "f_{2}"] & X_{1} \arrow[r, "f_{1}"] & X
\end{tikzcd} \] of $f$, the sequence \[ \begin{tikzcd}
X_{n-1} \arrow[r, "f_{n-1}"] &[0.7em] \cdots \arrow[r, "f_{1}"] & X \arrow[r, "f"] & Y
\end{tikzcd} \] is an $n$-cokernel of $f_{n}$.
\end{enumerate}

Rigorously, we should label the above axioms with the positive integer $n$. However, $n$ is fixed throughout this text, so we feel comfortable not doing so.

Note that $1$-kernels, $1$-cokernels and $1$-abelian categories coincide with kernels, cokernels and abelian categories, respectively. Hence $n = 1$ recovers the classical case.

Finally, recall that an \textit{$n$-exact sequence} in $\C$ is a sequence \[ \begin{tikzcd}
Z_{n+1} \arrow[r, "h_{n+1}"] &[0.7em] Z_{n} \arrow[r, "h_{n}"] &[0.05em] \cdots \arrow[r, "h_{2}"] & Z_{1} \arrow[r, "h_{1}"] & Z_{0}
\end{tikzcd} \] of morphisms in $\C$ with the property that {\relsize{-0.5} \[ \begin{tikzcd}
0 \arrow[r] & {\C(-,Z_{n+1})} \arrow[r, "{\C(-,h_{n+1})}"] &[2.7em] {\C(-,Z_{n})} \arrow[r, "{\C(-,h_{n})}"] &[1.85em] \cdots \arrow[r, "{\C(-,h_{2})}"] &[1.8em] {\C(-,Z_{1})} \arrow[r, "{\C(-,h_{1})}"] &[1.8em] {\C(-,Z_{0})}
\end{tikzcd} \]} \hspace{-0.667em} and {\relsize{-0.5} \[ \begin{tikzcd}
0 \arrow[r] & {\C(Z_{0},-)} \arrow[r, "{\C(h_{1},-)}"] &[1.8em] {\C(Z_{1},-)} \arrow[r, "{\C(h_{2},-)}"] &[1.8em] \cdots \arrow[r, "{\C(h_{n},-)}"] &[1.85em] {\C(Z_{n},-)} \arrow[r, "{\C(h_{n+1},-)}"] &[2.7em] {\C(Z_{n+1},-)}
\end{tikzcd} \]} \hspace{-0.88em} are exact sequences in $\Mod \C$ and in $\Mod \C^{\op}$, respectively. Observe that $1$-exact sequences are the same as kernel-cokernel pairs (which are also known as short exact sequences). Thus, we might think of $n$-exact sequences as ``$n$-kernel-$n$-cokernel pairs''.

\section{The functorial approach}\label{section.3}

The idea behind the functorial approach is that ``representations'' of a category $\C$, which are modules over $\C$, can be used to understand $\C$. Following this perspective, we might also choose to restrict our attention to modules that are finitely presented, and then aim to understand $\C$ by investigating the categories $\mod \C$ and $\mod \C^{\op}$. The goal of this paper is to develop this idea for the case when $\C$ is an $n$-abelian category. In this section, we take the first steps into this direction, and we reformulate the axioms for a category $\C$ to be $n$-abelian in terms of $\mod \C$ and $\mod \C^{\op}$.

We begin by giving equivalent statements for the axioms (A1) and (A1$^{\op}$) of an $n$-abelian category. Later on, we will do the same for the axioms (A2) and (A2$^{\op}$).

Let $f \in \C(X,Y)$ and $g \in \C(Y,Z)$. Recall that $f$ is a \textit{weak kernel} of $g$ if \[ \begin{tikzcd}
{\C(-,X)} \arrow[r, "{\C(-,f)}"] &[1.2em] {\C(-,Y)} \arrow[r, "{\C(-,g)}"] &[1.2em] {\C(-,Z)}
\end{tikzcd} \] is exact in $\Mod \C$. Dually, $g$ is a \textit{weak cokernel} of $f$ if \[ \begin{tikzcd}
{\C(Z,-)} \arrow[r, "{\C(g,-)}"] &[1.2em] {\C(Y,-)} \arrow[r, "{\C(f,-)}"] &[1.2em] {\C(X,-)}
\end{tikzcd} \] is exact in $\Mod \C^{\op}$.

We say that $\C$ \textit{has weak kernels} if every morphism in $\C$ has a weak kernel, and that $\C$ \textit{has weak cokernels} if every morphism in $\C$ has a weak cokernel.

Note that if $\C$ has $n$-kernels, then $\C$ has weak kernels. Dually, if $\C$ has $n$-cokernels, then $\C$ has weak cokernels. It is, therefore, of our interest to consider categories having weak kernels and weak cokernels. Regarding these properties, we have the following well known result:

\begin{proposition}\label{proposition.1}
An additive and idempotent complete category $\C$ has weak kernels if and only if $\mod \C$ is an abelian category.
\end{proposition}

\begin{proof}
See \cite[Proposition 2.1]{MR0212070} and \cite[page 41]{auslander1971representation}, or \cite[Theorem 1.4]{MR0209333} or \cite[Lemma 2.1.6]{MR4327095}.
\end{proof}

Observe that, by taking $\C^{\op}$ in place of $\C$ in Proposition \ref{proposition.1}, we conclude that $\C$ has weak cokernels if and only if $\mod \C^{\op}$ is an abelian category.

Before we continue, let us introduce some nomenclature. We say that $\C$ is \textit{right coherent} if the category $\mod \C$ is abelian. Dually, $\C$ is called \textit{left coherent} if the category $\mod \C^{\op}$ is abelian. If $\C$ is both right and left coherent, then we say that $\C$ is \textit{coherent}.

It follows from the Yoneda lemma that if $\C$ is right coherent, then a sequence in $\mod \C$ is exact in $\mod \C$ if and only if it is exact in $\Mod \C$, see \cite[page 41]{auslander1971representation}. Similarly, if $\C$ is left coherent, then a sequence in $\mod \C^{\op}$ is exact in $\mod \C^{\op}$ if and only if it is exact in $\Mod \C^{\op}$.

We can now describe the axioms (A1) and (A1$^{\op}$) in terms of the categories $\mod \C$ and $\mod \C^{\op}$, respectively. In fact, consider the following axioms:

\begin{enumerate}
    \item[(F1)] $\C$ is right coherent and $\gldim (\mod \C) \leqslant n+1$.
    \item[(F1$^{\op}$)] $\C$ is left coherent and $\gldim (\mod \C^{\op}) \leqslant n+1$.
\end{enumerate}

\begin{proposition}\label{proposition.2}
The axioms \textup{(A1)} and \textup{(A1$^{\op}$)} are equivalent to \textup{(F1)} and \textup{(F1$^{\op}$)}, respectively.
\end{proposition}

\begin{proof}
We only prove that (A1) and (F1) are equivalent. Then, by taking $\C^{\op}$ in place of $\C$, we can deduce that (A1$^{\op}$) and (F1$^{\op}$) are equivalent.

Suppose that $\C$ satisfies the axiom (A1). Then $\C$ has weak kernels, and we conclude that $\C$ is right coherent, by Proposition \ref{proposition.1}. Moreover, let $F \in \mod \C$ be arbitrary, and take a projective presentation \[ \begin{tikzcd}
{\C(-,X)} \arrow[r, "{\C(-,f)}"] &[1.2em] {\C(-,Y)} \arrow[r] & F \arrow[r] & 0
\end{tikzcd} \] of $F$ in $\mod \C$ with $f \in \C(X,Y)$. Then, by taking an $n$-kernel of $f$, we get a projective resolution \[ \begin{tikzcd}
0 \arrow[r] & {\C(-,X_{n})} \arrow[r] & \cdots \arrow[r] & {\C(-,X_{1})} \arrow[r] & {\C(-,X)} \arrow[r] & {\C(-,Y)} \arrow[r] & F \arrow[r] & 0
\end{tikzcd} \] of $F$ in $\mod \C$, so that $\pdim F \leqslant n + 1$. Therefore, $\gldim (\mod \C) \leqslant n + 1$.

Conversely, assume that $\C$ satisfies the axiom (F1). Given a morphism $f \in \C(X,Y)$, let $F \in \mod \C$ be the cokernel of $\C(-,f)$, and consider the projective presentation \[ \begin{tikzcd}
{\C(-,X)} \arrow[r, "{\C(-,f)}"] &[1.2em] {\C(-,Y)} \arrow[r] & F \arrow[r] & 0
\end{tikzcd} \] of $F$ in $\mod \C$ induced by $f$. Because $\gldim (\mod \C) \leqslant n + 1$, we can extend this projective presentation to a projective resolution of $F$ in $\mod \C$ as the one in the previous paragraph, which gives an $n$-kernel of $f$, by the Yoneda lemma.
\end{proof}

We define a \textit{pre-$n$-abelian} category to be an additive and idempotent complete category that satisfies the axioms (A1) and (A1$^{\op}$). By Proposition \ref{proposition.2}, an additive and idempotent complete category is pre-$n$-abelian if and only if it satisfies the axioms (F1) and (F1$^{\op}$).

At this point, an intriguing phenomenon is brought to light. In fact, note that, by Theorem \ref{theorem.5}, if $\C$ is coherent, then the global dimensions of $\mod \C$ and $\mod \C^{\op}$ coincide. Thus, if $\C$ is coherent, then $\C$ satisfies (F1) if and only if $\C$ satisfies (F1)$^{\op}$. By Propositions \ref{proposition.1} and \ref{proposition.2}, this means that if $\C$ has weak kernels and weak cokernels, then $\C$ satisfies (A1) if and only if $\C$ satisfies (A1)$^{\op}$. There is, therefore, a redundancy in the definition of a pre-$n$-abelian category (and also in the definition of an $n$-abelian category).

\begin{proposition}
Let $\C$ be an additive and idempotent complete category. If $\C$ has weak kernels and weak cokernels, then $\C$ has $n$-kernels if and only if $\C$ has $n$-cokernels.
\end{proposition}

\begin{proof}
Follows from the previous discussion.
\end{proof}

Before moving on to the axioms (A2) and (A2$^{\op}$), let us pause to ponder on how $n$-kernels and $n$-cokernels were considered in the proof of Proposition \ref{proposition.2}.

Given a morphism $f$ in $\C$, let us define $\mr(f) = \Coker \C(-,f)$, which belongs to $\mod \C$, and $\ml(f) = \Coker \C(f,-)$, which lies in $\mod \C^{\op}$. The proof of Proposition \ref{proposition.2} shows that if $\C$ has $n$-kernels, then the $n$-kernels of $f$ are obtained by taking projective resolutions of $\mr(f)$ in $\mod \C$ that extend the projective presentation of $\mr(f)$ induced by $f$. Dually, if $\C$ has $n$-cokernels, then the $n$-cokernels of $f$ are obtained by taking projective resolutions of $\ml(f)$ in $\mod \C^{\op}$ that extend the projective presentation of $\ml(f)$ induced by $f$.\footnote{With this perspective, we see that $n$-kernels and $n$-cokernels are not unique if $n \geqslant 2$, as it was observed in \cite[Remark 2.3]{MR3519980}. Clearly, this is the case because projective resolutions are not unique. One way of achieving uniqueness for $n$-kernels and $n$-cokernels would be by considering minimal projective resolutions of objects in $\mod \C$ and in $\mod \C^{\op}$, respectively. We do not explore this subject in this paper, but we refer the reader to \cite[Section 3.2]{VitorGulisz} for some insights on it.}

Observe how in the above paragraph we insisted that $n$-kernels of $f$ do not come from arbitrary projective resolutions of $\mr(f)$, but by those extending the projective presentation of $\mr(f)$ induced by $f$. Dually, for $n$-cokernels of $f$. This is because $\mr(f)$ and $\ml(f)$ also have projective presentations induced by morphisms other than $f$. Motivated by \cite[Chapter III]{auslander1971representation} and \cite[Section 3]{MR2027559}, let us elaborate on this comment.

Let $\Mor \C$ be the category of morphisms in $\C$. Given $f \in \C(X,Y)$, $g \in \C(Z,W)$ and a morphism $(x,y) : f \to g$ in $\Mor \C$, that is, a commutative square \[ \begin{tikzcd}
X \arrow[r, "f"] \arrow[d, "x"'] & Y \arrow[d, "y"] \\
Z \arrow[r, "g"']                & W               
\end{tikzcd} \] in $\C$, we define $\mr(x,y)$ to be the unique morphism $\mr(f) \to \mr(g)$ in $\mod \C$ which makes the diagram with exact rows below commute. \[ \begin{tikzcd}
{\C(-,X)} \arrow[r, "{\C(-,f)}"] \arrow[d, "{\C(-,x)}"'] &[1.2em] {\C(-,Y)} \arrow[d, "{\C(-,y)}"] \arrow[r] & \mr(f) \arrow[r] \arrow[d, "{\mr(x,y)}"] & 0 \\
{\C(-,Z)} \arrow[r, "{\C(-,g)}"']                        & {\C(-,W)} \arrow[r]                        & \mr(g) \arrow[r]                                & 0
\end{tikzcd} \] It is easy to see that this defines an additive functor $\mr : \Mor \C \to \mod \C$, which is dense and full. However, the functor $\mr$ is not an equivalence of categories, because it is not faithful. In fact, we can check that the following are equivalent for a morphism $(x,y)$ in $\Mor \C$ as above: \begin{enumerate}
    \item[(a)] $\mr(x,y) = 0$.
    \item[(b)] There is some $r \in \C(Y,Z)$ with $y = gr$.
    \item[(c)] $(x,y)$ factors in $\Mor \C$ through a split epimorphism in $\C$. 
\end{enumerate} Nonetheless, if we let $\SplitEpi$ be the subcategory of $\Mor \C$ consisting of the split epimorphisms in $\C$, then, by the above discussion, the functor $\mr$ induces an equivalence of categories $\Mor \C / \langle \SplitEpi \rangle \to \mod \C$, which we also denote by $\mr$. In the references that we have previously mentioned, this equivalence is given in \cite[page 38]{auslander1971representation} and \cite[Corollary 3.9]{MR2027559}.

Dually, there is also an additive (but contravariant) functor $\ml : \Mor \C \to \mod \C^{\op}$, which is dense and full. Under the previous settings, we can verify that the following are equivalent:
\begin{enumerate}
    \item[(a)] $\ml(x,y) = 0$.
    \item[(b)] There is some $r \in \C(Y,Z)$ with $x = rf$.
    \item[(c)] $(x,y)$ factors in $\Mor \C$ through a split monomorphism in $\C$. 
\end{enumerate} Consequently, by letting $\SplitMono$ be the subcategory of $\Mor \C$ consisting of the split monomorphisms in $\C$, we conclude that $\ml$ induces an anti-equivalence of categories $\Mor \C / \langle \SplitMono \rangle \to \mod \C^{\op}$, which we also denote by $\ml$.

We can now compare how different the projective presentations of objects in $\mod \C$ and in $\mod \C^{\op}$ are. Let $F \in \mod \C$ and suppose that $f,g \in \Mor \C$ are such that $\mr(f) \simeq F$ and $\mr(g) \simeq F$. Then\footnote{Recall the notation introduced in Subsection \ref{subsection.ideals}.} $\mr(f \oplus \SplitEpi) \simeq \mr(g \oplus \SplitEpi)$, which implies that $f \oplus \SplitEpi \simeq g \oplus \SplitEpi$. Hence it follows from Theorem \ref{theorem.1} that there are $s,t \in \SplitEpi$ such that $f \oplus s \simeq g \oplus t$ in $\Mor \C$. Therefore, $f$ and $g$ differ up to direct sums of split epimorphisms in $\C$. Dually, if $F \in \mod \C^{\op}$ and $f,g \in \Mor \C$ satisfy $\ml(f) \simeq F$ and $\ml(g) \simeq F$, then there are $s,t \in \SplitMono$ with $f \oplus s \simeq g \oplus t$ in $\Mor \C$. In this case, $f$ and $g$ differ up to direct sums of split monomorphisms in $\C$

To summarize our pondering, we can say that $n$-kernels and $n$-cokernels in $\C$ are given by well-chosen projective resolutions of objects in $\mod \C$ and in $\mod \C^{\op}$, respectively. But there seems to be an issue with this approach to $n$-kernels and $n$-cokernels: it relies on two distinct categories, namely, $\mod \C$ and $\mod \C^{\op}$. In view of this apparent issue, it would be interesting to have a tool which could enable us to transit between $\mod \C$ and $\mod \C^{\op}$ in such a way that for a morphism $f$ in $\C$, the morphism $\C(-,f)$ in $\mod \C$ would be sent to $\C(f,-)$ in $\mod \C^{\op}$, and vice versa. Let us explain how we can obtain such a tool.

Following \cite[page 336]{MR342505}, given $F \in \mod \C$, let $F^{\ast}$ be the $\C^{\op}$-module defined by $F^{\ast}(X) = \Hom(F,\C(-,X))$ for each $X \in \C$ and $F^{\ast}(f) = \Hom(F,\C(-,f))$ for each morphism $f$ in $\C$. We call $F^{\ast}$ the \textit{dual} of $F$. If $\alpha \in \Hom(F,G)$ for $F,G \in \mod \C$, then we let $\alpha^{\ast} : G^{\ast} \to F^{\ast}$ be the morphism of $\C^{\op}$-modules given by $(\alpha^{\ast})_{X} = \Hom(\alpha, \C(-,X))$ for each $X \in \C$. This defines a contravariant functor $(-)^{\ast} : \mod \C \to \Mod \C^{\op}$. It follows from the Yoneda lemma that $\C(-,X)^{\ast} \simeq \C(X,-)$ for each $X \in \C$, and also that, for each $f \in \C(X,Y)$, there is a commutative diagram \[ \begin{tikzcd}
{\C(-,Y)^{\ast}} \arrow[r, "{\C(-,f)^{\ast}}"] \arrow[d, "\simeq"'] &[1.7em] {\C(-,X)^{\ast}} \arrow[d, "\simeq"] \\
{\C(Y,-)} \arrow[r, "{\C(f,-)}"']                                   & {\C(X,-)}                           
\end{tikzcd} \] in $\Mod \C^{\op}$ whose vertical arrows are the isomorphisms given by the Yoneda lemma. Thus, $(-)^{\ast}$ sends $\C(-,X)$ to $\C(X,-)$, and $\C(-,f)$ to $\C(f,-)$, up to isomorphisms. Moreover, observe that if \[ \begin{tikzcd}
H \arrow[r, "\beta"] & G \arrow[r, "\alpha"] & F \arrow[r] & 0
\end{tikzcd} \] is a sequence in $\mod \C$ which is exact in $\Mod \C$, then \[ \begin{tikzcd}
0 \arrow[r] & F^{\ast} \arrow[r, "\alpha^{\ast}"] & G^{\ast} \arrow[r, "\beta^{\ast}"] & H^{\ast}
\end{tikzcd} \] is exact in $\Mod \C^{\op}$. Therefore, if $F \in \mod \C$ and if \[ \begin{tikzcd}
{\C(-,X)} \arrow[r, "{\C(-,f)}"] &[1.3em] {\C(-,Y)} \arrow[r] & F \arrow[r] & 0
\end{tikzcd} \] is an exact sequence in $\Mod \C$ with $f \in \C(X,Y)$, then, from the previous observations, we conclude that there is an exact sequence \[ \begin{tikzcd}
0 \arrow[r] & F^{\ast} \arrow[r] & {\C(Y,-)} \arrow[r, "{\C(f,-)}"] &[1.3em] {\C(X,-)}
\end{tikzcd} \] in $\Mod \C^{\op}$. Consequently, if $\C$ is left coherent, then $F^{\ast} \in \mod \C^{\op}$. Thus, when $\C$ is left coherent, we can consider $(-)^{\ast}$ as a functor from $\mod \C$ to $\mod \C^{\op}$.

By taking $\C^{\op}$ in place of $\C$ in the above discussion, we also obtain a contravariant functor $(-)^{\ast} : \mod \C^{\op} \to \Mod \C$. In this case, it follows from the Yoneda lemma that, for each $X \in \C$ and each morphism $f$ in $\C$, the functor $(-)^{\ast}$ sends $\C(X,-)$ to $\C(-,X)$, and $\C(f,-)$ to $\C(-,f)$, up to isomorphisms. Moreover, when $\C$ is right coherent, we can consider $(-)^{\ast}$ as a functor from $\mod \C^{\op}$ to $\mod \C$.

To conclude, the tool that we have mentioned before is the contravariant functor $(-)^{\ast}$. It gives dualities $(-)^{\ast} : \proj \C \leftrightarrow \proj \C^{\op}$ and, when $\C$ is coherent, it gives contravariant additive functors $(-)^{\ast} : \mod \C \leftrightarrow \mod \C^{\op}$.

Next, we move towards the axioms (A2) and (A2$^{\op}$). In order to understand these axioms in terms of the categories $\mod \C$ and $\mod \C^{\op}$, first we need to define a ``transpose'' of a module.

Following \cite[page 337]{MR342505} and \cite[Definition 2.5]{MR0269685}, given $F \in \mod \C$, if $f$ is a morphism in $\C$ for which $\mr(f) \simeq F$, then $\ml(f)$ is called a \textit{transpose} of $F$. Note that transposes are not unique. Indeed, if $f$ is a morphism in $\C$ with $\mr(f) \simeq F$ and if $s$ is a split epimorphism in $\C$, then $\mr(f \oplus s) \simeq F$, but $\ml(f \oplus s) \simeq \ml(f) \oplus \ml(s)$ and $\ml(s) \neq 0$ if $s$ is not an isomorphism. However, if $f$ and $g$ are morphisms in $\C$ such that $\mr(f) \simeq F$ and $\mr(g) \simeq F$, then we know from previous discussions that there are $s,t \in \SplitEpi$ with $f \oplus s \simeq g \oplus t$ in $\Mor \C$. Hence $\ml(f) \oplus \ml(s) \simeq \ml(g) \oplus \ml(t)$, and we can check that $\ml(r) \in \proj \C^{\op}$ for all $r \in \SplitEpi$. Consequently, by Theorem \ref{theorem.1}, there is an isomorphism $\ml(f) \oplus \proj \C^{\op} \simeq \ml(g) \oplus \proj \C^{\op}$ in the projectively stable category $\undermod \C^{\op}$. Thus, a transpose of $F$ is unique up to isomorphism in $\undermod \C^{\op}$, see also \cite[pages 50 and 51]{MR0269685}.

For $F \in \mod \C$, we use the notation $\Tr F$ to denote a transpose of $F$. By the above paragraph, $\Tr F \oplus \proj \C^{\op}$ is unique up to isomorphism. We also remark that, from the previous discussion concerning the functor $(-)^{\ast}$, for a morphism $f \in \C(X,Y)$, there is an exact sequence \[ \begin{tikzcd}
0 \arrow[r] & \mr(f)^{\ast} \arrow[r] & {\C(Y,-)} \arrow[r, "{\C(f,-)}"] &[1.3em] {\C(X,-)} \arrow[r] & \ml(f) \arrow[r] & 0
\end{tikzcd} \] in $\Mod \C^{\op}$. If $\mr(f) \simeq F$, then we might use the notation $\Tr F$ for $\ml(f)$, and the above exact sequence becomes \[ \begin{tikzcd}
0 \arrow[r] & F^{\ast} \arrow[r] & {\C(Y,-)} \arrow[r, "{\C(f,-)}"] &[1.3em] {\C(X,-)} \arrow[r] & \Tr F \arrow[r] & 0
\end{tikzcd}. \]

By interchanging the roles of $\C$ and $\C^{\op}$, we see that if $F \in \mod \C^{\op}$, then a transpose of $F$ is given by $\mr(f)$, where $f$ is a morphism in $\C$ for which $\ml(f) \simeq F$. In this case, $\Tr F \oplus \proj \C$ is unique up to isomorphism. Furthermore, for a morphism $f \in \C(X,Y)$, there is an exact sequence \[ \begin{tikzcd}
0 \arrow[r] & \ml(f)^{\ast} \arrow[r] & {\C(-,X)} \arrow[r, "{\C(-,f)}"] &[1.3em] {\C(-,Y)} \arrow[r] & \mr(f) \arrow[r] & 0
\end{tikzcd} \] in $\Mod \C$. If $\ml(f) \simeq F$, then we might write the above exact sequence as \[ \begin{tikzcd}
0 \arrow[r] & F^{\ast} \arrow[r] & {\C(-,X)} \arrow[r, "{\C(-,f)}"] &[1.3em] {\C(-,Y)} \arrow[r] & \Tr F \arrow[r] & 0
\end{tikzcd}. \]

We remark that the transpose defines contravariant functors $\Tr : \undermod \C \to \undermod \C^{\op}$ and $\Tr : \undermod \C^{\op} \to \undermod \C$ in the obvious ways. Moreover, it is not difficult to check that $\Tr^{2} = \Tr \circ \Tr \simeq 1$, hence $\Tr$ is a duality between the projectively stable
categories $\undermod \C$ and $\undermod \C^{\op}$. We will not use these facts in this paper, though. The interested reader may find more details in \cite[pages 337 and 338]{MR342505}, \cite[Proposition 2.6]{MR0269685} and \cite[page 13]{MR480688}.

Now, if $F \in \mod \C$ and if $\C$ is left coherent, then, for each positive integer $i$, define $\E^{i}(F)$ to be the $\C$-module given by $\E^{i}(F)(X) = \Ext^{i}(\Tr F, \C(X,-))$ for $X \in \C$ and $\E^{i}(F)(f) = \Ext^{i}(\Tr F, \C(f,-))$ for a morphism $f$ in $\C$. Here, $\Ext^{i}$ is taken in the abelian category $\mod \C^{\op}$. Observe that, even though $\Tr F$ is not uniquely determined by $F$, we know that transposes of $F$ differ only up to direct sums with projectives, hence $\E^{i}(F)$ is well defined (up to isomorphism). As in \cite[Definition 2.15]{MR0269685}, for a positive integer $k$, we say that $F$ is \textit{$k$-torsion free} if $\E^{i}(F) = 0$ for all $1 \leqslant i \leqslant k$. We remark that if $F \in \proj \C$, then $F$ is $k$-torsion free for every positive integer $k$.

By taking $\C^{\op}$ in place of $\C$, note that if $F \in \mod \C^{\op}$ and if $\C$ is right coherent, then $\E^{i}(F)$ is the $\C^{\op}$-module given by $\E^{i}(F)(X) = \Ext^{i}(\Tr F, \C(-,X))$ for $X \in \C$ and $\E^{i}(F)(f) = \Ext^{i}(\Tr F, \C(-,f))$ for a morphism $f$ in $\C$. In this case, $F$ is $k$-torsion free when $\E^{i}(F) = 0$ for all $1 \leqslant i \leqslant k$, and if $F \in \proj \C^{\op}$, then $F$ is $k$-torsion free for every positive integer $k$.

In the next result, Proposition \ref{proposition.14}, we describe how we can determine when a finitely presented $\C$-module $F$ is $k$-torsion free by making use of a morphism $f$ in $\C$ for which $\mr(f) \simeq F$. Of course, a similar result also holds for the case of finitely presented $\C^{\op}$-modules. The reader might want to compare Proposition \ref{proposition.14} with \cite[Theorem 2.17]{MR0269685} and \cite[Proposition 1.1.1]{MR2298819}.

\begin{proposition}\label{proposition.14}
Assume that $\C$ is coherent, let $f \in \C(X,Y)$ be a morphism in $\C$, and let $k$ be a positive integer. The following are equivalent:
\begin{enumerate}
        \item[(a)] $\mr(f)$ is $k$-torsion free.
        \item[(b)] For every exact sequence \[ \begin{tikzcd}
{\C(Y_{k},-)} \arrow[r, "{\C(g_{k},-)}"] &[1.45em] \cdots \arrow[r, "{\C(g_{2},-)}"] &[1.4em] {\C(Y_{1},-)} \arrow[r, "{\C(g_{1},-)}"] &[1.4em] {\C(Y,-)} \arrow[r, "{\C(f,-)}"] &[1.2em] {\C(X,-)}
\end{tikzcd} \] in $\mod \C^{\op}$, the sequence \[ \begin{tikzcd}
{\C(-,X)} \arrow[r, "{\C(-,f)}"] &[1.2em] {\C(-,Y)} \arrow[r, "{\C(-,g_{1})}"] &[1.4em] {\C(-,Y_{1})} \arrow[r, "{\C(-,g_{2})}"] &[1.4em] \cdots \arrow[r, "{\C(-,g_{k})}"] &[1.45em] {\C(-,Y_{k})}
\end{tikzcd} \] is exact in $\mod \C$.
        \item[(c)] There is an exact sequence \[ \begin{tikzcd}
{\C(Y_{k},-)} \arrow[r, "{\C(g_{k},-)}"] &[1.45em] \cdots \arrow[r, "{\C(g_{2},-)}"] &[1.4em] {\C(Y_{1},-)} \arrow[r, "{\C(g_{1},-)}"] &[1.4em] {\C(Y,-)} \arrow[r, "{\C(f,-)}"] &[1.2em] {\C(X,-)}
\end{tikzcd} \] in $\mod \C^{\op}$ for which \[ \begin{tikzcd}
{\C(-,X)} \arrow[r, "{\C(-,f)}"] &[1.2em] {\C(-,Y)} \arrow[r, "{\C(-,g_{1})}"] &[1.4em] {\C(-,Y_{1})} \arrow[r, "{\C(-,g_{2})}"] &[1.4em] \cdots \arrow[r, "{\C(-,g_{k})}"] &[1.45em] {\C(-,Y_{k})}
\end{tikzcd} \] is exact in $\mod \C$.
    \end{enumerate}
Furthermore, if $\mr(f)$ is $k$-torsion free and $\pdim \ml(f) \leqslant k + 1$, then the morphism $g_{k}$ in item (c) can be chosen so that $\C(g_{k},-)$ is a monomorphism.\footnote{That is, $g_{k}$ can be chosen to be an epimorphism in $\C$.}
\end{proposition}

\begin{proof}
Let \[ \begin{tikzcd}
{\C(Y_{k},-)} \arrow[r, "{\C(g_{k},-)}"] &[1.45em] \cdots \arrow[r, "{\C(g_{2},-)}"] &[1.4em] {\C(Y_{1},-)} \arrow[r, "{\C(g_{1},-)}"] &[1.4em] {\C(Y,-)} \arrow[r, "{\C(f,-)}"] &[1.2em] {\C(X,-)}
\end{tikzcd} \] be an exact sequence in $\mod \C^{\op}$. Such a sequence can be considered as the beginning of a projective resolution of $\ml(f)$ in $\mod \C^{\op}$. If $Z \in \C$, then, by applying $\Hom(-,\C(Z,-))$ to the above sequence\footnote{Observe that this is the same as applying the functor $(-)^{\ast}$ and then evaluating at $Z$.}, we get a complex in $\Ab$ which is, by the Yoneda lemma, isomorphic to \[ \begin{tikzcd}
{\C(Z,X)} \arrow[r, "{\C(Z,f)}"] &[1.2em] {\C(Z,Y)} \arrow[r, "{\C(Z,g_{1})}"] &[1.4em] {\C(Z,Y_{1})} \arrow[r, "{\C(Z,g_{2})}"] &[1.4em] \cdots \arrow[r, "{\C(Z,g_{k})}"] &[1.45em] {\C(Z,Y_{k})}
\end{tikzcd}. \] This complex is an exact sequence in $\Ab$ if and only if $\Ext^{i}(\ml(f),\C(Z,-)) = 0$ for all $1 \leqslant i \leqslant k$. Therefore, the sequence \[ \begin{tikzcd}
{\C(-,X)} \arrow[r, "{\C(-,f)}"] &[1.2em] {\C(-,Y)} \arrow[r, "{\C(-,g_{1})}"] &[1.4em] {\C(-,Y_{1})} \arrow[r, "{\C(-,g_{2})}"] &[1.4em] \cdots \arrow[r, "{\C(-,g_{k})}"] &[1.45em] {\C(-,Y_{k})}
\end{tikzcd} \] is exact in $\mod \C$ if and only if $\mr(f)$ is $k$-torsion free.

From the above paragraph, we deduce that (a) implies (b), and that (c) implies (a). Next, we show that (b) implies (c).

Suppose that item (b) holds. Since $\mod \C^{\op}$ is an abelian category with enough projectives, there is an exact sequence \[ \begin{tikzcd}
{\C(Y_{k},-)} \arrow[r, "{\C(g_{k},-)}"] &[1.45em] \cdots \arrow[r, "{\C(g_{2},-)}"] &[1.4em] {\C(Y_{1},-)} \arrow[r, "{\C(g_{1},-)}"] &[1.4em] {\C(Y,-)} \arrow[r, "{\C(f,-)}"] &[1.2em] {\C(X,-)}
\end{tikzcd} \] in $\mod \C^{\op}$. In this case, it follows from item (b) that \[ \begin{tikzcd}
{\C(-,X)} \arrow[r, "{\C(-,f)}"] &[1.2em] {\C(-,Y)} \arrow[r, "{\C(-,g_{1})}"] &[1.4em] {\C(-,Y_{1})} \arrow[r, "{\C(-,g_{2})}"] &[1.4em] \cdots \arrow[r, "{\C(-,g_{k})}"] &[1.45em] {\C(-,Y_{k})}
\end{tikzcd} \] is exact in $\mod \C$.

Finally, observe that if $\pdim \ml(f) \leqslant k + 1$, then the above sequence in $\mod \C^{\op}$ can be chosen so that $\C(g_{k},-)$ is a monomorphism. In fact, it suffices to take a projective resolution of $\ml(f)$ in $\mod \C^{\op}$ with the smallest length as possible.
\end{proof}

We are ready to describe the axioms (A2) and (A2$^{\op}$) of an $n$-abelian category in terms of $\mod \C$ and $\mod \C^{\op}$, respectively. For a coherent category $\C$, consider the following axioms:

\begin{enumerate}
    \item[(F2)] Every $F \in \mod \C$ with $\pdim F \leqslant 1$ is $n$-torsion free.
    \item[(F2$^{\op}$)] Every $F \in \mod \C^{\op}$ with $\pdim F \leqslant 1$ is $n$-torsion free.
\end{enumerate}

\begin{proposition}\label{proposition.3} If $\C$ is a pre-$n$-abelian category, then the axioms \textup{(A2)} and \textup{(A2$^{\op}$)} are equivalent to \textup{(F2)} and \textup{(F2$^{\op}$)}, respectively.
\end{proposition}

\begin{proof}
Assume that $\C$ is a pre-$n$-abelian category, so that, by Proposition \ref{proposition.2}, both $\mod \C$ and $\mod \C^{\op}$ are abelian categories of global dimension at most $n+1$. Below, we only prove that (A2) and (F2) are equivalent. By duality, we obtain that (A2)$^{\op}$ and (F2)$^{\op}$ are also equivalent.

Suppose that $\C$ satisfies the axiom (A2). Let $F \in \mod \C$ be such that $\pdim F \leqslant 1$. Then there is a projective resolution \[ \begin{tikzcd}
0 \arrow[r] & {\C(-,X)} \arrow[r, "{\C(-,f)}"] &[1.2em] {\C(-,Y)} \arrow[r] & F \arrow[r] & 0
\end{tikzcd} \] of $F$ in $\mod \C$ with $f \in \C(X,Y)$, so that $\mr(f) \simeq F$. Note that $f$ is a monomorphism in $\C$. Since $\gldim (\mod \C^{\op}) \leqslant n + 1$, we can extend the projective presentation \[ \begin{tikzcd}
{\C(Y,-)} \arrow[r, "{\C(f,-)}"] &[1.2em] {\C(X,-)} \arrow[r] & \ml(f) \arrow[r] & 0
\end{tikzcd} \] of $\ml(f)$ in $\mod \C^{\op}$ to a projective resolution {\relsize{-1} \[ \begin{tikzcd}
0 \arrow[r] &[-0.18em] {\C(Y_{n},-)} \arrow[r, "{\C(g_{n},-)}"] &[1.7em] \cdots \arrow[r, "{\C(g_{2},-)}"] &[1.6em] {\C(Y_{1},-)} \arrow[r, "{\C(g_{1},-)}"] &[1.6em] {\C(Y,-)} \arrow[r, "{\C(f,-)}"] &[1.2em] {\C(X,-)} \arrow[r] &[-0.18em] {\ml(f)} \arrow[r] &[-0.18em] 0
\end{tikzcd} \]} \hspace{-0.667em} of $\ml(f)$ in $\mod \C^{\op}$. In this case, \[ \begin{tikzcd}
Y \arrow[r, "g_{1}"] & Y_{1} \arrow[r, "g_{2}"] & \cdots \arrow[r, "g_{n}"] &[0.05em] Y_{n}
\end{tikzcd} \] is an $n$-cokernel of $f$. From the axiom (A2), we conclude that \[ \begin{tikzcd}
X \arrow[r, "f"] & Y \arrow[r, "g_{1}"] & \cdots \arrow[r, "g_{n-1}"] &[0.7em] Y_{n-1}
\end{tikzcd} \] is an $n$-kernel of $g_{n}$. Therefore, the sequence \[ \begin{tikzcd}
0 \arrow[r] & {\C(-,X)} \arrow[r, "{\C(-,f)}"] &[1.2em] {\C(-,Y)} \arrow[r, "{\C(-,g_{1})}"] &[1.4em] {\C(-,Y_{1})} \arrow[r, "{\C(-,g_{2})}"] &[1.4em] \cdots \arrow[r, "{\C(-,g_{n})}"] &[1.45em] {\C(-,Y_{n})}
\end{tikzcd} \] is exact in $\mod \C$. Thus, it follows from Proposition \ref{proposition.14} that $\mr(f)$ is $n$-torsion free, hence so is $F$.

Conversely, suppose that $\C$ satisfies the axiom (F2). If $f \in \C(X,Y)$ is a monomorphism in $\C$, then the sequence \[ \begin{tikzcd}
0 \arrow[r] & {\C(-,X)} \arrow[r, "{\C(-,f)}"] &[1.2em] {\C(-,Y)} \arrow[r] & {\mr(f)} \arrow[r] & 0
\end{tikzcd} \] is exact in $\mod \C$. Hence $\pdim \mr(f) \leqslant 1$, which implies that $\mr(f)$ is $n$-torsion free. If \[ \begin{tikzcd}
Y \arrow[r, "g_{1}"] & Y_{1} \arrow[r, "g_{2}"] & \cdots \arrow[r, "g_{n}"] &[0.05em] Y_{n}
\end{tikzcd} \] is an $n$-cokernel of $f$, then \[ \begin{tikzcd}
0 \arrow[r] & {\C(Y_{n},-)} \arrow[r, "{\C(g_{n},-)}"] &[1.45em] \cdots \arrow[r, "{\C(g_{2},-)}"] &[1.4em] {\C(Y_{1},-)} \arrow[r, "{\C(g_{1},-)}"] &[1.4em] {\C(Y,-)} \arrow[r, "{\C(f,-)}"] &[1.2em] {\C(X,-)}
\end{tikzcd} \] is an exact sequence in $\mod \C^{\op}$. By Proposition \ref{proposition.14}, the sequence \[ \begin{tikzcd}
0 \arrow[r] & {\C(-,X)} \arrow[r, "{\C(-,f)}"] &[1.2em] {\C(-,Y)} \arrow[r, "{\C(-,g_{1})}"] &[1.4em] {\C(-,Y_{1})} \arrow[r, "{\C(-,g_{2})}"] &[1.4em] \cdots \arrow[r, "{\C(-,g_{n})}"] &[1.45em] {\C(-,Y_{n})}
\end{tikzcd} \] is exact in $\mod \C$. Consequently, \[ \begin{tikzcd}
X \arrow[r, "f"] & Y \arrow[r, "g_{1}"] & \cdots \arrow[r, "g_{n-1}"] &[0.7em] Y_{n-1}
\end{tikzcd} \] is an $n$-kernel of $g_{n}$.
\end{proof}

We can now state the main result of this paper.

\begin{theorem}\label{theorem.2}
An additive and idempotent complete category $\C$ is $n$-abelian if and only if $\C$ satisfies the following axioms: \begin{enumerate}
    \item[(F1)] $\C$ is right coherent and $\gldim (\mod \C) \leqslant n+1$.
    \item[(F1$^{\op}$)] $\C$ is left coherent and $\gldim (\mod \C^{\op}) \leqslant n+1$.
    \item[(F2)] Every $F \in \mod \C$ with $\pdim F \leqslant 1$ is $n$-torsion free.
    \item[(F2$^{\op}$)] Every $F \in \mod \C^{\op}$ with $\pdim F \leqslant 1$ is $n$-torsion free.
\end{enumerate}
\end{theorem}

\begin{proof}
Follows from Propositions \ref{proposition.2} and \ref{proposition.3}.
\end{proof}

From now on, we focus on the axioms presented in Theorem \ref{theorem.2} to study $n$-abelian categories. We call them the \textit{functorial axioms} of an $n$-abelian category. Two remarks about these axioms are worth mentioning. First, note that we could replace the condition $\pdim F \leqslant 1$ by $\pdim F = 1$ in the axioms (F2) and (F2$^{\op}$) since projective modules are always $n$-torsion free. Second, although the axioms (F1) and (F1$^{\op}$) state that the global dimensions of $\mod \C$ and $\mod \C^{\op}$ are at most $n+1$, we prove in Corollary \ref{corollary.1} that, except for a trivial case, if $\C$ is an $n$-abelian category, then these global dimensions are actually equal to $n+1$. This trivial case is when the category is ``von Neumann regular'', and we devote the next section to analyze it.

\section{The von Neumann regular case}\label{section.4}

When investigating whether a category is $n$-abelian or not, it can happen that it is $n$-abelian for every positive integer $n$. This is the case precisely when the category is ``von Neumann regular''. In this section, we prove this statement, and we give several characterizations of these categories.

An additive and idempotent complete category $\C$ is called \textit{von Neumann regular} if every morphism $f$ in $\C$ can be written as $f = jp$ in $\C$, where $p$ is a split epimorphism and $j$ is a split monomorphism. Although such categories are usually called ``semisimple'', as in \cite{MR3519980}, for example, we believe that our choice of nomenclature is more appropriate since these categories can be thought of as generalizations of von Neumann regular rings. Indeed, $\C$ is von Neumann regular if and only if every morphism $f$ in $\C$ can be written as $f = fgf$ for some morphism $g$ in $\C$, see \cite[Proposition 3.4]{MR2116320}. Another fact that supports our choice of nomenclature is that $\C$ is von Neumann regular if and only if every finitely presented $\C$-module is projective, which we prove in Proposition \ref{proposition.4}. This is analogous to the fact that a ring $\Lambda$ is von Neumann regular if and only if every finitely presented $\Lambda$-module is projective, see \cite[Section 4]{MR1653294}. Also, note that, due to \cite[Proposition 1.1]{MR306265}, the equivalence between $\C$ being von Neumann regular and the conditions (b) and (d) of Proposition \ref{proposition.4} is analogous to the fact that a ring is von Neumann regular if and only if it has weak dimension zero, see \cite[Theorem 4.21]{MR1653294}.

\begin{proposition}\label{proposition.4}
Let $\C$ be an additive and idempotent complete category. The following are equivalent:
\begin{enumerate}
    \item[(a)] $\C$ is von Neumann regular.
    \item[(b)] $\C$ is right coherent and $\gldim (\mod \C) = 0$.
    \item[(c)] $\mod \C = \proj \C$.
    \item[(d)] $\C$ is left coherent and $\gldim (\mod \C^{\op}) = 0$.
    \item[(e)] $\mod \C^{\op} = \proj \C^{\op}$. 
\end{enumerate}
\end{proposition}

\begin{proof}
We only prove the equivalences between (a), (b) and (c). Then, by duality, we can deduce that (a), (d) and (e) are equivalent, since $\C$ is von Neumann regular if and only if $\C^{\op}$ is von Neumann regular.\footnote{It is also possible to prove directly that item (b) implies item (d), and vice versa, by making use of the functor $(-)^{\ast}$.}

Assume that $\C$ is von Neumann regular. Then we can verify that $\C$ has kernels, so that it follows from Proposition \ref{proposition.2} that $\C$ is right coherent and $\gldim (\mod \C) \leqslant 2$. Furthermore, it is easy to see that every monomorphism in $\C$ is a split monomorphism. Consequently, every monomorphism in $\mod \C$ whose domain and codomain are in $\proj \C$ is a split monomorphism. Therefore, as the projective objects of $\mod \C$ are given by $\proj \C$, and $\mod \C$ is an abelian category with enough projectives and finite global dimension, it is straightforward to conclude that $\gldim (\mod \C) = 0$. Hence (a) implies (b). Also, it is clear that (b) implies (c), and we prove below that (c) implies (a).

Suppose that $\mod \C = \proj \C$. Let $f \in \C(X,Y)$ be an arbitrary morphism in $\C$, and consider the exact sequence \[ \begin{tikzcd}
{\C(-,X)} \arrow[r, "{\C(-,f)}"] &[1.3em] {\C(-,Y)} \arrow[r] & \mr(f) \arrow[r] & 0
\end{tikzcd} \] in $\Mod \C$. Since $\mr(f)$ is projective, $\C(-,Y) \to \mr(f)$ is a split epimorphism, which implies that its kernel is a split monomorphism. Consequently, the image (object) of $\C(-,f)$ is isomorphic to $\C(-,Z)$ for some $Z \in \C$, and we conclude that there are morphisms $p \in \C(X,Z)$ and $j \in \C(Z,Y)$ such that $\C(-,f) = \C(-,j) \C(-,p)$, where $\C(-,p)$ is an epimorphism in $\Mod \C$ and $\C(-,j)$ is a split monomorphism. However, $\C(-,p)$ being an epimorphism to a projective implies that it is a split epimorphism. Thus, by the Yoneda lemma, we get that $f = jp$, where $p$ is a split epimorphism and $j$ is a split monomorphism.
\end{proof}

While Proposition \ref{proposition.4} characterizes von Neumann regular categories in terms of their categories of finitely presented modules, we mention next a few other characterizations of these categories, which are given in terms of their intrinsic properties.

\begin{proposition}\label{proposition.7}
Let $\C$ be an additive and idempotent complete category. The following are equivalent:
\begin{enumerate}
    \item[(a)] $\C$ is von Neumann regular.
    \item[(b)] $\C$ has kernels and every monomorphism in $\C$ is a split monomorphism.
    \item[(c)] $\C$ has cokernels and every epimorphism in $\C$ is a split epimorphism.
    \item[(d)] $\C$ is an abelian category and every object in $\C$ is projective and injective.
\end{enumerate}
\end{proposition}

\begin{proof}
We leave it to the reader to verify that (a) implies both (b) and (c). From the proof of Proposition \ref{proposition.4}, we see that (b) implies (a). By duality, (c) implies (a). Therefore, (a), (b) and (c) are equivalent. Next, observe that (d) implies both (b) and (c). Furthermore, we can conclude from Proposition \ref{proposition.4} that (a) implies (d), as we show below.

Assume that $\C$ is von Neumann regular. Then, by Proposition \ref{proposition.4}, $\C$ is right coherent and $\mod \C = \proj \C$. But recall that the Yoneda embedding induces an equivalence of categories $\C \approx \proj \C$. Consequently, $\C$ is an abelian category with the property that all of its objects are projective, which implies that its objects are also injective.
\end{proof}

It follows from Proposition \ref{proposition.4} and  Theorem \ref{theorem.2} that if a category is von Neumann regular, then it is $n$-abelian for every positive integer $n$. Our goal now is to show that von Neumann regular categories are the only categories with this property. In fact, we prove in Proposition \ref{proposition.6} that if a category is $n$-abelian for more than one positive integer $n$, then it must be von Neumann regular.

In order to prove Proposition \ref{proposition.6}, we need to state a few results first. We start with the following lemma, which is well known. For the convenience of the reader, we include a proof.

\begin{lemma}\label{lemma.1}
Let $\A$ be an abelian category with enough projectives. Given $X \in \A$, let $d = \pdim X$. If $1 \leqslant d < \infty$, then there is a projective object $P \in \A$ with $\Ext_{\A}^{d}(X,P) \neq 0$.
\end{lemma}

\begin{proof}
Assume that $1 \leqslant d < \infty$ and let \[ \begin{tikzcd}
0 \arrow[r] & P_{d} \arrow[r] & P_{d-1} \arrow[r] & \cdots \arrow[r] & P_{1} \arrow[r] & P_{0} \arrow[r] & X \arrow[r] & 0
\end{tikzcd} \] be a projective resolution of $X$ in $\A$. We claim that $\Ext_{\A}^{d}(X,P_{d}) \neq 0$. In fact, if this was not the case, then, when computing $\Ext_{\A}^{d}(X,P_{d})$ by applying $\A(-,P_{d})$ in the above projective resolution and taking homology, we would get that $\A(P_{d-1},P_{d}) \to \A(P_{d},P_{d})$ is an epimorphism. But this would imply that $P_{d} \to P_{d-1}$ is a split monomorphism, so that its cokernel would be projective, which would give that $\pdim X \leqslant d - 1$, a contradiction.
\end{proof}

%Then $\Ext_{\A}^{d}(X,P_{d}) \neq 0$, otherwise $P_{d} \to P_{d-1}$ would be a split monomorphism, which would imply that $\pdim X \leqslant d - 1$.

The next lemma is also well known. It follows from Proposition \ref{proposition.23} and the fact that transposes define dualities $\Tr : \undermod \C \leftrightarrow \undermod \C^{\op}$, as remarked in Section \ref{section.3}. However, in an effort to make this paper as self-contained as possible, we provide a direct proof.

\begin{lemma}\label{lemma.2}
Let $\C$ be an additive and idempotent complete category. If $F \in \mod \C$, then $F \in \proj \C$ if and only if $\Tr F \in \proj \C^{\op}$.
\end{lemma}

\begin{proof}
Let $F \in \mod \C$. If $F \in \proj \C$, then it is clear that $\Tr F \in \proj \C^{\op}$. So, assume that $\Tr F \in \proj \C^{\op}$. Take a morphism $f \in \C(X,Y)$ in $\C$ such that $\mr(f) \simeq F$, and consider the exact sequence \[ \begin{tikzcd}
{\C(-,X)} \arrow[r, "{\C(-,f)}"] &[1.3em] {\C(-,Y)} \arrow[r] & F \arrow[r] & 0
\end{tikzcd} \] in $\Mod \C$. By applying $(-)^{\ast}$ to it, we get an exact sequence \[ \begin{tikzcd}
{\C(Y,-)} \arrow[r, "{\C(f,-)}"] &[1.3em] {\C(X,-)} \arrow[r] & \ml(f) \arrow[r] & 0
\end{tikzcd} \] in $\Mod \C^{\op}$. Since $\ml(f)$ is a transpose of $F$, it is projective. Thus, as in the proof of Proposition \ref{proposition.4}, we can deduce that there is some $Z \in \C$ and morphisms $y \in \C(Z,Y)$ and $x \in \C(X,Z)$ such that $\C(f,-) = \C(x,-) \C(y,-)$, where $\C(y,-)$ is a split epimorphism and $\C(x,-)$ is a split monomorphism. Therefore, by applying $(-)^{\ast}$ to this decomposition of $\C(f,-)$, we obtain that  $\C(-,f) = \C(-,y) \C(-,x)$, where $\C(-,x)$ is a split epimorphism and $\C(-,y)$ is a split monomorphism. Consequently, $\C(-,Y) \to F$ is a split epimorphism since its kernel $\C(-,y)$ is a split monomorphism, so that $F \in \proj \C$.
\end{proof}

\begin{proposition}\label{proposition.5}
Let $\C$ be a pre-$n$-abelian category that satisfies the axiom \textup{(F2)}. If $F \in \mod \C$ is such that $\pdim F = 1$, then $\pdim \Tr F = n + 1$.
\end{proposition}

\begin{proof}
Let $F \in \mod \C$ be such that $\pdim F = 1$. By Lemma \ref{lemma.2}, we have $\pdim \Tr F \geqslant 1$. Moreover, we know from Proposition \ref{proposition.2} that $\gldim (\mod \C^{\op}) \leqslant n + 1$, which implies that $\pdim \Tr F \leqslant n + 1$. Therefore, because $F$ is $n$-torsion free, it follows from Lemma \ref{lemma.1} that $\pdim \Tr F = n + 1$.
\end{proof}

\begin{corollary}\label{corollary.1}
Let $\C$ be an $n$-abelian category. If $\C$ is not von Neumann regular, then $\gldim (\mod \C) = \gldim (\mod \C^{\op}) = n + 1$.
\end{corollary}

\begin{proof}
By Theorem \ref{theorem.2}, $\C$ satisfies the axioms (F1), (F1$^{\op}$), (F2) and (F2$^{\op}$). In particular, the inequalities $\gldim (\mod \C) \leqslant n + 1$ and $\gldim (\mod \C^{\op}) \leqslant n + 1$ hold. Assume that $\C$ is not von Neumann regular. Then we obtain from Proposition \ref{proposition.4} that $\gldim (\mod \C) > 0$. Hence there is $F \in \mod \C$ with $\pdim F = 1$, so that $\pdim \Tr F = n + 1$, by Proposition \ref{proposition.5}. Thus, $\gldim (\mod \C^{\op}) = n + 1$. By duality, $\gldim (\mod \C) = n + 1$ also holds.\footnote{Of course, we could also use Theorem \ref{theorem.5} to conclude the proof.}
\end{proof}

It follows from Corollary \ref{corollary.1} that if a category is not von Neumann regular, then it cannot be $m$-abelian and $n$-abelian for two distinct positive integers $m$ and $n$. This argument allows us to characterize von Neumann regular categories as the categories that are $n$-abelian for more than one (or for every) positive integer $n$. Let us state this characterization, which was first proved by Jasso in \cite[Corollary 3.10]{MR3519980}.

\begin{proposition}\label{proposition.6}
Let $\C$ be an additive and idempotent complete category. The following are equivalent:
\begin{enumerate}
    \item[(a)] $\C$ is von Neumann regular.
    \item[(b)] $\C$ is $n$-abelian for every positive integer $n$.
    \item[(c)] There are two distinct positive integers $m$ and $n$ for which $\C$ is $m$-abelian and $n$-abelian.
\end{enumerate}
\end{proposition}

\begin{proof}
It follows from Proposition \ref{proposition.4} and Theorem \ref{theorem.2} that (a) implies (b). Trivially, (b) implies (c), and by Corollary \ref{corollary.1}, (c) implies (a).
\end{proof}

As we see from the characterizations of von Neumann regular categories given in this section, such categories can be considered to be trivial from the point of view of (classical and higher) homological algebra. In Appendix \ref{section.11}, we give more evidence for this claim, and we also present another proof of Proposition \ref{proposition.6}.

\section{The double dual sequence}\label{section.5}

In this section, we prove the existence of the ``double dual sequence'' of a finitely presented module, and we explore some of its properties and consequences, which will be important for the development of the rest of this paper.

Recall from Section \ref{section.3} that if $F$ is a finitely presented $\C$-module or $\C^{\op}$-module, then we can consider its dual $F^{\ast}$. As we have already observed, when $\C$ is coherent, the dual of a module defines contravariant additive functors $(-)^{\ast} : \mod \C \leftrightarrow \mod \C^{\op}$. If $\C$ is coherent, we define the \textit{double dual} of $F$ to be $(F^{\ast})^{\ast}$, which we denote by $F^{\ast \ast}$. Clearly, there is a canonical morphism $F \to F^{\ast \ast}$, see \cite[page 336]{MR342505}, and $F$ is called \textit{reflexive} when this morphism is an isomorphism.

The next result, which was proved by Auslander in \cite{MR0212070}, offers a way to understand when a finitely presented module $F$ is reflexive. Following \cite{MR3537819}, we call the exact sequence in the proposition below the \textit{double dual sequence} of $F$.\footnote{Some authors call it the \textit{Auslander--Bridger sequence} of $F$, due to its appearance in \cite{MR0269685}.}

\begin{proposition}\label{proposition.10}
Assume that $\C$ is coherent. For each $F \in \mod \C$, there is an exact sequence \[ \begin{tikzcd}
0 \arrow[r] & \E^{1}(F) \arrow[r] & F \arrow[r] & F^{\ast \ast} \arrow[r] & \E^{2}(F) \arrow[r] & 0
\end{tikzcd} \] in $\mod \C$, where $F \to F^{\ast \ast}$ is the canonical morphism. Moreover, the morphisms in this sequence are natural in $F$.
\end{proposition}

\begin{proof}
Let \[ \begin{tikzcd}
{\C(-,X)} \arrow[r, "{\C(-,f)}"] &[1.2em] {\C(-,Y)} \arrow[r] & F \arrow[r] & 0
\end{tikzcd} \] be a projective presentation of $F$ in $\mod \C$ with $f \in \C(X,Y)$. By applying $(-)^{\ast}$ to this projective presentation, we get an exact sequence \[ \begin{tikzcd}
0 \arrow[r] & F^{\ast} \arrow[r] & {\C(Y,-)} \arrow[r, "{\C(f,-)}"] &[1.2em] {\C(X,-)} \arrow[r] & \ml(f) \arrow[r] & 0
\end{tikzcd} \] in $\mod \C^{\op}$. Write this sequence as the splice of two short exact sequences \[ \begin{tikzcd}
0 \arrow[r] & F^{\ast} \arrow[r] & {\C(Y,-)} \arrow[r] & G \arrow[r] & 0
\end{tikzcd} \] and \[ \begin{tikzcd}
0 \arrow[r] & G \arrow[r] & {\C(X,-)} \arrow[r] & \ml(f) \arrow[r] & 0
\end{tikzcd} \] in $\mod \C^{\op}$. By applying $(-)^{\ast}$ to these short exact sequences, we obtain two exact sequences \[ \begin{tikzcd}
0 \arrow[r] & G^{\ast} \arrow[r] & {\C(-,Y)} \arrow[r] & F^{\ast \ast} \arrow[r] & \E^{2}(F) \arrow[r] & 0
\end{tikzcd} \] and \[ \begin{tikzcd}
0 \arrow[r] & \ml(f)^{\ast} \arrow[r] & {\C(-,X)} \arrow[r] & G^{\ast} \arrow[r] & \E^{1}(F) \arrow[r] & 0
\end{tikzcd} \] in $\mod \C$. We can see that the cokernels of $\C(-,Y) \to F^{\ast \ast}$ and $\C(-,X) \to G^{\ast}$ are given by $\E^{2}(F)$ and $\E^{1}(F)$, respectively, by applying $\Hom(-,\C(Z,-))$ in the previous short exact sequences in $\mod \C^{\op}$ and then considering the induced long exact sequences, for each $Z \in \C$, and by using the fact that $\Ext^{1}(G,-) \simeq \Ext^{2}(\ml(f),-)$.

Next, we apply $(-)^{\ast}$ to \[ \begin{tikzcd}
F^{\ast} \arrow[r] & {\C(Y,-)} \arrow[r, "{\C(f,-)}"] &[1.2em] {\C(X,-)}
\end{tikzcd} \] and we get morphisms \[ \begin{tikzcd}
{\C(-,X)} \arrow[r, "{\C(-,f)}"] &[1.2em] {\C(-,Y)} \arrow[r] & F^{\ast \ast}
\end{tikzcd} \] in $\mod \C$ whose composition is zero, so that there is a unique morphism $F \to F^{\ast \ast}$ in $\mod \C$ which makes the diagram

%vertical space

\[ \begin{tikzcd}
{\C(-,X)} \arrow[r, "{\C(-,f)}"] &[1.2em] {\C(-,Y)} \arrow[r] \arrow[rd] & F \arrow[r] \arrow[d] & 0 \\
                                 &                                & F^{\ast \ast}         &  
\end{tikzcd} \] commute. In this case, $F \to F^{\ast \ast}$ is the canonical morphism from a module to its double dual. The details are left to the reader.

Finally, we put the previous exact sequences in $\mod \C$ into a commutative diagram \[ \begin{tikzcd}
            &                                                        &                                                    &[1.2em] 0 \arrow[d]                                            & 0 \arrow[d]                  &   \\
0 \arrow[r] & \ml(f)^{\ast} \arrow[d, equal] \arrow[r] & {\C(-,X)} \arrow[d, equal] \arrow[r] & G^{\ast} \arrow[r] \arrow[d]                           & \E^{1}(F) \arrow[r] \arrow[d] & 0 \\
0 \arrow[r] & \ml(f)^{\ast} \arrow[r]                                & {\C(-,X)} \arrow[r, "{\C(-,f)}"']                   & {\C(-,Y)} \arrow[r] \arrow[d]                          & F \arrow[r] \arrow[d]        & 0 \\
            &                                                        &                                                    & F^{\ast \ast} \arrow[r, equal] \arrow[d] & F^{\ast \ast} \arrow[d]      &   \\
            &                                                        &                                                    & \E^{2}(F) \arrow[d] \arrow[r, equal]      & \E^{2}(F) \arrow[d]           &   \\
            &                                                        &                                                    & 0                                                      & 0                            &  
\end{tikzcd} \] where $F \to F^{\ast \ast}$ is the canonical morphism. By filling this diagram with zeros, we conclude from the $4 \times 4$ lemma that its fourth column is exact in $\mod \C$, \mbox{see \cite[Lemma 2.6]{MR2909639}.}

Lastly, given that the canonical morphism $F \to F^{\ast \ast}$ is natural in $F$, so are the other morphisms in the double dual sequence of $F$.
\end{proof}

Observe that, by Proposition \ref{proposition.10}, if $\C$ is coherent and $F \in \mod \C$, then $F$ is reflexive if and only if $F$ is $2$-torsion free. Furthermore, we also conclude from Proposition \ref{proposition.10} that if $F \to G$ is a morphism in $\mod \C$, then it induces a commutative diagram \[ \begin{tikzcd}
\E^{1}(F) \arrow[r] \arrow[d] & F \arrow[d] \\
\E^{1}(G) \arrow[r]           & G          
\end{tikzcd} \] in $\mod \C$, where the horizontal arrows are monomorphisms. In particular, note that if $F \to G$ is a monomorphism, then so is $\E^{1}(F) \to \E^{1}(G)$. Therefore, if $F \to G$ is a monomorphism and $G$ is $1$-torsion free, then so is $F$. Consequently, because projectives are $1$-torsion free, if there is a monomorphism in $\mod \C$ from $F$ to a projective, then $F$ is $1$-torsion free. By Proposition \ref{proposition.14}, the converse also holds. Let us record these facts below, which are, of course, well known. For ease of reference, it is convenient to recall that an object $X$ in an abelian category $\A$ is called a \textit{syzygy} if it embeds into a projective, that is, if there is a monomorphism $X \to P$ in $\A$ with $P$ projective.

\begin{corollary}\label{corollary.3}
Assume that $\C$ is coherent, and let $F \in \mod \C$.
\begin{enumerate}
    \item[(a)] $F$ is $1$-torsion free if and only if $F$ is a syzygy.
    \item[(b)] $F$ is $2$-torsion free if and only if $F$ is reflexive.
\end{enumerate}
\end{corollary}

\begin{proof}
Follows from the above discussion.
\end{proof}

Because of Corollary \ref{corollary.3}, we can reformulate the functorial axioms (F2) and (F2$^{\op}$) of an $n$-abelian category for the cases $n = 1$ and $n = 2$ in terms of finitely presented modules of projective dimension at most $1$ being syzygies and reflexive, respectively. Moreover, as we show in Theorems \ref{theorem.7} and \ref{theorem.9}, the conditions of being syzygies and reflexive can be used to describe the axioms (F2) and (F2$^{\op}$) for every positive integer $n \geqslant 2$. This will be discussed in Section \ref{section.6}.

Observe that it also follows from Corollary \ref{corollary.3} that if $\C$ is coherent, then every projective object in $\mod \C$ is reflexive since projectives are $2$-torsion free.\footnote{Alternatively, we could prove directly that projectives are reflexive by using the Yoneda lemma.} Thus, it is natural to ask if there are reflexive objects in $\mod \C$ that are not projective. For the case when $\C$ is an $n$-abelian category which is not von Neumann regular, the answer to this question depends on $n$. Actually, it gives an interesting distinction between $1$-abelian categories and $n$-abelian categories with $n \geqslant 2$, as we show below.

\begin{corollary}\label{corollary.5}
Let $\C$ be an $n$-abelian category that is not von Neumann regular. The following are equivalent:
\begin{enumerate}
    \item[(a)] $n \geqslant 2$.
    \item[(b)] Every $F \in \mod \C$ with $\pdim F = 1$ is reflexive.
    \item[(c)] There is some $F \in \mod \C$ with $\pdim F = 1$ that is reflexive.
    \item[(d)] There is some $F \in \mod \C$ with $\pdim F \geqslant 1$ that is reflexive.
\end{enumerate}
\end{corollary}

\begin{proof}
It follows from Theorem \ref{theorem.2} and Corollary \ref{corollary.3} that (a) implies (b). Moreover, (b) implies (c) since Corollary \ref{corollary.1} says that $\gldim (\mod \C) = n + 1$, which implies the existence of some $F \in \mod \C$ with $\pdim F = 1$. Trivially, (c) implies (d), and we show below that (d) implies (a).

Suppose that there is some $F \in \mod \C$ with $\pdim F \geqslant 1$ which is reflexive. Then, by Lemma \ref{lemma.2} and Theorem \ref{theorem.2}, we have $1 \leqslant \pdim \Tr F \leqslant n + 1$. Furthermore, we know from Corollary \ref{corollary.3} that $F$ is $2$-torsion free, hence it follows from Lemma \ref{lemma.1} that $3 \leqslant \pdim \Tr F$. Consequently, $n \geqslant 2$.
\end{proof}

Motivated by \cite[Proposition 2.8]{MR3638352}, we end this section by recalling some well known properties of the double dual sequence, which will be used to prove Proposition \ref{proposition.20}.

\begin{proposition}\label{proposition.19}
Assume that $\C$ is coherent, and let $F \in \mod \C$. The double dual sequence \[ \begin{tikzcd}
0 \arrow[r] & \E^{1}(F) \arrow[r] & F \arrow[r] & F^{\ast \ast} \arrow[r] & \E^{2}(F) \arrow[r] & 0
\end{tikzcd} \] of $F$ has the following properties: \begin{enumerate}
    \item[(a)] The image of $F \to F^{\ast \ast}$ is isomorphic to $\Omega \Tr \Omega \Tr F$ for suitable choices of transposes and syzygies.\footnote{More precisely, we prove that if $f$ is a morphism in $\C$ and $g$ is a weak cokernel of $f$, then the image of the canonical morphism $\mr(f) \to \mr(f)^{\ast \ast}$ is isomorphic to the image of $\C(-,g)$.}
    \item[(b)] The dual of $\E^{1}(F) \to F$ is the zero morphism.\footnote{That is, when we apply $(-)^{\ast}$ to $\E^{1}(F) \to F$, we obtain the zero morphism.}
\end{enumerate}
\end{proposition}

\begin{proof}
(a) This is similar to the proof of Proposition \ref{proposition.10}, so we will be more concise.

Let \[ \begin{tikzcd}
{\C(-,X)} \arrow[r, "{\C(-,f)}"] &[1.2em] {\C(-,Y)} \arrow[r] & F \arrow[r] & 0
\end{tikzcd} \] be a projective presentation of $F$ in $\mod \C$ with $f \in \C(X,Y)$, and consider the exact sequence \[ \begin{tikzcd}
0 \arrow[r] & F^{\ast} \arrow[r] & {\C(Y,-)} \arrow[r, "{\C(f,-)}"] &[1.2em] {\C(X,-)} \arrow[r] & \ml(f) \arrow[r] & 0
\end{tikzcd} \] in $\mod \C^{\op}$. Since $\mod \C^{\op}$ has enough projectives, we can extend the above sequence to an exact sequence \[ \begin{tikzcd}
{\C(Z,-)} \arrow[r, "{\C(g,-)}"] &[1.2em] {\C(Y,-)} \arrow[r, "{\C(f,-)}"] &[1.2em] {\C(X,-)} \arrow[r] & \ml(f) \arrow[r] & 0
\end{tikzcd} \] in $\mod \C^{\op}$ with $g \in \C(Y,Z)$. Now, write it as the splice of two exact sequences \[ \begin{tikzcd}
{\C(Z,-)} \arrow[r, "{\C(g,-)}"] &[1.2em] {\C(Y,-)} \arrow[r] & G \arrow[r] & 0
\end{tikzcd} \] and \[ \begin{tikzcd}
0 \arrow[r] & G \arrow[r] & {\C(X,-)} \arrow[r] & \ml(f) \arrow[r] & 0
\end{tikzcd} \] in $\mod \C^{\op}$. By applying $(-)^{\ast}$ to the last two sequences, we get exact sequences \[ \begin{tikzcd}
0 \arrow[r] & G^{\ast} \arrow[r] & {\C(-,Y)} \arrow[r, "{\C(-,g)}"] &[1.2em] {\C(-,Z)} \arrow[r] & \mr(g) \arrow[r] & 0
\end{tikzcd} \] and \[ \begin{tikzcd}
0 \arrow[r] & \ml(f)^{\ast} \arrow[r] & {\C(-,X)} \arrow[r] & G^{\ast} \arrow[r] & \E^{1}(F) \arrow[r] & 0
\end{tikzcd} \] in $\mod \C$. 

Next, observe that the composition \[ \begin{tikzcd}
{\C(-,X)} \arrow[r, "{\C(-,f)}"] &[1.2em] {\C(-,Y)} \arrow[r, "{\C(-,g)}"] &[1.2em] {\C(-,Z)}
\end{tikzcd} \] is zero, hence there is a unique morphism $F \to \C(-,Z)$ in $\mod \C$ making the diagram \[ \begin{tikzcd}
{\C(-,X)} \arrow[r, "{\C(-,f)}"] &[1.2em] {\C(-,Y)} \arrow[r] \arrow[rd, "{\C(-,g)}"'] &[-0.4em] F \arrow[r] \arrow[d] &[-0.8em] 0 \\
                                 &                                             & {\C(-,Z)}             &  
\end{tikzcd} \] commute.

Finally, we put the previous exact sequences in $\mod \C$ into a commutative diagram \[ \begin{tikzcd}
            &                                          &                                      &[1.2em] 0 \arrow[d]                                & 0 \arrow[d]                   &   \\
0 \arrow[r] & \ml(f)^{\ast} \arrow[d, equal] \arrow[r] & {\C(-,X)} \arrow[d, equal] \arrow[r] & G^{\ast} \arrow[r] \arrow[d]               & \E^{1}(F) \arrow[r] \arrow[d] & 0 \\
0 \arrow[r] & \ml(f)^{\ast} \arrow[r]                  & {\C(-,X)} \arrow[r, "{\C(-,f)}"']    & {\C(-,Y)} \arrow[r] \arrow[d, "{\C(-,g)}"'] & F \arrow[r] \arrow[d]         & 0 \\[0.5em]
            &                                          &                                      & {\C(-,Z)} \arrow[r, equal] \arrow[d]       & {\C(-,Z)} \arrow[d]           &   \\
            &                                          &                                      & \mr(g) \arrow[d] \arrow[r, equal]          & \mr(g) \arrow[d]              &   \\
            &                                          &                                      & 0                                          & 0                             &  
\end{tikzcd} \] and by filling it with zeros, we conclude from the $4 \times 4$ lemma that its fourth column is exact in $\mod \C$, see \cite[Lemma 2.6]{MR2909639}. Therefore, the cokernel of $\E^{1}(F) \to F$ is of the form $\Omega \Tr \Omega \Tr F$ since it coincides with a syzygy of $\mr(g)$, and $\mr(g)$ is a transpose of $G$, which is a syzygy of a transpose of $F$. Thus, because the cokernel of $\E^{1}(F) \to F$ is isomorphic to the image of the canonical morphism $F \to F^{\ast \ast}$, we are done.

(b) In what follows, we make use of the morphisms described above, in the proof of item (a). By definition, $F^{\ast} \to \E^{1}(F)^{\ast}$ is zero if and only if for every $W \in \C$ and every morphism $F \to \C(-,W)$ in $\mod \C$, the composition $\E^{1}(F) \to F \to \C(-,W)$ is zero.

Well, take $W \in \C$, and consider a morphism $F \to \C(-,W)$ in $\mod \C$. By the Yoneda lemma, the composition $\C(-,Y) \to F \to \C(-,W)$ is given by $\C(-,h)$ for some morphism $h \in \C(Y,W)$. In this case, we have $\C(-,h)\C(-,f) = 0$, so that $hf = 0$. Then, because $g$ is a weak cokernel of $f$, we conclude that $h$ factors through $g$. Consequently, $\C(-,h)$ factors through $\C(-,g)$. Since $\C(-,g)$ coincides with the composition $\C(-,Y) \to F \to \C(-,Z)$ and $\C(-,Y) \to F$ is an epimorphism, we deduce that $F \to \C(-,W)$ factors through $F \to \C(-,Z)$. Therefore, the composition $\E^{1}(F) \to F \to \C(-,W)$ is zero.
\end{proof}

\section{The second axioms}\label{section.6}

The axioms (A2) and (A2$^{\op}$) of an $n$-abelian category were introduced by Jasso in \cite[Definition 3.1]{MR3519980} as generalizations of the axioms ``every monomorphism is the kernel of its cokernel'' and ``every epimorphism is the cokernel of its kernel'' of an abelian category. But, when defining an abelian category, it is also common to find these axioms in the form ``every monomorphism is a kernel'' and ``every epimorphism is a cokernel''. Therefore, it is natural to ask whether we can also generalize these latter axioms to the case of $n$-abelian categories. In this section, we present two such possible generalizations, which are achieved through the functorial approach.

Let us give an overview of the results of this section.

When contemplating the functorial axioms (F2) and (F2$^{\op}$) in Theorem \ref{theorem.2}, it is natural to ask if, for an $n$-abelian category $\C$, every finitely presented $\C$-module or $\C^{\op}$-module $F$ with $\pdim F \leqslant m$ is $(n+1-m)$-torsion free, whenever $m$ is a positive integer such that $1 \leqslant m \leqslant n$. It turns out that this is indeed the case, as it was essentially proved by Iyama and Jasso in \cite[Proposition 3.7]{MR3638352}. We present this result in Proposition \ref{proposition.20}. In addition, we also prove directly a similar statement in Proposition \ref{proposition.13}, by replacing the condition ``$\pdim F \leqslant m$'' by ``$F$ is $m$-spherical''.

The two results mentioned above lead to a few axioms equivalent to (F2) and (F2$^{\op}$), which are still expressed in terms of categories of finitely presented functors, see Theorems \ref{theorem.7} and \ref{theorem.9}. Once these axioms are obtained, we ``translate'' them to the language of higher homological algebra, thereby obtaining equivalent ways of stating the axioms (A2) and (A2$^{\op}$), see Theorems \ref{theorem.3} and \ref{theorem.10}. Among these, there are two possible generalizations of the axioms ``every monomorphism is a kernel'' and ``every epimorphism is a cokernel'', and we use them to give two alternative (but equivalent) definitions of an $n$-abelian category in Theorems \ref{theorem.8} and \ref{theorem.11}.

\subsection{The spherical case}\label{subsection.spherical}

In this subsection, we start to investigate the question of whether finitely presented modules of projective dimension at most $m$ over an $n$-abelian category are $(n+1-m)$-torsion free by first considering the ``$m$-spherical'' case.

Let $\A$ be an abelian category with enough projectives, and let $m$ be a positive integer. Motivated by \cite{MR0269685}, we say that an object $X \in \A$ is \textit{$m$-spherical} if $\pdim X \leqslant m$ and $\Ext_{\A}^{i}(X,P) = 0$ for every $1 \leqslant i \leqslant m - 1$ and every projective object $P \in \A$. Note that projective objects are $m$-spherical for every positive integer $m$. On the other hand, if $X$ is $m$-spherical and not projective, then $\pdim X = m$, by Lemma \ref{lemma.1}. We also remark that $X$ is $1$-spherical if and only if $\pdim X \leqslant 1$.

In particular, if $\C$ is right coherent, then $F \in \mod \C$ is $m$-spherical when $\pdim F \leqslant m$ and $\Ext^{i}(F,\C(-,X)) = 0$ for all $1 \leqslant i \leqslant m - 1$ and all $X \in \C$. Dually, if $\C$ is left coherent, then $F \in \mod \C^{\op}$ is $m$-spherical when $\pdim F \leqslant m$ and $\Ext^{i}(F,\C(X,-)) = 0$ for all $1 \leqslant i \leqslant m - 1$ and all $X \in \C$.

\begin{proposition}\label{proposition.13}
Assume that $\C$ is coherent, and let $m$ and $k$ be positive integers, where $k \geqslant 2$. If every $m$-spherical object in $\mod \C$ is $k$-torsion free, then every $(m+1)$-spherical object in $\mod \C$ is $(k-1)$-torsion free.
\end{proposition}

\begin{proof}
Suppose that every $m$-spherical object in $\mod \C$ is $k$-torsion free. Let $F \in \mod \C$ be $(m+1)$-spherical, and consider the beginning of a projective resolution \[ \begin{tikzcd}
{\C(-,X_{2})} \arrow[r, "{\C(-,f_{2})}"] &[1.5em] {\C(-,X_{1})} \arrow[r, "{\C(-,f_{1})}"] &[1.5em] {\C(-,X_{0})} \arrow[r] & F \arrow[r] & 0
\end{tikzcd} \] of $F$ in $\mod \C$. Let $H = \mr(f_{2})$ be the cokernel (object) of $\C(-,f_{2})$. Observe that $H$ is $m$-spherical, hence it is $k$-torsion free, and it follows from Proposition \ref{proposition.14} that there is an exact sequence \[ \begin{tikzcd}
{\C(Y_{k},-)} \arrow[r, "{\C(g_{k},-)}"] &[1.65em] \cdots \arrow[r, "{\C(g_{2},-)}"] &[1.6em] {\C(Y_{1},-)} \arrow[r, "{\C(g_{1},-)}"] &[1.6em] {\C(X_{1},-)} \arrow[r, "{\C(f_{2},-)}"] &[1.5em] {\C(X_{2},-)}
\end{tikzcd} \] in $\mod \C^{\op}$ for which \[ \begin{tikzcd}
{\C(-,X_{2})} \arrow[r, "{\C(-,f_{2})}"] &[1.5em] {\C(-,X_{1})} \arrow[r, "{\C(-,g_{1})}"] &[1.6em] {\C(-,Y_{1})} \arrow[r, "{\C(-,g_{2})}"] &[1.6em] \cdots \arrow[r, "{\C(-,g_{k})}"] &[1.65em] {\C(-,Y_{k})}
\end{tikzcd} \] is exact in $\mod \C$. Let $G = \mr(g_{1})$ be the cokernel (object) of $\C(-,g_{1})$. Note that, by Proposition \ref{proposition.14}, $G$ is $(k-1)$-torsion free. Now, from the previous exact sequence in $\mod \C^{\op}$, we see that $g_{1}$ is a weak cokernel of $f_{2}$. Thus, because $f_{1}f_{2} = 0$, it follows that $f_{1}$ factors through $g_{1}$, so that $\C(-,f_{1})$ factors through $\C(-,g_{1})$. Hence we get a commutative diagram with exact rows \[ \begin{tikzcd}
{\C(-,X_{2})} \arrow[d, equal] \arrow[r, "{\C(-,f_{2})}"] &[1.5em] {\C(-,X_{1})} \arrow[d, equal] \arrow[r, "{\C(-,g_{1})}"] &[1.6em] {\C(-,Y_{1})} \arrow[r] \arrow[d] & G \arrow[r] \arrow[d] & 0 \\
{\C(-,X_{2})} \arrow[r, "{\C(-,f_{2})}"']                  & {\C(-,X_{1})} \arrow[r, "{\C(-,f_{1})}"']                   & {\C(-,X_{0})} \arrow[r]                   & F \arrow[r]                   & 0
\end{tikzcd} \] in $\mod \C$, which leads to a commutative diagram with exact rows \[ \begin{tikzcd}
0 \arrow[r] & H \arrow[d, equal] \arrow[r] & {\C(-,Y_{1})} \arrow[r] \arrow[d] & G \arrow[r] \arrow[d] & 0 \\
0 \arrow[r] & H \arrow[r]                    & {\C(-,X_{0})} \arrow[r]           & F \arrow[r]           & 0
\end{tikzcd} \] in $\mod \C$. Consequently, the rightmost square in the above diagram is a pullback diagram, so that there is a short exact sequence \[ \begin{tikzcd}
0 \arrow[r] & {\C(-,Y_{1})} \arrow[r] & {G \oplus \C(-,X_{0})} \arrow[r] & F \arrow[r] & 0
\end{tikzcd} \] in $\mod \C$. But since $F$ is $(m+1)$-spherical, we have $\Ext^{1}(F,\C(-,Y_{1})) = 0$, and the above short exact sequence splits. Thus, $G \oplus \C(-,X_{0}) \simeq F \oplus \C(-,Y_{1})$, which implies that $\E^{i}(G) \simeq \E^{i}(F)$ for all $i \geqslant 1$. Therefore, as $G$ is $(k-1)$-torsion free, so is $F$.
\end{proof}

It follows from Theorem \ref{theorem.2} and Proposition \ref{proposition.13} that if $\C$ is an $n$-abelian category, then every $m$-spherical object in $\mod \C$ is $(n+1-m)$-torsion free, for all $1 \leqslant m \leqslant n$. By replacing $\C$ by $\C^{\op}$, we deduce that this property also holds for $\C^{\op}$-modules when $\C$ is $n$-abelian.\footnote{We register these facts in Theorem \ref{theorem.7}.} Furthermore, we can use Proposition \ref{proposition.13} to state the axioms (F2) and (F2$^{\op}$) in a few different ways, as we show in Theorem \ref{theorem.7}. But before we do that, let us describe $m$-spherical modules in the same spirit that $k$-torsion free modules were described in Proposition \ref{proposition.14}.

\begin{proposition}\label{proposition.11}
Assume that $\C$ is coherent, let $f \in \C(X,Y)$ be a morphism in $\C$, and let $m$ be a positive integer. The following are equivalent:
\begin{enumerate}
    \item[(a)] $\mr(f)$ is $m$-spherical.
    \item[(b)] $\pdim \mr(f) \leqslant m$ and for every exact sequence \[ \begin{tikzcd}
{\C(-,X_{m})} \arrow[r, "{\C(-,f_{m})}"] &[1.65em] \cdots \arrow[r, "{\C(-,f_{3})}"] &[1.4em] {\C(-,X_{2})} \arrow[r, "{\C(-,f_{2})}"] &[1.4em] {\C(-,X)} \arrow[r, "{\C(-,f)}"] &[1.2em] {\C(-,Y)}
\end{tikzcd} \] in $\mod \C$, the sequence \[ \begin{tikzcd}
{\C(Y,-)} \arrow[r, "{\C(f,-)}"] &[1.2em] {\C(X,-)} \arrow[r, "{\C(f_{2},-)}"] &[1.4em] {\C(X_{2},-)} \arrow[r, "{\C(f_{3},-)}"] &[1.4em] \cdots \arrow[r, "{\C(f_{m},-)}"] &[1.65em] {\C(X_{m},-)}
\end{tikzcd} \] is exact in $\mod \C^{\op}$. We agree that $X_{1} = X$, $X_{0} = Y$ and $f_{1} = f$.
    \item[(c)] There is an exact sequence \[ \begin{tikzcd}
0 \arrow[r] & {\C(-,X_{m})} \arrow[r, "{\C(-,f_{m})}"] &[1.65em] \cdots \arrow[r, "{\C(-,f_{3})}"] &[1.4em] {\C(-,X_{2})} \arrow[r, "{\C(-,f_{2})}"] &[1.4em] {\C(-,X)} \arrow[r, "{\C(-,f)}"] &[1.2em] {\C(-,Y)}
\end{tikzcd} \] in $\mod \C$ for which \[ \begin{tikzcd}
{\C(Y,-)} \arrow[r, "{\C(f,-)}"] &[1.2em] {\C(X,-)} \arrow[r, "{\C(f_{2},-)}"] &[1.4em] {\C(X_{2},-)} \arrow[r, "{\C(f_{3},-)}"] &[1.4em] \cdots \arrow[r, "{\C(f_{m},-)}"] &[1.65em] {\C(X_{m},-)}
\end{tikzcd} \] is exact in $\mod \C^{\op}$. We agree that $X_{1} = X$, $X_{0} = Y$ and $f_{1} = f$.
\end{enumerate}
\end{proposition}

\begin{proof}
Let \[ \begin{tikzcd}
{\C(-,X_{m})} \arrow[r, "{\C(-,f_{m})}"] &[1.65em] \cdots \arrow[r, "{\C(-,f_{3})}"] &[1.4em] {\C(-,X_{2})} \arrow[r, "{\C(-,f_{2})}"] &[1.4em] {\C(-,X)} \arrow[r, "{\C(-,f)}"] &[1.2em] {\C(-,Y)}
\end{tikzcd} \] be an exact sequence in $\mod \C$, which can be regarded as the beginning of a projective resolution of $\mr(f)$ in $\mod \C$. Given $Z \in \C$, by applying $\Hom(-,\C(-,Z))$ to the above sequence, we get a complex in $\Ab$ which is isomorphic to \[ \begin{tikzcd}
{\C(Y,Z)} \arrow[r, "{\C(f,Z)}"] &[1.2em] {\C(X,Z)} \arrow[r, "{\C(f_{2},Z)}"] &[1.4em] {\C(X_{2},Z)} \arrow[r, "{\C(f_{3},Z)}"] &[1.4em] \cdots \arrow[r, "{\C(f_{m},Z)}"] &[1.65em] {\C(X_{m},Z)}
\end{tikzcd}, \] by the Yoneda lemma. This complex is an exact sequence in $\Ab$ if and only if $\Ext^{i}(\mr(f),\C(-,Z)) = 0$ for all $1 \leqslant i \leqslant m - 1$. Consequently, the sequence \[ \begin{tikzcd}
{\C(Y,-)} \arrow[r, "{\C(f,-)}"] &[1.2em] {\C(X,-)} \arrow[r, "{\C(f_{2},-)}"] &[1.4em] {\C(X_{2},-)} \arrow[r, "{\C(f_{3},-)}"] &[1.4em] \cdots \arrow[r, "{\C(f_{m},-)}"] &[1.65em] {\C(X_{m},-)}
\end{tikzcd} \] is exact in $\mod \C^{\op}$ if and only if $\Ext^{i}(\mr(f),\C(-,Z)) = 0$ for all $1 \leqslant i \leqslant m - 1$ and all $Z \in \C$.

From the above paragraph, it follows that (a) implies (b), and that (c) implies (a). Below, we show that (b) implies (c).

Suppose that item (b) holds. Since $\pdim \mr(f) \leqslant m$, there is an exact sequence \[ \begin{tikzcd}
0 \arrow[r] & {\C(-,X_{m})} \arrow[r, "{\C(-,f_{m})}"] &[1.65em] \cdots \arrow[r, "{\C(-,f_{3})}"] &[1.4em] {\C(-,X_{2})} \arrow[r, "{\C(-,f_{2})}"] &[1.4em] {\C(-,X)} \arrow[r, "{\C(-,f)}"] &[1.2em] {\C(-,Y)}
\end{tikzcd} \] in $\mod \C$, which is obtained from a projective resolution of $\mr(f)$ in $\mod \C$. Then it follows from item (b) that \[ \begin{tikzcd}
{\C(Y,-)} \arrow[r, "{\C(f,-)}"] &[1.2em] {\C(X,-)} \arrow[r, "{\C(f_{2},-)}"] &[1.4em] {\C(X_{2},-)} \arrow[r, "{\C(f_{3},-)}"] &[1.4em] \cdots \arrow[r, "{\C(f_{m},-)}"] &[1.65em] {\C(X_{m},-)}
\end{tikzcd} \] is exact in $\mod \C^{\op}$.
\end{proof}

We are now ready to show a few different ways of stating the axioms (F2) and (F2$^{\op}$). Given positive integers $n$ and $k$ with $1 \leqslant k \leqslant n$, and a coherent category $\C$, consider the following axioms:

\begin{enumerate}
    \item[(F2$_{a}$)] Every $1$-spherical object in $\mod \C$ is $n$-torsion free.
    \item[(F2$_{b}$)] Every $m$-spherical object in $\mod \C$ is $(n+1-m)$-torsion free, for all $1 \leqslant m \leqslant n$.
    \item[(F2$_{c}$)] Every $m$-spherical object in $\mod \C$ is a syzygy, for all $1 \leqslant m \leqslant n$.
    \item[(F2$_{d_{k}}$)] Every $m$-spherical object in $\mod \C$ is $k$-torsion free, for all $1 \leqslant m \leqslant n+1-k$.
\end{enumerate}

For the sake of completeness, let us also state the duals of the above axioms.

\begin{enumerate}
    \item[(F2$_{a}^{\op}$)] Every $1$-spherical object in $\mod \C^{\op}$ is $n$-torsion free.
    \item[(F2$_{b}^{\op}$)] Every $m$-spherical object in $\mod \C^{\op}$ is $(n+1-m)$-torsion free, for \mbox{all $1 \leqslant m \leqslant n$.}
    \item[(F2$_{c}^{\op}$)] Every $m$-spherical object in $\mod \C^{\op}$ is a syzygy, for all $1 \leqslant m \leqslant n$.
    \item[(F2$_{d_{k}}^{\op}$)] Every $m$-spherical object in $\mod \C^{\op}$ is $k$-torsion free, for all $1 \leqslant m \leqslant n+1-k$.
\end{enumerate}

\begin{theorem}\label{theorem.7}
Assume that $\C$ is coherent. The following statement holds for $\square \in \{a,b,c,d_{k}\}$:
\begin{enumerate}
    \item[] \hspace{-2em} The axioms \textup{(F2)} and \textup{(F2$^{\op}$)} are equivalent to \textup{(F2$_{\square}$)} and \textup{(F2$_{\square}^{\op}$)}, respectively.
\end{enumerate}
\end{theorem}

\begin{proof}
It suffices to prove that the axioms (F2), (F2$_{a}$), (F2$_{b}$), (F2$_{c}$) and (F2$_{d_{k}}$) are all equivalent to each other. By duality, we deduce that their dual axioms are also equivalent to each other.

Trivially, (F2) is equivalent to (F2$_{a}$), and Proposition \ref{proposition.13} shows that (F2$_{a}$) implies (F2$_{b}$). Now, fix a positive integer $k$ such that $1 \leqslant k \leqslant n$. Clearly, (F2$_{b}$) implies (F2$_{d_{k}}$), and we prove below that (F2$_{d_{k}}$) implies (F2$_{a}$).

Suppose that every $m$-spherical object in $\mod \C$ is $k$-torsion free, for all $1 \leqslant m \leqslant n+1-k$. Let $F \in \mod \C$ be $1$-spherical, and take a projective resolution \[ \begin{tikzcd}
0 \arrow[r] & {\C(-,X)} \arrow[r, "{\C(-,f)}"] &[1.2em] {\C(-,Y)} \arrow[r] & F \arrow[r] & 0
\end{tikzcd} \] of $F$ in $\mod \C$. Then the morphism $f \in \C(X,Y)$ is such that $\mr(f) \simeq F$. Consequently, $\mr(f)$ is $k$-torsion free, so that it is $1$-torsion free. By Proposition \ref{proposition.14}, there is an exact sequence \[ \begin{tikzcd}
{\C(Y_{1},-)} \arrow[r, "{\C(g_{1},-)}"] &[1.4em] {\C(Y,-)} \arrow[r, "{\C(f,-)}"] &[1.2em] {\C(X,-)}
\end{tikzcd} \] in $\mod \C^{\op}$ for which \[ \begin{tikzcd}
0 \arrow[r] & {\C(-,X)} \arrow[r, "{\C(-,f)}"] &[1.2em] {\C(-,Y)} \arrow[r, "{\C(-,g_{1})}"] &[1.4em] {\C(-,Y_{1})}
\end{tikzcd} \] is exact in $\mod \C$. In this case, it follows from Proposition \ref{proposition.11} that $\mr(g_{1})$ is $2$-spherical. By repeating this argument consecutively, we obtain an exact sequence \[ \begin{tikzcd}
{\C(Y_{n-k},-)} \arrow[r, "{\C(g_{n-k},-)}"] &[2.3em] \cdots \arrow[r, "{\C(g_{2},-)}"] &[1.4em] {\C(Y_{1},-)} \arrow[r, "{\C(g_{1},-)}"] &[1.4em] {\C(Y,-)} \arrow[r, "{\C(f,-)}"] &[1.2em] {\C(X,-)}
\end{tikzcd} \] in $\mod \C^{\op}$ for which \[ \begin{tikzcd}
0 \arrow[r] &[-0.1em] {\C(-,X)} \arrow[r, "{\C(-,f)}"] &[1.2em] {\C(-,Y)} \arrow[r, "{\C(-,g_{1})}"] &[1.4em] {\C(-,Y_{1})} \arrow[r, "{\C(-,g_{2})}"] &[1.4em] \cdots \arrow[r, "{\C(-,g_{n-k})}"] &[2.3em] {\C(-,Y_{n-k})}
\end{tikzcd} \] is exact in $\mod \C$, so that $\mr(g_{n-k})$ is $(n+1-k)$-spherical. Therefore, as $\mr(g_{n-k})$ is $k$-torsion free, from Proposition \ref{proposition.14}, we deduce the existence of an exact sequence \[ \begin{tikzcd}
{\C(Y_{n},-)} \arrow[r, "{\C(g_{n},-)}"] &[1.45em] \cdots \arrow[r, "{\C(g_{2},-)}"] &[1.4em] {\C(Y_{1},-)} \arrow[r, "{\C(g_{1},-)}"] &[1.4em] {\C(Y,-)} \arrow[r, "{\C(f,-)}"] &[1.2em] {\C(X,-)}
\end{tikzcd} \] in $\mod \C^{\op}$ for which \[ \begin{tikzcd}
0 \arrow[r] & {\C(-,X)} \arrow[r, "{\C(-,f)}"] &[1.2em] {\C(-,Y)} \arrow[r, "{\C(-,g_{1})}"] &[1.4em] {\C(-,Y_{1})} \arrow[r, "{\C(-,g_{2})}"] &[1.4em] \cdots \arrow[r, "{\C(-,g_{n})}"] &[1.45em] {\C(-,Y_{n})}
\end{tikzcd} \] is exact in $\mod \C$. By Proposition \ref{proposition.14}, $\mr(f)$ is $n$-torsion free, hence so is $F$.

We have proved that, for a fixed positive integer $k$ with $1 \leqslant k \leqslant n$, the axioms (F2), (F2$_{a}$), (F2$_{b}$) and (F2$_{d_{k}}$) are all equivalent. Since we know from Corollary \ref{corollary.3} that (F2$_{c}$) corresponds to (F2$_{d_{k}}$) when $k = 1$, we are done.
\end{proof}

Observe that the axioms (F2$_{c}$) and (F2$_{c}^{\op}$) are stated only in terms of $\mod \C$ and $\mod \C^{\op}$, respectively, and do not depend on the transpose or on the dual of a module. Furthermore, note that, by Corollary \ref{corollary.3}, when $n \geqslant 2$ and $k = 2$, the axioms (F2$_{d_{k}}$) and (F2$_{d_{k}}^{\op}$) can be rephrased as ``$m$-spherical objects are reflexive, for all $1 \leqslant m \leqslant n - 1$''.

\subsection{Segments and cosegments}

Recall that, in Section \ref{section.3}, we reformulated the axioms (A1), (A1$^{\op}$), (A2) and (A2$^{\op}$) of an $n$-abelian category in terms of its categories of finitely presented functors, thereby obtaining the functorial axioms (F1), (F1$^{\op}$), (F2) and (F2$^{\op}$). Now, going in the opposite direction, we will show how to ``translate'' the axioms (F2$_{a}$), (F2$_{b}$), (F2$_{c}$), (F2$_{d_{k}}$), (F2$_{a}^{\op}$), (F2$_{b}^{\op}$), (F2$_{c}^{\op}$) and (F2$_{d_{k}}^{\op}$) to the language of higher homological algebra. But first, we need to define new terms in this language.

Let $m$ be a positive integer. Motivated by Proposition \ref{proposition.11}, we define an \textit{$m$-segment} in $\C$ to be a sequence \[ \begin{tikzcd}
X_{m} \arrow[r, "f_{m}"] &[0.2em] \cdots \arrow[r, "f_{2}"] & X_{1} \arrow[r, "f_{1}"] & X_{0}
\end{tikzcd} \] of morphisms in $\C$ for which \[ \begin{tikzcd}
0 \arrow[r] & {\C(-,X_{m})} \arrow[r, "{\C(-,f_{m})}"] &[1.65em] \cdots \arrow[r, "{\C(-,f_{2})}"] &[1.4em] {\C(-,X_{1})} \arrow[r, "{\C(-,f_{1})}"] &[1.4em] {\C(-,X_{0})}
\end{tikzcd} \] and \[ \begin{tikzcd}
{\C(X_{0},-)} \arrow[r, "{\C(f_{1},-)}"] &[1.4em] {\C(X_{1},-)} \arrow[r, "{\C(f_{2},-)}"] &[1.4em] \cdots \arrow[r, "{\C(f_{m},-)}"] &[1.65em] {\C(X_{m},-)}
\end{tikzcd} \] are exact sequences in $\Mod \C$ and in $\Mod \C^{\op}$, respectively. Dually, we define an \textit{$m$-cosegment} in $\C$ to be a sequence \[ \begin{tikzcd}
Y_{0} \arrow[r, "g_{1}"] & Y_{1} \arrow[r, "g_{2}"] & \cdots \arrow[r, "g_{m}"] &[0.2em] Y_{m}
\end{tikzcd} \] in $\C$ such that \[ \begin{tikzcd}
0 \arrow[r] & {\C(Y_{m},-)} \arrow[r, "{\C(g_{m},-)}"] &[1.65em] \cdots \arrow[r, "{\C(g_{2},-)}"] &[1.4em] {\C(Y_{1},-)} \arrow[r, "{\C(g_{1},-)}"] &[1.4em] {\C(Y_{0},-)}
\end{tikzcd} \] and \[ \begin{tikzcd}
{\C(-,Y_{0})} \arrow[r, "{\C(-,g_{1})}"] &[1.4em] {\C(-,Y_{1})} \arrow[r, "{\C(-,g_{2})}"] &[1.4em] \cdots \arrow[r, "{\C(-,g_{m})}"] &[1.65em] {\C(-,Y_{m})}
\end{tikzcd} \] are exact sequences in $\Mod \C^{\op}$ and in $\Mod \C$, respectively. Observe that a $1$-segment is the same as a monomorphism, and a $1$-cosegment is the same as an epimorphism.

Given two sequences of morphisms \[ \begin{tikzcd}
X_{m} \arrow[r, "f_{m}"] &[0.2em] \cdots \arrow[r, "f_{2}"] & X_{1} \arrow[r, "f_{1}"] & X_{0}
\end{tikzcd} \] and \[ \begin{tikzcd}
Y_{0} \arrow[r, "g_{1}"] & Y_{1} \arrow[r, "g_{2}"] & \cdots \arrow[r, "g_{k}"] & Y_{k}
\end{tikzcd} \] in $\C$ with $X_{0} = Y_{0}$, we define their \textit{concatenation} to be the sequence \[ \begin{tikzcd}
X_{m} \arrow[r, "f_{m}"] &[0.2em] \cdots \arrow[r, "f_{2}"] & X_{1} \arrow[r, "f_{1}"] & X_{0} \arrow[r, "g_{1}"] & Y_{1} \arrow[r, "g_{2}"] & \cdots \arrow[r, "g_{k}"] & Y_{k}
\end{tikzcd}. \] Note that every $n$-exact sequence in $\C$ is the concatenation of an $m$-segment in $\C$ with an $(n+1-m)$-cosegment in $\C$, for each $1 \leqslant m \leqslant n$.

For a positive integer $m$ with $1 \leqslant m \leqslant n$, we say that an $m$-segment \[ \begin{tikzcd}
X_{m} \arrow[r, "f_{m}"] &[0.2em] \cdots \arrow[r, "f_{2}"] & X_{1} \arrow[r, "f_{1}"] & X_{0}
\end{tikzcd} \] in $\C$ \textit{fits into an $n$-exact sequence} in $\C$ if there is an $(n+1-m)$-cosegment \[ \begin{tikzcd}
Y_{0} \arrow[r, "g_{1}"] & Y_{1} \arrow[r, "g_{2}"] & \cdots \arrow[r, "g_{n+1-m}"] &[1.8em] Y_{n+1-m}
\end{tikzcd} \] in $\C$ with $X_{0} = Y_{0}$ for which the concatenation \[ \begin{tikzcd}
X_{m} \arrow[r, "f_{m}"] &[0.2em] \cdots \arrow[r, "f_{2}"] & X_{1} \arrow[r, "f_{1}"] & X_{0} \arrow[r, "g_{1}"] & Y_{1} \arrow[r, "g_{2}"] & \cdots \arrow[r, "g_{n+1-m}"] &[1.8em] Y_{n+1-m}
\end{tikzcd} \] is an $n$-exact sequence in $\C$. Dually, we say that an $m$-cosegment \[ \begin{tikzcd}
Y_{0} \arrow[r, "g_{1}"] & Y_{1} \arrow[r, "g_{2}"] & \cdots \arrow[r, "g_{m}"] &[0.2em] Y_{m}
\end{tikzcd} \] in $\C$ \textit{fits into an $n$-exact sequence} in $\C$ if there is an $(n+1-m)$-segment \[ \begin{tikzcd}
X_{n+1-m} \arrow[r, "f_{n+1-m}"] &[1.8em] \cdots \arrow[r, "f_{2}"] & X_{1} \arrow[r, "f_{1}"] & X_{0}
\end{tikzcd} \] in $\C$ with $Y_{0} = X_{0}$ for which the concatenation \[ \begin{tikzcd}
X_{n+1-m} \arrow[r, "f_{n+1-m}"] &[1.8em] \cdots \arrow[r, "f_{2}"] & X_{1} \arrow[r, "f_{1}"] & Y_{0} \arrow[r, "g_{1}"] & Y_{1} \arrow[r, "g_{2}"] & \cdots \arrow[r, "g_{m}"] &[0.2em] Y_{m}
\end{tikzcd} \] is an $n$-exact sequence in $\C$.

\begin{proposition}\label{proposition.9}
Assume that $\C$ is coherent, and let $m$ and $k$ be positive integers. The following are equivalent:
\begin{enumerate}
    \item[(a)] Every $m$-spherical object in $\mod \C$ is $k$-torsion free.
    \item[(b)] Every $m$-segment in $\C$ can be extended to an $(m+k)$-segment in $\C$.
\end{enumerate}
Furthermore, if $\C$ is pre-$n$-abelian and $m + k = n + 1$, then the above items are also equivalent to:
\begin{enumerate}
    \item[(c)] Every $m$-segment in $\C$ fits into an $n$-exact sequence in $\C$.
\end{enumerate}
\end{proposition}

\begin{proof}
Suppose that every $m$-spherical object in $\mod \C$ is $k$-torsion free. If \[ \begin{tikzcd}
X_{m} \arrow[r, "f_{m}"] &[0.2em] \cdots \arrow[r, "f_{2}"] & X_{1} \arrow[r, "f_{1}"] & X_{0}
\end{tikzcd} \] is an $m$-segment in $\C$, then it follows from Proposition \ref{proposition.11} that $\mr(f_{1})$ is $m$-spherical. Therefore, $\mr(f_{1})$ is $k$-torsion free, and it follows from Proposition \ref{proposition.14} that there is a sequence \[ \begin{tikzcd}
X_{0} \arrow[r, "g_{1}"] & Y_{1} \arrow[r, "g_{2}"] & \cdots \arrow[r, "g_{k}"] & Y_{k}
\end{tikzcd} \] of morphisms in $\C$ for which \[ \begin{tikzcd}
X_{m} \arrow[r, "f_{m}"] &[0.2em] \cdots \arrow[r, "f_{2}"] & X_{1} \arrow[r, "f_{1}"] & X_{0} \arrow[r, "g_{1}"] & Y_{1} \arrow[r, "g_{2}"] & \cdots \arrow[r, "g_{k}"] & Y_{k}
\end{tikzcd} \] is an $(m+k)$-segment in $\C$.

Conversely, suppose that every $m$-segment in $\C$ can be extended to an $(m+k)$-segment in $\C$. If $F \in \mod \C$ is $m$-spherical, then it has a projective resolution \[ \begin{tikzcd}
0 \arrow[r] & {\C(-,X_{m})} \arrow[r, "{\C(-,f_{m})}"] &[1.65em] \cdots \arrow[r, "{\C(-,f_{2})}"] &[1.4em] {\C(-,X_{1})} \arrow[r, "{\C(-,f_{1})}"] &[1.4em] {\C(-,X_{0})} \arrow[r] & F \arrow[r] & 0
\end{tikzcd} \] in $\mod \C$, so that $\mr(f_{1}) \simeq F$. By Proposition \ref{proposition.11}, the sequence \[ \begin{tikzcd}
{\C(X_{0},-)} \arrow[r, "{\C(f_{1},-)}"] &[1.4em] {\C(X_{1},-)} \arrow[r, "{\C(f_{2},-)}"] &[1.4em] \cdots \arrow[r, "{\C(f_{m},-)}"] &[1.65em] {\C(X_{m},-)}
\end{tikzcd} \] is exact in $\mod \C^{\op}$. Therefore, \[ \begin{tikzcd}
X_{m} \arrow[r, "f_{m}"] &[0.2em] \cdots \arrow[r, "f_{2}"] & X_{1} \arrow[r, "f_{1}"] & X_{0}
\end{tikzcd} \] is an $m$-segment in $\C$, and we can extend it to an $(m + k)$-segment \[ \begin{tikzcd}
X_{m} \arrow[r, "f_{m}"] &[0.2em] \cdots \arrow[r, "f_{2}"] & X_{1} \arrow[r, "f_{1}"] & X_{0} \arrow[r, "g_{1}"] & Y_{1} \arrow[r, "g_{2}"] & \cdots \arrow[r, "g_{k}"] & Y_{k}
\end{tikzcd} \] in $\C$. Then, by Proposition \ref{proposition.14}, $\mr(f_{1})$ is $k$-torsion free, hence so is $F$.

We have proved that the items (a) and (b) are equivalent. Now, assume that $\C$ is pre-$n$-abelian and that $m + k = n + 1$. Clearly, (c) implies (b). Below, we show that (a) implies (c).

Suppose that every $m$-spherical object in $\mod \C$ is $k$-torsion free. If \[ \begin{tikzcd}
X_{m} \arrow[r, "f_{m}"] &[0.2em] \cdots \arrow[r, "f_{2}"] & X_{1} \arrow[r, "f_{1}"] & X_{0}
\end{tikzcd} \] is an $m$-segment in $\C$, then we know from Proposition \ref{proposition.11} that $\mr(f_{1})$ is $m$-spherical, hence $\mr(f_{1})$ is $k$-torsion free. Moreover, observe that $\pdim \ml(f_{1}) \leqslant n + 2 - m$, that is, $\pdim \ml(f_{1}) \leqslant k + 1$. Thus, it follows from Proposition \ref{proposition.14} that there is a $k$-cosegment \[ \begin{tikzcd}
X_{0} \arrow[r, "g_{1}"] & Y_{1} \arrow[r, "g_{2}"] & \cdots \arrow[r, "g_{k}"] & Y_{k}
\end{tikzcd} \] in $\C$ for which \[ \begin{tikzcd}
X_{m} \arrow[r, "f_{m}"] &[0.2em] \cdots \arrow[r, "f_{2}"] & X_{1} \arrow[r, "f_{1}"] & X_{0} \arrow[r, "g_{1}"] & Y_{1} \arrow[r, "g_{2}"] & \cdots \arrow[r, "g_{k}"] & Y_{k}
\end{tikzcd} \] is an $n$-exact sequence in $\C$.
\end{proof}

By using Proposition \ref{proposition.9}, we can now easily ``translate'' the axioms (F2$_{a}$), (F2$_{b}$), (F2$_{c}$), (F2$_{d_{k}}$), (F2$_{a}^{\op}$), (F2$_{b}^{\op}$), (F2$_{c}^{\op}$) and (F2$_{d_{k}}^{\op}$) to the language of higher homological algebra. Indeed, for an additive and idempotent complete category $\C$, and positive integers $n$ and $k$ with $1 \leqslant k \leqslant n$, consider the following axioms:

\begin{enumerate}
    \item[(A2$_{a}$)] Every monomorphism in $\C$ fits into an $n$-exact sequence in $\C$.
    \item[(A2$_{b}$)] Every $m$-segment in $\C$ fits into an $n$-exact sequence in $\C$, for all $1 \leqslant m \leqslant n$.
    \item[(A2$_{c}$)] Every $m$-segment in $\C$ is an $m$-kernel in $\C$, for all $1 \leqslant m \leqslant n$.
    \item[(A2$_{d_{k}}$)] Every $m$-segment in $\C$ can be extended to an $(m+k)$-segment in $\C$, for all $1 \leqslant m \leqslant n+1-k$.
\end{enumerate}

And their duals:

\begin{enumerate}
    \item[(A2$_{a}^{\op}$)] Every epimorphism in $\C$ fits into an $n$-exact sequence in $\C$.
    \item[(A2$_{b}^{\op}$)] Every $m$-cosegment in $\C$ fits into an $n$-exact sequence in $\C$, for all $1 \leqslant m \leqslant n$.
    \item[(A2$_{c}^{\op}$)] Every $m$-cosegment in $\C$ is an $m$-cokernel in $\C$, for all $1 \leqslant m \leqslant n$.
    \item[(A2$_{d_{k}}^{\op}$)] Every $m$-cosegment in $\C$ can be extended to an $(m+k)$-cosegment in $\C$, for all $1 \leqslant m \leqslant n+1-k$.
\end{enumerate}

\begin{proposition}\label{proposition.12}
Assume that $\C$ is coherent, and consider the following statement:
\begin{enumerate}
    \item[] \hspace{-2.5em} The axioms \textup{(F2$_{\square}$)} and \textup{(F2$_{\square}^{\op}$)} are equivalent to \textup{(A2$_{\square}$)} and \textup{(A2$_{\square}^{\op}$)}, respectively.
\end{enumerate}
The above statement holds for $\square \in \{c,d_{k}\}$. Furthermore, if $\C$ is pre-$n$-abelian, then the statement also holds for $\square \in \{a,b\}$.
\end{proposition}

\begin{proof}
The assertions concerning $\square \in \{a,b,d_{k}\}$ follow from Proposition \ref{proposition.9}, while the claim concerning $\square = c$ is easily verified with the assistance of Proposition \ref{proposition.11}.
\end{proof}

We can now conclude that the previous axioms are  equivalent ways of stating the axioms (A2) and (A2$^{\op}$) for $\C$, assuming that $\C$ is pre-$n$-abelian. We remark that the equivalences between the axioms (A2) and (A2$_{a}$) and between (A2$^{\op}$) and (A2$_{a}^{\op}$) were already known, see \cite[Remark 3.2]{MR3519980}.

\begin{theorem}\label{theorem.3}
Assume that $\C$ is pre-$n$-abelian. The following statement holds for $\square \in \{a,b,c,d_{k}\}$:
\begin{enumerate}
    \item[] \hspace{-2em} The axioms \textup{(A2)} and \textup{(A2$^{\op}$)} are equivalent to \textup{(A2$_{\square}$)} and \textup{(A2$_{\square}^{\op}$)}, respectively.
\end{enumerate}
\end{theorem}

\begin{proof}
Follows from Theorem \ref{theorem.7} and Propositions \ref{proposition.3} and \ref{proposition.12}.
\end{proof}

Finally, observe that, when $n = 1$, the axioms (A2$_{c}$) and (A2$_{c}^{\op}$) can be rephrased as ``every monomorphism is a kernel'' and ``every epimorphism is a cokernel'', respectively. Because of the fundamental importance of these axioms for the case $n = 1$, we present the following alternative definition of an $n$-abelian category:

\begin{theorem}\label{theorem.8}
An additive and idempotent complete category $\C$ is $n$-abelian if and only if $\C$ satisfies the following axioms: \begin{enumerate}
    \item[(A1)] $\C$ has $n$-kernels.
    \item[(A1$^{\op}$)] $\C$ has $n$-cokernels.
    \item[(A2$_{c}$)] Every $m$-segment in $\C$ is an $m$-kernel in $\C$, for all $1 \leqslant m \leqslant n$.
    \item[(A2$_{c}^{\op}$)] Every $m$-cosegment in $\C$ is an $m$-cokernel in $\C$, for all $1 \leqslant m \leqslant n$.
\end{enumerate}
\end{theorem}

\begin{proof}
Follows from Theorem \ref{theorem.3}.
\end{proof}

\subsection{The general case}

In this subsection, we prove that finitely presented modules of projective dimension at most $m$ over an $n$-abelian category are $(n+1-m)$-torsion free, even if the modules are not $m$-spherical. Then, as in Subsection \ref{subsection.spherical}, we present a few equivalent statements for the axioms (F2) and (F2$^{\op}$).

\begin{lemma}\label{lemma.4}
Assume that $\C$ is left coherent. Let $F \in \mod \C$ and consider $\Tr F \in \mod \C^{\op}$, a transpose of $F$. If $F^{\ast} = 0$, then $\pdim \Tr F \leqslant 1$.
\end{lemma}

\begin{proof}
This is straightforward and left to the reader.
\end{proof}

The next result (and its proof) is essentially from \cite[Proposition 3.7]{MR3638352}.

\begin{proposition}\label{proposition.20}
Assume that $\C$ is coherent, and let $m$ and $k$ be positive integers, where $k \geqslant 2$. If every object of projective dimension at most $m$ in $\mod \C$ is $k$-torsion free, then every object of projective dimension at most $m + 1$ in $\mod \C$ is $(k-1)$-torsion free.
\end{proposition}

\begin{proof}
Suppose that every object of projective dimension at most $m$ in $\mod \C$ is $k$-torsion free. Let $F \in \mod \C$ be such that $\pdim F \leqslant m + 1$. It follows from \mbox{Proposition \ref{proposition.19}} that, for certain choices of transposes and syzygies, there is a transpose $\Tr F$ of $F$ for which its double dual sequence induces a short exact sequence \[ \begin{tikzcd}
0 \arrow[r] & \E^{1}(\Tr F) \arrow[r] & \Tr F \arrow[r] & \Omega \Tr \Omega F \arrow[r] & 0
\end{tikzcd} \] in $\mod \C^{\op}$. By applying the functor $\Hom(-,\C(X,-))$ in the above sequence and considering the induced long exact sequence, for each $X \in \C$, then using the fact that $\Ext^{i}(\Omega \Tr \Omega F, -) \simeq \Ext^{i+1}(\Tr \Omega F, -)$, we deduce that there is a long exact sequence \[ \begin{tikzcd}
0 \arrow[r] & (\Omega \Tr \Omega F)^{\ast} \arrow[r] & (\Tr F)^{\ast} \arrow[r] & (\E^{1}(\Tr F))^{\ast} \arrow[out=0, in=180, looseness=2, overlay, lld]    \\[0.4em]
            & \E^{2}(\Omega F) \arrow[r]             & \E^{1}(F) \arrow[r]      & \E^{1}(\Tr \E^{1}(\Tr F)) \arrow[out=0, in=180, looseness=2, overlay, lld] \\[0.4em]
            & \E^{3}(\Omega F) \arrow[r]                      & \E^{2}(F) \arrow[r]               & \E^{2}(\Tr \E^{1}(\Tr F)) \arrow[out=0, in=180, looseness=2, overlay, lld, shorten >= 1.08em,shorten <= 0em] \\[0.4em]
            & \cdots                                 &                          &                                             
\end{tikzcd} \] in $\mod \C$. Since the projective dimension of $\Omega F$ is at most $m$, it is $k$-torsion free. Hence $\E^{i}(\Omega F) = 0$ for all $1 \leqslant i \leqslant k$, so that there is an exact sequence \[ \begin{tikzcd}
0 \arrow[r] & \E^{i}(F) \arrow[r] & \E^{i}(\Tr \E^{1}(\Tr F))
\end{tikzcd} \] in $\mod \C$ for each $1 \leqslant i \leqslant k - 1$. Now, because $\E^{2}(\Omega F) = 0$, the previous long exact sequence in $\mod \C$ leads to a short exact sequence \[ \begin{tikzcd}
0 \arrow[r] & (\Omega \Tr \Omega F)^{\ast} \arrow[r] & (\Tr F)^{\ast} \arrow[r] & (\E^{1}(\Tr F))^{\ast} \arrow[r] & 0
\end{tikzcd} \] in $\mod \C$. Moreover, by Proposition \ref{proposition.19}, $(\Tr F)^{\ast} \to (\E^{1}(\Tr F))^{\ast}$ is the zero morphism. Consequently, $(\E^{1}(\Tr F))^{\ast} = 0$, and it follows from Lemma \ref{lemma.4} that $\pdim \Tr \E^{1}(\Tr F) \leqslant 1$. Hence $\Tr \E^{1}(\Tr F)$ is $k$-torsion free, that is, $\E^{i}(\Tr \E^{1}(\Tr F)) = 0$ for all $1 \leqslant i \leqslant k$. Therefore, we deduce from above that $\E^{i}(F) = 0$ for all $1 \leqslant i \leqslant k - 1$, that is, $F$ is $(k-1)$-torsion free.
\end{proof}

We conclude from Theorem \ref{theorem.2} and Proposition \ref{proposition.20} that if $\C$ is an $n$-abelian category, then finitely presented $\C$-modules and $\C^{\op}$-modules of projective dimension at most $m$ are $(n+1-m)$-torsion free, for all $1 \leqslant m \leqslant n$. Moreover, Proposition \ref{proposition.20} makes it possible to describe the axioms (F2) and (F2$^{\op}$) in a few different ways. In fact, given positive integers $n$ and $k$ with $1 \leqslant k \leqslant n$, and a coherent category $\C$, consider the following axioms:

\begin{enumerate}
    \item[(F2$_{e}$)] Every $F \in \mod \C$ with $\pdim F \leqslant m$ is $(n+1-m)$-torsion free, for all $1 \leqslant m \leqslant n$.
    \item[(F2$_{f}$)] Every $F \in \mod \C$ with $\pdim F \leqslant n$ is a syzygy.
    \item[(F2$_{g_{k}}$)] Every $F \in \mod \C$ with $\pdim F \leqslant n+1-k$ is $k$-torsion free.
\end{enumerate}

And their duals:

\begin{enumerate}
    \item[(F2$_{e}^{\op}$)] Every $F \in \mod \C^{\op}$ with $\pdim F \leqslant m$ is $(n+1-m)$-torsion free, for \mbox{all $1 \leqslant m \leqslant n$.}
    \item[(F2$_{f}^{\op}$)] Every $F \in \mod \C^{\op}$ with $\pdim F \leqslant n$ is a syzygy.
    \item[(F2$_{g_{k}}^{\op}$)] Every $F \in \mod \C^{\op}$ with $\pdim F \leqslant n+1-k$ is $k$-torsion free.
\end{enumerate}

\begin{theorem}\label{theorem.9}
Assume that $\C$ is coherent. The following statement holds for $\square \in \{e,f,g_{k}\}$:
\begin{enumerate}
    \item[] \hspace{-2em} The axioms \textup{(F2)} and \textup{(F2$^{\op}$)} are equivalent to \textup{(F2$_{\square}$)} and \textup{(F2$_{\square}^{\op}$)}, respectively.
\end{enumerate}
\end{theorem}

\begin{proof}
This is similar to the proof of Theorem \ref{theorem.7}, with the difference that it relies on Proposition \ref{proposition.20} instead of Proposition \ref{proposition.13}. The details are left to the reader.
\end{proof}

We remark that, as it was the case for the axioms (F2$_{c}$) and (F2$_{c}^{\op}$), the axioms (F2$_{f}$) and (F2$_{f}^{\op}$) do not depend on the transpose or on the dual of a module. Moreover, by Corollary \ref{corollary.3}, when $n \geqslant 2$ and $k = 2$, the axioms (F2$_{g_{k}}$) and (F2$_{g_{k}}^{\op}$) can be stated as ``finitely presented modules of projective dimension at most $n - 1$ are reflexive''.

\subsection{Pre-segments and pre-cosegments}

As we saw in Proposition \ref{proposition.11}, when $\C$ is coherent, $m$-spherical objects in $\mod \C$ and in $\mod \C^{\op}$ correspond to $m$-segments in $\C$ and $m$-cosegments in $\C$, respectively. This idea allowed us to describe the axioms (F2$_{a}$), (F2$_{b}$), (F2$_{c}$), (F2$_{d_{k}}$) and their duals in terms of $m$-segments and $m$-cosegments. Now, we could aim to proceed similarly with the axioms (F2$_{e}$), (F2$_{f}$), (F2$_{g_{k}}$) and their duals. To achieve this goal, we rely on the idea that objects of projective dimension at most $m$ in $\mod \C$ and in $\mod \C^{\op}$ correspond to ``pre-$m$-segments'' in $\C$ and ``pre-$m$-cosegments'' in $\C$, respectively, which are concepts that we now define.

Let $m$ be a positive integer. We define a \textit{pre-$m$-segment} in $\C$ to be a sequence \[ \begin{tikzcd}
X_{m} \arrow[r, "f_{m}"] &[0.2em] \cdots \arrow[r, "f_{2}"] & X_{1} \arrow[r, "f_{1}"] & X_{0}
\end{tikzcd} \] of morphisms in $\C$ for which \[ \begin{tikzcd}
0 \arrow[r] & {\C(-,X_{m})} \arrow[r, "{\C(-,f_{m})}"] &[1.65em] \cdots \arrow[r, "{\C(-,f_{2})}"] &[1.4em] {\C(-,X_{1})} \arrow[r, "{\C(-,f_{1})}"] &[1.4em] {\C(-,X_{0})}
\end{tikzcd} \] is an exact sequence in $\Mod \C$. Dually, we define a \textit{pre-$m$-cosegment} in $\C$ to be a sequence \[ \begin{tikzcd}
Y_{0} \arrow[r, "g_{1}"] & Y_{1} \arrow[r, "g_{2}"] & \cdots \arrow[r, "g_{m}"] &[0.2em] Y_{m}
\end{tikzcd} \] in $\C$ such that \[ \begin{tikzcd}
0 \arrow[r] & {\C(Y_{m},-)} \arrow[r, "{\C(g_{m},-)}"] &[1.65em] \cdots \arrow[r, "{\C(g_{2},-)}"] &[1.4em] {\C(Y_{1},-)} \arrow[r, "{\C(g_{1},-)}"] &[1.4em] {\C(Y_{0},-)}
\end{tikzcd} \] is an exact sequence in $\Mod \C^{\op}$. Note that every $m$-segment is a pre-$m$-segment as well as every $m$-cosegment is a pre-$m$-cosegment. Moreover, a pre-$1$-segment is the same as a monomorphism, and a pre-$1$-cosegment is the same as an epimorphism.

It is worth remarking that, when $m \leqslant n$, every pre-$m$-segment can be regarded as a pre-$n$-segment and every pre-$m$-cosegment can be regarded as a pre-$n$-cosegment, by adding zeros to their left and right, respectively. Furthermore, observe that every $n$-kernel is a pre-$n$-segment and every $n$-cokernel is a pre-$n$-cosegment. We can now understand the axioms (F2$_{f}$) and (F2$_{f}^{\op}$) as the converses of the these statements. In fact, consider the following axioms:

\begin{enumerate}
    \item[(A2$_{f}$)] Every pre-$n$-segment in $\C$ is an $n$-kernel in $\C$.
    \item[(A2$_{f}^{\op}$)] Every pre-$n$-cosegment in $\C$ is an $n$-cokernel in $\C$.
\end{enumerate}

\begin{proposition}\label{proposition.21}
If $\C$ is coherent, then the axioms \textup{(F2$_{f}$)} and \textup{(F2$_{f}^{\op}$)} are equivalent to \textup{(A2$_{f}$)} and \textup{(A2$_{f}^{\op}$)}, respectively.
\end{proposition}

\begin{proof}
This is straightforward and left to the reader.
\end{proof}

We can also describe the axioms (F2$_{e}$), (F2$_{g_{k}}$) and their duals in terms of pre-$m$-segments and pre-$m$-cosegments, but we leave this task to the reader. Note that their descriptions would not necessarily involve $n$-exact sequences.

Let us conclude that the axioms (A2$_{f}$) and (A2$_{f}^{\op}$) for $\C$ are equivalent formulations of (A2) and (A2$^{\op}$), assuming that $\C$ is pre-$n$-abelian.

\begin{theorem}\label{theorem.10}
Assume that $\C$ is pre-$n$-abelian. The axioms \textup{(A2)} and \textup{(A2$^{\op}$)} are equivalent to \textup{(A2$_{f}$)} and \textup{(A2$_{f}^{\op}$)}, respectively.
\end{theorem}

\begin{proof}
Follows from Theorem \ref{theorem.9} and Propositions \ref{proposition.3} and \ref{proposition.21}.
\end{proof}

Observe that, as it was the case for the axioms (A2$_{c}$) and (A2$_{c}^{\op}$), the axioms (A2$_{f}$) and (A2$_{f}^{\op}$) also generalize the statements ``every monomorphism is a kernel'' and ``every epimorphism is a cokernel'', which are recovered when $n = 1$. For this reason, in addition to Theorem \ref{theorem.8}, we present another alternative definition \mbox{of an $n$-abelian category:}

\begin{theorem}\label{theorem.11}
An additive and idempotent complete category $\C$ is $n$-abelian if and only if $\C$ satisfies the following axioms: \begin{enumerate}
    \item[(A1)] $\C$ has $n$-kernels.
    \item[(A1$^{\op}$)] $\C$ has $n$-cokernels.
    \item[(A2$_{f}$)] Every pre-$n$-segment in $\C$ is an $n$-kernel in $\C$.
    \item[(A2$_{f}^{\op}$)] Every pre-$n$-cosegment in $\C$ is an $n$-cokernel in $\C$.
\end{enumerate}
\end{theorem}

\begin{proof}
Follows from Theorem \ref{theorem.10}.
\end{proof}

The axioms (A2$_{f}$) and (A2$_{f}^{\op}$) suggest that, perhaps, ``pre-$m$-segments'' and ``pre-$m$-cosegments'' should be called ``$m$-monomorphisms'' and ``$m$-epimorphisms'', respectively. However, at the moment, it is not clear what higher (that is, longer) analogues of monomorphisms and epimorphisms should be in an $n$-abelian category when $n \geqslant 2$. For example, due to the axioms (A2$_{c}$) and (A2$_{c}^{\op}$), other candidates for such analogues are ``$m$-segments'' and ``$m$-cosegments''. The question of what these analogues should be requires, therefore, further investigation.

\section{Categories with enough injectives or projectives}\label{section.7}

In this section, we describe when an additive and idempotent complete category $\C$ is $n$-abelian and has enough injectives in terms of $\mod \C$. This is done by expressing the axiom (F2$_{f}$) for $\C$ in terms of the ``dominant dimension'' of $\mod \C$, provided that $\C$ is right coherent, ``right comprehensive'' and has enough injectives. By duality, we also obtain a similar description for when $\C$ is $n$-abelian and has enough projectives, which is given in terms of $\mod \C^{\op}$.

Since the ``dominant dimension'' of an abelian category is defined via projective injective objects, we first need to understand what these objects are in the categories $\mod \C$ and $\mod \C^{\op}$. In this direction, we have the following general result, which motivates our subsequent definition:

\begin{proposition}\label{proposition.24}
Let $\C$ be an additive and idempotent complete category. Every object in $\proj \C$ that is injective in $\mod \C$ is injective in $\proj \C$.
\end{proposition}

\begin{proof}
To begin with, recall that the Yoneda embedding induces an equivalence of categories $\C \approx \proj \C$. Thus, all objects in $\proj \C$ are isomorphic to $\C(-,X)$ for some $X \in \C$, and fixed $X,Y \in \C$, a morphism from $\C(-,X)$ to $\C(-,Y)$ is given by $\C(-,f)$, where $f \in \C(X,Y)$. Furthermore, a morphism $f$ in $\C$ is a monomorphism in $\C$ if and only if $\C(-,f)$ is a monomorphism in $\proj \C$.

Let $Z \in \C$ and assume that $\C(-,Z)$ is injective in $\mod \C$. Consider morphisms $\C(-,f)$ and $\C(-,g)$ in $\proj \C$, where $f \in \C(X,Y)$ and $g \in \C(X,Z)$, and suppose that $\C(-,f)$ is a monomorphism in $\proj \C$. In this case, $f$ is a monomorphism in $\C$, which implies that $\C(-,f)$ is a monomorphism in $\mod \C$.\footnote{Actually, since $f$ is a monomorphism, $\C(-,f)$ is a monomorphism in $\Mod \C$, hence it is a monomorphism in every subcategory of $\Mod \C$ containing it.} Now, because $\C(-,Z)$ is injective in $\mod \C$, there is a morphism $h \in \C(Y,Z)$ such that $\C(-,g) = \C(-,h) \C(-,f)$. Therefore, $\C(-,Z)$ is injective in $\proj \C$.
\end{proof}

We say that $\C$ is \textit{right comprehensive} if every object in $\proj \C$ that is injective in $\proj \C$ is injective in $\mod \C$. Dually, $\C$ is called \textit{left comprehensive} if all the objects in $\proj \C^{\op}$ that are injective in $\proj \C^{\op}$ are also injective in $\mod \C^{\op}$. If $\C$ is both right and left comprehensive, then we say that $\C$ is \textit{comprehensive}. These definitions are explained in more detail below.

In summary, Proposition \ref{proposition.24} says that the inclusion functor $\proj \C \to \mod \C$ always reflects injectivity, and when it also preserves injectivity, we say that $\C$ right comprehensive. Dually, the inclusion $\proj \C^{\op} \to \mod \C^{\op}$ reflects injectivity, and when injectivity is also preserved, we call $\C$ left comprehensive. From another perspective, observe from its proof that Proposition \ref{proposition.24} says that if $X \in \C$ is such that $\C(-,X)$ is injective in $\mod \C$, then $X$ is injective in $\C$. Following this point of view, it is easy to conclude that $\C$ is right comprehensive if and only if $\C(-,I)$ is injective in $\mod \C$ for every object $I \in \C$ that is injective in $\C$. By duality, we get that if $X \in \C$ and $\C(X,-)$ is injective in $\mod \C^{\op}$, then $X$ is projective in $\C$. Moreover, $\C$ is left comprehensive if and only if $\C(P,-)$ is injective in $\mod \C^{\op}$ for every object $P \in \C$ that is \mbox{projective in $\C$.}

Trivially, by Proposition \ref{proposition.4}, every von Neumann regular category is comprehensive. Furthermore, we can conclude from the next proposition and a result of Jasso, \cite[Theorem 3.12]{MR3519980}, that every $n$-abelian category is comprehensive.

\begin{proposition}\label{proposition.25}
Assume that $\C$ is right coherent. Then $\C$ is right comprehensive if and only if for every sequence \[ \begin{tikzcd}
X \arrow[r, "f"] & Y \arrow[r, "g"] & Z
\end{tikzcd} \] in $\C$, where $f$ is a weak kernel of $g$, and for every injective object $I \in \C$, the sequence \[ \begin{tikzcd}
{\C(Z,I)} \arrow[r, "{\C(g,I)}"] &[1em] {\C(Y,I)} \arrow[r, "{\C(f,I)}"] &[1em] {\C(X,I)}
\end{tikzcd} \] is exact in $\Ab$.
\end{proposition}

\begin{proof}
Let $I \in \C$ be an injective object in $\C$ and consider a sequence \[ \begin{tikzcd}
X \arrow[r, "f"] & Y \arrow[r, "g"] & Z
\end{tikzcd} \] in $\C$. Suppose that $f$ is a weak kernel of $g$, which is equivalent to assuming that \[ \begin{tikzcd}
{\C(-,X)} \arrow[r, "{\C(-,f)}"] &[1.2em] {\C(-,Y)} \arrow[r, "{\C(-,g)}"] &[1.2em] {\C(-,Z)}
\end{tikzcd} \] is an exact sequence in $\mod \C$. Then this sequence can be regarded as the beginning of a projective resolution of $\mr(g)$ in $\mod \C$. By applying $\Hom(-,\C(-,I))$ to this sequence, we get a complex in $\Ab$ which is, by the Yoneda lemma, isomorphic to \[ \begin{tikzcd}
{\C(Z,I)} \arrow[r, "{\C(g,I)}"] &[1em] {\C(Y,I)} \arrow[r, "{\C(f,I)}"] &[1em] {\C(X,I)}
\end{tikzcd}. \] Thus, this complex is an exact sequence in $\Ab$ if and only if $\Ext^{1}(\mr(g), \C(-,I)) = 0$.

We can easily deduce our desired result from the above discussion, given that $\C$ is right comprehensive if and only if $\C(-,I)$ is injective in $\mod \C$ for every injective object $I \in \C$, that is, if and only if $\Ext^{1}(-,\C(-,I)) = 0$ for every injective object $I \in \C$.
\end{proof}

By taking $\C^{\op}$ in place of $\C$ in Proposition \ref{proposition.25}, we deduce that if $\C$ is left coherent, then $\C$ is left comprehensive if and only if for every sequence \[ \begin{tikzcd}
X \arrow[r, "f"] & Y \arrow[r, "g"] & Z
\end{tikzcd} \] in $\C$, where $g$ is a weak cokernel of $f$, and for every projective object $P \in \C$, the sequence \[ \begin{tikzcd}
{\C(P,X)} \arrow[r, "{\C(P,f)}"] &[1.2em] {\C(P,Y)} \arrow[r, "{\C(P,g)}"] &[1.2em] {\C(P,Z)}
\end{tikzcd} \] is exact in $\Ab$.

\begin{proposition}\label{proposition.27}
Every $n$-abelian category is comprehensive.
\end{proposition}

\begin{proof}
Follows from Proposition \ref{proposition.25} and \cite[Theorem 3.12]{MR3519980}, given that $n$-abelian categories are coherent, by Proposition \ref{proposition.1}.
\end{proof}

Because we are interested in $n$-abelian categories, and these are comprehensive, for the rest of this section, we will focus on the case where $\C$ is right or left comprehensive. Under such assumptions, we can characterize when $\C$ has enough injectives and enough projectives in terms of $\mod \C$ and $\mod \C^{\op}$, respectively. Indeed, we have the following:

\begin{proposition}\label{proposition.26}
Let $\C$ be an additive and idempotent complete category. Then $\C$ is right comprehensive and has enough injectives if and only if every projective object in $\mod \C$ embeds into a projective injective object in $\mod \C$.
\end{proposition}

\begin{proof}
Suppose that $\C$ is right comprehensive and has enough injectives. Given $X \in \C$, there is an injective object $I \in \C$ and a monomorphism $h \in \C(X,I)$ in $\C$. In this case, \[ \begin{tikzcd}
{\C(-,X)} \arrow[r, "{\C(-,h)}"] &[1.2em] {\C(-,I)}
\end{tikzcd} \] is a monomorphism in $\mod \C$, and $\C(-,I)$ is injective in $\mod \C$ since $\C$ is right comprehensive. Consequently, every object in $\proj \C$ embeds into a projective injective object in $\mod \C$.

Conversely, assume that every projective object in $\mod \C$ embeds into a projective injective object in $\mod \C$. Then, given $X \in \C$, there is a monomorphism \[ \begin{tikzcd}
{\C(-,X)} \arrow[r, "{\C(-,f)}"] &[1.2em] {\C(-,Y)}
\end{tikzcd} \] in $\mod \C$ with $f \in \C(X,Y)$ and $\C(-,Y)$ injective in $\mod \C$. In particular, $\C(-,f)$ is a monomorphism in $\proj \C$, which implies that $f$ is a monomorphism in $\C$ since the Yoneda embedding induces an equivalence $\C \approx \proj \C$. Moreover, it follows from Proposition \ref{proposition.24} that $Y$ is injective in $\C$. Hence $\C$ has enough injectives.

Now, in the above paragraph, assume that $X$ is injective in $\C$. Then $f$ is a split monomorphism, so that $\C(-,f)$ is a split monomorphism. Thus, given that $\C(-,Y)$ is injective in $\mod \C$, we conclude that $\C(-,X)$ is injective in $\mod \C$. Therefore, $\C$ is right comprehensive.
\end{proof}

Next, we recall a classical definition. Let $\A$ be an abelian category. Given $X \in \A$, the \textit{dominant dimension} of $X$, denoted by $\domdim X$, is the supremum of the set of positive integers $m$ for which there is an exact sequence \[ \begin{tikzcd}
0 \arrow[r] & X \arrow[r] & Y_{1} \arrow[r] & \cdots \arrow[r] & Y_{m}
\end{tikzcd} \] in $\A$ with each $Y_{i}$ being projective and injective in $\A$. If there is no such positive integer $m$, then we say that $\domdim X = 0$. The \textit{dominant dimension} of $\A$, denoted by $\domdim \A$, is the supremum of the set of nonnegative integers $m$ for which $\domdim P \geqslant m$ for every projective object $P \in \A$. Equivalently, the dominant dimension of $\A$ is the infimum of the dominant dimensions of projective objects in $\A$.

\begin{corollary}\label{corollary.6}
Assume that $\C$ is right coherent. Then $\C$ is right comprehensive and has enough injectives if and only if $\domdim (\mod \C) \geqslant 1$.
\end{corollary}

\begin{proof}
Follows from Proposition \ref{proposition.26}.
\end{proof}

By duality, we can deduce from Corollary \ref{corollary.6} that if $\C$ is left coherent, then $\C$ is left comprehensive and has enough projectives if and only if $\domdim (\mod \C^{\op}) \geqslant 1$.

\begin{proposition}\label{proposition.28}
Assume that $\C$ is an $n$-abelian category. Then the following hold:
\begin{enumerate}
    \item[(a)] $\C$ has enough injectives if and only if $\domdim (\mod \C) \geqslant 1$.
    \item[(b)] $\C$ has enough projectives if and only if $\domdim (\mod \C^{\op}) \geqslant 1$.
\end{enumerate}
\end{proposition}

\begin{proof}
Follows from Corollary \ref{corollary.6}, given that $n$-abelian categories are coherent and comprehensive, by Propositions \ref{proposition.1} and \ref{proposition.27}.
\end{proof}

Observe that if $\C$ is coherent, then it is not always the case that the dominant dimensions of $\mod \C$ and $\mod \C^{\op}$ coincide. Indeed, by Proposition \ref{proposition.28}, any abelian category that has enough injectives but not enough projectives (or vice versa) serves as a counterexample. Nonetheless, we will conclude in Corollary \ref{corollary.7} that if the dominant dimensions of $\mod \C$ and $\mod \C^{\op}$ are both nonzero, then they coincide. This will follow from the next few results, which deal with the relation between these dimensions and the axioms (F2) and (F2$^{\op}$) for $\C$.

We begin by proving that, under suitable conditions, the axioms (F2$_{f}$) and (F2$_{f}^{\op}$) for $\C$ are equivalent to $\domdim (\mod \C) \geqslant n + 1$ and $\domdim (\mod \C^{\op}) \geqslant n + 1$, respectively. In order to prove this result, let us first state some general facts.

\begin{lemma}\label{lemma.7}
Let $\A$ be an abelian category, let $m$ be a positive integer, and let \[ \begin{tikzcd}
0 \arrow[r] & X \arrow[r] & Y \arrow[r] & Z \arrow[r] & 0
\end{tikzcd} \] be a short exact sequence in $\A$. If $\domdim X \geqslant m$ and $\domdim Y \geqslant m - 1$, then $\domdim Z \geqslant m - 1$.
\end{lemma}

\begin{proof}
The proof of \cite[Lemma 4.19]{MR4392222} carries over.
\end{proof}

\begin{proposition}\label{proposition.43}
Let $\A$ be an abelian category with enough projectives. The following are equivalent:
\begin{enumerate}
    \item[(a)] $\domdim \A \geqslant n + 1$.
    \item[(b)] Every $X \in \A$ with $\pdim X \leqslant m$ satisfies $\domdim X \geqslant n + 1 - m$, for all $0 \leqslant m \leqslant n$.
    \item[(c)] Every $X \in \A$ with $\pdim X \leqslant n$ satisfies $\domdim X \geqslant 1$.
\end{enumerate}
\end{proposition}

\begin{proof}
Let us prove that (a) implies (b). Suppose that $\domdim \A \geqslant n + 1$. Let $m$ be an integer with $0 \leqslant m \leqslant n$, and let $X \in \A$ be such that $\pdim X \leqslant m$. If $m = 0$, then $\domdim X \geqslant n + 1$, hence assume that $m \geqslant 1$. Given that $\A$ has enough projectives, there is an exact sequence \[ \begin{tikzcd}
0 \arrow[r] & P_{m} \arrow[r] & \cdots \arrow[r] & P_{1} \arrow[r] & P_{0} \arrow[r] & X \arrow[r] & 0
\end{tikzcd} \] in $\A$ with $P_{i}$ projective in $\A$ for each $0 \leqslant i \leqslant m$. By writing this sequence as the splice of short exact sequences and by applying Lemma \ref{lemma.7} successively, we conclude that $\domdim X \geqslant n + 1 - m$.

Now, trivially, (b) implies (c), and it is easy to see that (c) implies (a).
\end{proof}

\begin{lemma}\label{lemma.8}
Let $\A$ be an abelian category such that $\domdim \A \geqslant 1$, and let $X \in \A$. Then $X$ is a syzygy if and only if $\domdim X \geqslant 1$.
\end{lemma}

\begin{proof}
Suppose that $X$ is a syzygy. Then there is a monomorphism $X \to P$ in $\A$ with $P$ projective in $\A$. But as $\domdim \A \geqslant 1$, there is also a monomorphism $P \to Y$ in $\A$ with $Y$ projective and injective in $\A$. Since the composition of these morphisms gives a monomorphism $X \to Y$, we get that $\domdim X \geqslant 1$. The converse is trivial.
\end{proof}

We can now deduce the following:

\begin{proposition}\label{proposition.29}
Assume that $\C$ is right coherent and $\domdim (\mod \C) \geqslant 1$. Then $\C$ satisfies the axiom \textup{(F2$_{f}$)} if and only if $\domdim (\mod \C) \geqslant n + 1$.
\end{proposition}

\begin{proof}
Follows from Proposition \ref{proposition.43} and Lemma \ref{lemma.8}.
\end{proof}

By taking $\C^{\op}$ in place of $\C$ in Proposition \ref{proposition.29}, we conclude that if $\C$ is left coherent and $\domdim (\mod \C^{\op}) \geqslant 1$, then $\C$ satisfies the axiom (F2$_{f}^{\op}$) if and only if $\domdim (\mod \C^{\op}) \geqslant n + 1$. What is somewhat surprising is that the following also holds:

\begin{proposition}\label{proposition.30}
Assume that $\C$ is coherent. If $\domdim (\mod \C) \geqslant n + 1$, then $\C$ satisfies the axiom \textup{(F2$^{\op}$)}.\footnote{Actually, the proof we present shows that if $\C$ is right coherent and $\domdim (\mod \C) \geqslant n + 1$, then $\C$ satisfies the axiom (F2$_{h}^{\op}$), see Appendix \ref{section.10}.}
\end{proposition}

\begin{proof}
Suppose that $\domdim (\mod \C) \geqslant n + 1$. Let $F \in \mod \C^{\op}$ be such that $\pdim F \leqslant 1$, and take a projective resolution \[ \begin{tikzcd}
0 \arrow[r] & {\C(Y,-)} \arrow[r, "{\C(f,-)}"] &[1.2em] {\C(X,-)} \arrow[r] & F \arrow[r] & 0
\end{tikzcd} \] of $F$ in $\mod \C^{\op}$ with $f \in \C(X,Y)$. Then $\mr(f) \in \mod \C$ is a transpose of $F$ with the property that $\mr(f)^{\ast} = 0$. Hence $\Hom (\mr(f) , \C(-,W)) = 0$ for all $W \in \C$. Now, let $Z \in \C$ be arbitrary. Because $\domdim (\mod \C) \geqslant n + 1$, there is an exact sequence \[ \begin{tikzcd}
0 \arrow[r] & {\C(-,Z)} \arrow[r] & {\C(-,W_{1})} \arrow[r] & {\C(-,W_{2})} \arrow[r] & \cdots \arrow[r] & {\C(-,W_{n+1})}
\end{tikzcd} \] in $\mod \C$ with $W_{i} \in \C$ such that $\C(-,W_{i})$ is injective in $\mod \C$ for all $1 \leqslant i \leqslant n + 1$. This exact sequence induces a short exact sequence \[ \begin{tikzcd}
0 \arrow[r] & G_{i-1} \arrow[r] & {\C(-,W_{i})} \arrow[r] & G_{i} \arrow[r] & 0
\end{tikzcd} \] in $\mod \C$ for each $1 \leqslant i \leqslant n + 1$, where $G_{0} = \C(-,Z)$. Thus, there is a long exact sequence \[ \hspace{-3em} \begin{tikzcd}
0 \arrow[r] & \Hom(\mr(f),G_{i-1}) \arrow[r] & \Hom(\mr(f),\C(-,W_{i})) \arrow[r] & \Hom(\mr(f),G_{i}) \arrow[out=0, in=180, looseness=1.6, overlay, lld]    \\[0.4em]
            & \Ext^{1}(\mr(f),G_{i-1}) \arrow[r]             & \Ext^{1}(\mr(f),\C(-,W_{i})) \arrow[r]      & \Ext^{1}(\mr(f),G_{i}) \arrow[out=0, in=180, looseness=1.6, overlay, lld] \\[0.4em]
            & \Ext^{2}(\mr(f),G_{i-1}) \arrow[r, shorten >= 3.99em,shorten <= 0em]                      & \cdots               & 
\end{tikzcd} \] in $\Ab$ for each $1 \leqslant i \leqslant n + 1$. Therefore, given that $\Hom(\mr(f),\C(-,W_{i})) = 0$ and $\Ext^{j}(\mr(f),\C(-,W_{i})) = 0$ for all $j \geqslant 1$ and all $1 \leqslant i \leqslant n + 1$, we conclude that \[ \Ext^{i}(\mr(f), \C(-,Z)) \simeq \Ext^{1}(\mr(f), G_{i-1}) \simeq \Hom(\mr(f), G_{i}) = 0 \] for each $1 \leqslant i \leqslant n$. Hence $F$ is $n$-torsion free, so that $\C$ satisfies the axiom (F2$^{\op}$).
\end{proof}

Dually, by Proposition \ref{proposition.30}, if $\C$ is coherent and $\domdim (\mod \C^{\op}) \geqslant n + 1$, then $\C$ satisfies the axiom (F2). Also, by combining the last two propositions, we obtain the following result:

\begin{proposition}\label{proposition.31}
Assume that $\C$ is coherent. If $\domdim (\mod \C) \geqslant 1$ and \linebreak $\domdim (\mod \C^{\op}) \geqslant 1$, then the following are equivalent:
\begin{enumerate}
    \item[(a)] $\C$ satisfies the axiom \textup{(F2)}.
    \item[(b)] $\domdim (\mod \C) \geqslant n + 1$.
    \item[(c)] $\C$ satisfies the axiom \textup{(F2$^{\op}$)}.
    \item[(d)] $\domdim (\mod \C^{\op}) \geqslant n + 1$.
\end{enumerate}
\end{proposition}

\begin{proof}
Follows from Theorem \ref{theorem.9} and Propositions \ref{proposition.29} and \ref{proposition.30}.
\end{proof}

\begin{corollary}\label{corollary.7}
Assume that $\C$ is coherent. If $\domdim (\mod \C) \geqslant 1$ and \linebreak $\domdim (\mod \C^{\op}) \geqslant 1$, then $\domdim (\mod \C) = \domdim (\mod \C^{\op})$.
\end{corollary}

\begin{proof}
Follows from Proposition \ref{proposition.31}.
\end{proof}

We remark that Corollary \ref{corollary.7} for the case of modules over a right and left artinian ring was already known, see the paragraph proceeding Corollary \ref{corollary.2}.

We can now state and prove the main result of this section.

\begin{theorem}\label{theorem.12}
Let $\C$ be an additive and idempotent complete category. Then $\C$ is $n$-abelian and has enough injectives if and only if $\C$ is coherent and $\gldim (\mod \C) \leqslant n + 1 \leqslant \domdim (\mod \C)$. Moreover, in this case, $\mod \C$ has enough injectives.
\end{theorem}

\begin{proof}
Suppose that $\C$ is $n$-abelian and has enough injectives. Then it follows from Theorem \ref{theorem.2} that $\C$ is coherent, $\gldim (\mod \C) \leqslant n + 1$ and $\C$ satisfies the axiom (F2), so that it also satisfies the axiom (F2$_{f}$), by Theorem \ref{theorem.9}. Furthermore, by Proposition \ref{proposition.28}, $\domdim (\mod \C) \geqslant 1$. Thus, we conclude from Proposition \ref{proposition.29} that $\domdim (\mod \C) \geqslant n + 1$.

Conversely, assume that $\C$ is coherent and $\gldim (\mod \C) \leqslant n + 1 \leqslant \domdim (\mod \C)$. From Theorem \ref{theorem.5}, we see that $\gldim (\mod \C^{\op}) \leqslant n + 1$. Hence $\C$ satisfies the axioms (F1) and (F1$^{\op}$). Moreover, it follows from Proposition \ref{proposition.29} and Theorem \ref{theorem.9} that $\C$ satisfies the axiom (F2), and Proposition \ref{proposition.30} says that $\C$ satisfies the axiom (F2$^{\op}$). Consequently, $\C$ is $n$-abelian, by Theorem \ref{theorem.2}, and it follows from Proposition \ref{proposition.28} that $\C$ has enough injectives.

If $\C$ satisfies the above assumptions, then a similar argument as the one in the proof of \cite[Proposition 4.20]{MR4392222} shows that $\mod \C$ has enough injectives.
\end{proof}

By duality, we deduce from Theorem \ref{theorem.12} that $\C$ is $n$-abelian and has enough projectives if and only if $\C$ is coherent and $\gldim (\mod \C^{\op}) \leqslant n + 1 \leqslant \domdim (\mod \C^{\op})$. Furthermore, in this case, $\mod \C^{\op}$ has enough injectives. Also, although there are inequalities for the global and dominant dimensions of $\mod \C$ and $\mod \C^{\op}$ when $\C$ is $n$-abelian and has enough injectives and enough projectives, respectively, there are only two possible cases. In fact, assume that $\C$ is $n$-abelian and has enough injectives, so that it satisfies the conditions of Theorem \ref{theorem.12}. If $\C$ is von Neumann regular, then it follows from Proposition \ref{proposition.4} that $\gldim (\mod \C) = 0$, which implies that $\domdim (\mod \C) = \infty$. On the other hand, if $\C$ is not von Neumann regular, then $\gldim (\mod \C) = n + 1$, by Corollary \ref{corollary.1}, and we can also verify that $\domdim (\mod \C) = n + 1$.\footnote{Indeed, we already know that $\domdim (\mod \C) \geqslant n + 1$, and if this inequality would be strict, then it would follow from Proposition \ref{proposition.30} that every $F \in \mod \C^{\op}$ with $\pdim F \leqslant 1$ is $(n+1)$-torsion free. But then, as $\gldim (\mod \C) = n + 1$, it would follow from Lemma \ref{lemma.1} that $\Tr F$ is projective for every $F \in \mod \C^{\op}$ with $\pdim F \leqslant 1$, which, by Lemma \ref{lemma.2}, would imply that $F$ is projective. Hence we would conclude that there is no $F \in \mod \C^{\op}$ with $\pdim F = 1$, which is a contradiction since we know from Theorem \ref{theorem.5} that $\gldim (\mod \C^{\op}) = n + 1$.} Dually, if $\C$ is $n$-abelian and has enough projectives, then either $\gldim (\mod \C^{\op}) = 0$ and $\domdim (\mod \C^{\op}) = \infty$ or both these dimensions are equal to $n + 1$.

Finally, we remark that Theorem \ref{theorem.12} can be seen as a general case of a result due to Iyama and Jasso, which says that if $\C$ is a ``dualizing $R$-variety'', then $\C$ is $n$-abelian if and only if $\gldim (\mod \C) \leqslant n + 1 \leqslant \domdim (\mod \C)$, see \cite[Theorem 1.2]{MR3638352}. The reader is also encouraged to compare Theorem \ref{theorem.12} with \cite[Theorems 8.23 and 9.6]{MR3406183}, \cite[Theorem 1.5]{MR4392222} and \cite[Theorem 5.2]{EbrahimiNasr-Isfahani}.

\section{Categories with additive generators}\label{section.8}

In this section, we specialize previous results to rings and modules over rings. This is done by considering additive and idempotent complete categories with additive generators. By doing so, we obtain a description of when the category of finitely generated projective modules over a ring is $n$-abelian. We also show that there is a correspondence between $n$-abelian categories with additive generators and rings under such a description, which extends the higher Auslander correspondence to rings that are not necessarily Artin algebras.

Throughout this paper, we have been following the philosophy that a module over an additive and idempotent complete category is analogous to a module over a ring. Moreover, we could have even considered the more general case of a module over a preadditive category, so that the case of a module over a ring would be recovered by considering a preadditive category with a single object, which is nothing but a ring. Thus, a ring is a particular case of the much more general notion of a preadditive category, which can be thought of as a ``ring with several objects'', as supported by Mitchell in \cite{MR0294454}. In view of these remarks, while working with modules over an additive and idempotent complete category, it is natural to ask when it is the case that we are actually dealing with modules over a ring. Here is an answer to this question:

\begin{proposition}\label{proposition.15}
Let $\C$ be an additive and idempotent complete category. The category $\mod \C$ is equivalent to $\mod \Lambda$ for some ring $\Lambda$ if and only if $\C$ has an additive generator.
\end{proposition}

\begin{proof}
If there is a ring $\Lambda$ and an equivalence of categories $\mod \C \approx \mod \Lambda$, then it induces an equivalence on the subcategories of projective objects. Thus, there is an equivalence $\proj \C \approx \proj \Lambda$. But we know that the Yoneda embedding induces an equivalence $\C \approx \proj \C$, hence we obtain that $\C \approx \proj \Lambda$. Since $\proj \Lambda = \add \Lambda$, that is, $\Lambda$ is an additive generator of $\proj \Lambda$, we deduce that $\C$ has an additive generator.

Conversely, suppose that $\C$ has an additive generator $X \in \C$, and let $\Lambda = \End(X)$ be the endomorphism ring of $X$. Then $\C$ is skeletally small, and there is a functor $\eval_{X} : \Mod \C \to \Mod \Lambda$ given by $\eval_{X}(F) = F(X)$ for each $F \in \Mod \C$ and $\eval_{X}(\alpha) = \alpha_{X}$ for each morphism $\alpha$ in $\Mod \C$, which is an equivalence of categories, see \mbox{\cite[Proposition 2.8.2]{VitorGulisz}.} Therefore, $\eval_{X}$ induces an equivalence $\mod \C \to \mod \Lambda$, see \cite[Proposition 3.1.1]{VitorGulisz}.
\end{proof}

The functor $\eval_{X} : \Mod \C \to \Mod \Lambda$ that was mentioned in the proof of Proposition \ref{proposition.15} is called the \textit{evaluation functor at $X$}, and it can be defined for any object $X \in \C$. As we remarked above, when $X$ is an additive generator of $\C$, the functor $\eval_{X}$ is an equivalence of categories, which induces an equivalence $\mod \C \to \mod \Lambda$. In this case, it also induces an equivalence $\proj \C \to \proj \Lambda$, and by composing it with the equivalence $\C \to \proj \C$ induced by the Yoneda embedding, we obtain an equivalence of categories $\prov_{X} : \C \to \proj \Lambda$. Clearly, the functor $\prov_{X}$ is given by $\prov_{X}(Y) = \C(X,Y)$ for each $Y \in \C$ and $\prov_{X}(f) = \C(X,f)$ for each morphism $f$ in $\C$. We call $\prov_{X}$ the \textit{projectivization functor at $X$}. Moreover, observe that if $X$ is an additive generator of $\C$, then it is also an additive generator of $\C^{\op}$, and the endomorphism ring of $X$ in $\C^{\op}$ is given by $\Lambda^{\op}$. Therefore, we also get equivalences of categories $\Mod \C^{\op} \to \Mod \Lambda^{\op}$, $\mod \C^{\op} \to \mod \Lambda^{\op}$ and $\C^{\op} \to \proj \Lambda^{\op}$. To summarize, if $\C$ has an additive generator $X$, then we can assign the ring $\Lambda = \End(X)$ to $\C$, modules over $\C$ are essentially modules over $\Lambda$, and there are equivalences of categories $\C \approx \proj \Lambda$ and $\C^{\op} \approx \proj \Lambda^{\op}$.

Now, going in the opposite direction of the above paragraph, note that if $\Lambda$ is an arbitrary ring, then $\proj \Lambda$ is an additive and idempotent complete category which has $\Lambda$ as an additive generator. Therefore, since $\End(\Lambda) \simeq \Lambda$, the evaluation functor at $\Lambda$ gives equivalences of categories $\Mod (\proj \Lambda) \to \Mod \Lambda$ and $\mod (\proj \Lambda) \to \mod \Lambda$ as well as equivalences $\Mod (\proj \Lambda)^{\op} \to \Mod \Lambda^{\op}$ and $\mod (\proj \Lambda)^{\op} \to \mod \Lambda^{\op}$. Thus, given any ring $\Lambda$, we can assign the category $\proj \Lambda$ to it, and modules over $\Lambda$ are essentially the same as modules over $\proj \Lambda$.

The moral is that every additive and idempotent complete category with an additive generator is (equivalent to) the category of finitely generated projective modules over a ring, and conversely. Moreover, the assignments of a category to a ring and of a ring to a category that were presented above are consistent in the sense that they preserve the corresponding categories of modules. Consequently, these assignments are inverses of each other, up to equivalences. Let us make this statement precise.

\begin{lemma}\label{lemma.3}
There is a bijective correspondence between the equivalence classes of additive and idempotent complete categories with additive generators and the Morita equivalence classes of rings.\footnote{Here, the equivalence class of a category is the class of categories that are equivalent to it.} The correspondence is given as follows:
\begin{enumerate}
    \item[(a)] If $\C$ is an additive and idempotent complete category with an additive generator, then send it to $\End(X)$, where $X$ is an additive generator of $\C$.
    \item[(b)] If $\Lambda$ is a ring, then send it to $\proj \Lambda$.
\end{enumerate}
\end{lemma}

\begin{proof}
Follows easily from the previous paragraphs and Proposition \ref{proposition.16}. But, for the sake of clarity, let us fill in the details below.

To begin with, let us show that the assignments are well defined. Let $\B$ and $\C$ be additive and idempotent complete categories with additive generators, say, $Y \in \B$ and $X \in \C$, and suppose that $\B$ and $\C$ are equivalent. Then there are equivalences \[ \Mod \End(Y) \approx \Mod \B \approx \Mod \C \approx \Mod \End(X), \] so that $\End(Y)$ and $\End(X)$ are Morita equivalent.\footnote{This argument also shows the independence of choice of additive generators.} Next, let $\Lambda$ and $\Gamma$ be rings that are Morita equivalent. Because there are equivalences \[ \Mod (\proj \Lambda) \approx \Mod \Lambda \approx \Mod \Gamma \approx \Mod (\proj \Gamma), \] it follows from Proposition \ref{proposition.16} that $\proj \Lambda$ and $\proj \Gamma$ are equivalent.

Now, we verify that the assignments are inverses of each other. Given an additive and idempotent complete category $\C$ with an additive generator, say $X \in \C$, we have \[ \Mod \C \approx \Mod \End(X) \approx \Mod (\proj \End(X)), \] and by Proposition \ref{proposition.16}, we get that $\C$ and $\proj \End(X)$ are equivalent. Finally, if $\Lambda$ is a ring and $P$ is an additive generator of $\proj \Lambda$, then \[ \Mod \Lambda \approx \Mod (\proj \Lambda) \approx \Mod \End(P), \] hence $\Lambda$ and $\End(P)$ are Morita equivalent.\footnote{Of course, we could have just taken $P = \Lambda$ and argued that $\End(\Lambda) \simeq \Lambda$.}
\end{proof}

For the rest of this section, let $\Lambda$ be a ring.

As it was previously mentioned, the evaluation functor at $\Lambda$ induces equivalences of categories $\Mod (\proj \Lambda) \to \Mod \Lambda$ and $\mod (\proj \Lambda) \to \mod \Lambda$ as well as equivalences $\Mod (\proj \Lambda)^{\op} \to \Mod \Lambda^{\op}$ and $\mod (\proj \Lambda)^{\op} \to \mod \Lambda^{\op}$. Therefore, we can use these equivalences to specialize previous definitions and results for $\C$ and for modules over $\C$ to $\Lambda$ and modules over $\Lambda$, respectively, by taking $\C = \proj \Lambda$. By doing so, we recover classical definitions and results for rings and modules over rings. To help the reader feel more comfortable with this idea, let us mention a few instances of this procedure.

We can define a ring $\Lambda$ to be \textit{right coherent}, \textit{left coherent} and \textit{coherent} if $\proj \Lambda$ is right coherent, left coherent and coherent, respectively. Due to the equivalences $\mod (\proj \Lambda) \approx \mod \Lambda$ and $\mod (\proj \Lambda)^{\op} \approx \mod \Lambda^{\op}$, we see that $\Lambda$ is right coherent and left coherent precisely when $\mod \Lambda$ is abelian and when $\mod \Lambda^{\op}$ is abelian, respectively. Hence our definitions coincide with the classical ones in ring theory. Moreover, because of the previous equivalences, it follows from Theorem \ref{theorem.5} that if $\Lambda$ is a coherent ring, then $\gldim (\mod \Lambda) = \gldim (\mod \Lambda^{\op})$. Also, we can use Corollary \ref{corollary.4} to conclude that, when $\Lambda$ is coherent, the global dimensions of $\mod \Lambda$ and $\mod \Lambda^{\op}$ coincide with the \textit{weak dimension} of $\Lambda$. These are, of course, well known results, see \cite[Theorem 5.63]{MR1653294} and \cite[Proposition 1.1]{MR306265}.

Similarly, we can define a ring $\Lambda$ to be \textit{von Neumann regular} if $\proj \Lambda$ is von Neumann regular. From Proposition \ref{proposition.4} and the equivalence $\mod (\proj \Lambda) \approx \mod \Lambda$, we deduce that $\Lambda$ is von Neumann regular if and only if every finitely presented $\Lambda$-module is projective. As it was already mentioned in Section \ref{section.4}, this latter condition is equivalent to saying that $\Lambda$ is von Neumann regular in the usual sense.\footnote{Sometimes, however, the nomenclature for rings and categories may not match. For example, $\Lambda$ is a \textit{semiperfect} ring if and only if $\proj \Lambda$ is a \textit{Krull--Schmidt} category, see \cite[Proposition 4.1]{MR3431480}.}

We also say that a ring $\Lambda$ is \textit{right comprehensive}, \textit{left comprehensive} and \textit{comprehensive} when $\proj \Lambda$ is right comprehensive, left comprehensive and comprehensive, respectively. Note that, because of the equivalence $\mod (\proj \Lambda) \approx \mod \Lambda$, we can conclude from Proposition \ref{proposition.24} that the inclusion functor $\proj \Lambda \to \mod \Lambda$ always reflects injectivity, and $\Lambda$ is right comprehensive precisely when it also preserves injectivity. Dually, due to the equivalence $\mod (\proj \Lambda)^{\op} \approx \mod \Lambda^{\op}$, the inclusion $\proj \Lambda^{\op} \to \mod \Lambda^{\op}$ reflects injectivity, and $\Lambda$ is left comprehensive precisely when injectivity is also preserved.\footnote{The author is not aware of any alternative characterization of such rings, though.}

Now, by considering the equivalence $\mod (\proj \Lambda) \approx \mod \Lambda$, given $M \in \mod \Lambda$, let $F \in \mod (\proj \Lambda)$ be such that $F(\Lambda) \simeq M$. We define a \textit{transpose} of $M$, denoted by $\Tr M$, to be a $\Lambda^{\op}$-module which is obtained by evaluating $\Tr F$ at $\Lambda$, where $\Tr F$ is a transpose of $F$. It is not hard to see that this definition indeed coincides with the classical definition of a transpose of a module over a ring, introduced in \cite{MR0269685}. Furthermore, if $\Lambda$ is left coherent and $k$ is a positive integer, then we can define $M$ to be \textit{$k$-torsion free} if $F$ is $k$-torsion free. Observe that if $\Lambda$ is left coherent, then, given a positive integer $i$, the $(\proj \Lambda)$-module $\E^{i}(F)$ corresponds to the $\Lambda$-module $\E^{i}(F)(\Lambda) = \Ext^{i}(\Tr F, \Hom_{\Lambda}(\Lambda,-))$ via the equivalence $\Mod (\proj \Lambda) \approx \Mod \Lambda$. But the latter is isomorphic to $\Ext^{i}(\Tr F(\Lambda), \Lambda)$, due to the equivalence $\mod (\proj \Lambda)^{\op} \approx \mod \Lambda^{\op}$. Therefore, $M$ is $k$-torsion free if and only if $\Ext^{i}(\Tr M, \Lambda) = 0$ for all $1 \leqslant i \leqslant k$. Hence we recover the classical concept of a $k$-torsion free module, as given in \cite{MR0269685}. By following similar arguments, we can also deduce that if $\Lambda$ is coherent, then the double dual sequence of $F$ induces an exact sequence \[ \begin{tikzcd}
0 \arrow[r] & {\Ext^{1}(\Tr M, \Lambda)} \arrow[r] & M \arrow[r] & M^{\ast \ast} \arrow[r] & {\Ext^{2}(\Tr M, \Lambda)} \arrow[r] & 0
\end{tikzcd} \] in $\mod \Lambda$, where $M \to M^{\ast \ast}$ is the canonical morphism from a module to its double dual. We call the above sequence the \textit{double dual sequence} of $M$.

Hopefully, the reader is now comfortable with the idea that we can specialize previous definitions and results to rings and modules over rings. Let us explore some consequences of this procedure, which leads to results concerning the category $\proj \Lambda$ of finitely generated projective $\Lambda$-modules, for a given ring $\Lambda$.

For example, the next result is Proposition \ref{proposition.1} specialized to rings.

\begin{proposition}\label{proposition.38}
Let $\Lambda$ be a ring.
\begin{enumerate}
    \item[(a)] The category $\proj \Lambda$ has weak kernels if and only if $\Lambda$ is right coherent.
    \item[(b)] The category $\proj \Lambda$ has weak cokernels if and only if $\Lambda$ is left coherent.
\end{enumerate}
\end{proposition}

\begin{proof}
Follows from Proposition \ref{proposition.1} and the equivalences of categories $\mod (\proj \Lambda) \approx \mod \Lambda$ and $\mod (\proj \Lambda)^{\op} \approx \mod \Lambda^{\op}$.
\end{proof}

We can also specialize the functorial axioms of an $n$-abelian category to axioms of a ring. Indeed, for a ring $\Lambda$, the axioms (F1) and (F1$^{\op}$) become:
\begin{enumerate}
    \item[(R1)] $\Lambda$ is right coherent and $\gldim (\mod \Lambda) \leqslant n+1$.
    \item[(R1$^{\op}$)] $\Lambda$ is left coherent and $\gldim (\mod \Lambda^{\op}) \leqslant n+1$.
\end{enumerate}
Alternatively, by \cite[Proposition 1.1]{MR306265}, these axioms could be written as:
\begin{enumerate}
    \item[(R1)] $\Lambda$ is right coherent and has weak dimension at most $n + 1$.
    \item[(R1$^{\op}$)] $\Lambda$ is left coherent and has weak dimension at most $n + 1$.
\end{enumerate}
And, under the assumption that $\Lambda$ is coherent, the axioms (F2) and (F2$^{\op}$) for rings become:
\begin{enumerate}
    \item[(R2)] Every $M \in \mod \Lambda$ with $\pdim M \leqslant 1$ is $n$-torsion free.
    \item[(R2$^{\op}$)] Every $M \in \mod \Lambda^{\op}$ with $\pdim M \leqslant 1$ is $n$-torsion free.
\end{enumerate}

Precisely, we have that $\proj \Lambda$ satisfies (F1) if and only if $\Lambda$ satisfies (R1), and the same statement also holds for the dual axioms. Moreover, if $\Lambda$ is coherent, then $\proj \Lambda$ satisfies (F2) if and only if $\Lambda$ satisfies (R2), and similarly for the dual axioms.

We can then conclude that the above axioms for $\Lambda$ are equivalent to the axioms (A1), (A1$^{\op}$), (A2) and (A2$^{\op}$) for $\proj \Lambda$. In fact, we have the following:

\begin{proposition}\label{proposition.8}
Let $\Lambda$ be a ring.
\begin{enumerate}
    \item[(a)] $\proj \Lambda$ satisfies the axiom \textup{(A1)} if and only if $\Lambda$ satisfies the axiom \textup{(R1)}.
    \item[(b)] $\proj \Lambda$ satisfies the axiom \textup{(A1$^{\op}$)} if and only if $\Lambda$ satisfies the axiom \textup{(R1$^{\op}$)}.
\end{enumerate}
Moreover, if $\Lambda$ satisfies both the axioms \textup{(R1)} and \textup{(R1$^{\op}$)}, then the following statements also hold:
\begin{enumerate}
    \item[(c)] $\proj \Lambda$ satisfies the axiom \textup{(A2)} if and only if $\Lambda$ satisfies the axiom \textup{(R2)}.
    \item[(d)] $\proj \Lambda$ satisfies the axiom \textup{(A2$^{\op}$)} if and only if $\Lambda$ satisfies the axiom \textup{(R2$^{\op}$)}.
\end{enumerate}
\end{proposition}

\begin{proof}
Follows from Propositions \ref{proposition.2} and \ref{proposition.3}, and from the equivalences of categories $\mod (\proj \Lambda) \approx \mod \Lambda$ and $\mod (\proj \Lambda)^{\op} \approx \mod \Lambda^{\op}$.
\end{proof}

As a consequence, we can describe when the category of finitely generated projective modules over a ring is $n$-abelian. The next result is Theorem \ref{theorem.2} specialized to rings.

\begin{theorem}\label{theorem.4}
Let $\Lambda$ be a ring, and let $n$ be a positive integer. The category $\proj \Lambda$ is $n$-abelian if and only if $\Lambda$ satisfies the following axioms: \begin{enumerate}
    \item[(R1)] $\Lambda$ is right coherent and $\gldim (\mod \Lambda) \leqslant n+1$.
    \item[(R1$^{\op}$)] $\Lambda$ is left coherent and $\gldim (\mod \Lambda^{\op}) \leqslant n+1$.
    \item[(R2)] Every $M \in \mod \Lambda$ with $\pdim M \leqslant 1$ is $n$-torsion free.
    \item[(R2$^{\op}$)] Every $M \in \mod \Lambda^{\op}$ with $\pdim M \leqslant 1$ is $n$-torsion free.
\end{enumerate}
\end{theorem}

\begin{proof}
Follows from Proposition \ref{proposition.8}.
\end{proof}

It is worth mentioning that, by Theorem \ref{theorem.4} and Corollary \ref{corollary.1}, if $\Lambda$ satisfies the axioms \textup{(R1)}, \textup{(R1$^{\op}$)}, \textup{(R2)} and \textup{(R2$^{\op}$)}, and if $\Lambda$ is not von Neumann regular, then \[ \gldim (\mod \Lambda) = \gldim (\mod \Lambda^{\op}) = n + 1, \] and this number coincides with the weak dimension of $\Lambda$, as we remarked earlier.

We can also specialize the axioms (F2$_{a}$), (F2$_{b}$), (F2$_{c}$), (F2$_{d_{k}}$), (F2$_{e}$), (F2$_{f}$), (F2$_{g_{k}}$) and their duals to axioms for rings, and specialize other previous results to rings and modules over rings. For example, consider the following axioms for a coherent ring $\Lambda$:

\begin{enumerate}
    \item[(R2$_{c}$)] Every $m$-spherical object in $\mod \Lambda$ is a syzygy, for all $1 \leqslant m \leqslant n$.
    \item[(R2$_{f}$)] Every $M \in \mod \Lambda$ with $\pdim M \leqslant n$ is a syzygy.
\end{enumerate}

And their duals:

\begin{enumerate}
    \item[(R2$_{c}^{\op}$)] Every $m$-spherical object in $\mod \Lambda^{\op}$ is a syzygy, for all $1 \leqslant m \leqslant n$.
    \item[(R2$_{f}^{\op}$)] Every $M \in \mod \Lambda^{\op}$ with $\pdim M \leqslant n$ is a syzygy.
\end{enumerate}

Then we can conclude the following:

\begin{proposition}\label{proposition.36}
Let $\Lambda$ be a coherent ring. The following statement holds for $\square \in \{c,f\}$:
\begin{enumerate}
    \item[] \hspace{-2em} The axioms \textup{(R2)} and \textup{(R2$^{\op}$)} are equivalent to \textup{(R2$_{\square}$)} and \textup{(R2$_{\square}^{\op}$)}, respectively.
\end{enumerate}
\end{proposition}

\begin{proof}
Follows from Theorems \ref{theorem.7} and \ref{theorem.9}, and from the equivalences of categories $\mod (\proj \Lambda) \approx \mod \Lambda$ and $\mod (\proj \Lambda)^{\op} \approx \mod \Lambda^{\op}$.
\end{proof}

We leave it to the reader the task of formulating the other axioms for rings, and also of specializing the results of Section \ref{section.6} to rings. Next, we will consider some of the results of Section \ref{section.7} in the ring case. Since these involve dominant dimensions, it is worth stating a few remarks first.

To begin with, recall from \cite{MR258888} that a $\Lambda$-module $M$ is called \textit{FP-injective} when $\Ext_{\Lambda}^{1}(N,M) = 0$ for every finitely presented $\Lambda$-module $N$. We note that FP-injective $\Lambda$-modules are also called ``absolutely pure'' by some authors, see \cite{MR224649} and \cite{MR294409}. It is easy to see that if $\Lambda$ is right coherent, then an object in $\mod \Lambda$ is injective in $\mod \Lambda$ if and only if it is FP-injective. Consequently, if $\Lambda$ is right coherent and $m$ is a positive integer, then $\domdim (\mod \Lambda) \geqslant m$ if and only if for every finitely generated projective right $\Lambda$-module $P$, there is an exact sequence of right $\Lambda$-modules \[ \begin{tikzcd}
0 \arrow[r] & P \arrow[r] & V_{1} \arrow[r] & \cdots \arrow[r] & V_{m}
\end{tikzcd} \] with each $V_{i}$ being finitely generated, projective and FP-injective. The same statement for left $\Lambda$-modules characterizes $\domdim (\mod \Lambda^{\op}) \geqslant m$ when $\Lambda$ is left coherent.

Now, as in \cite{MR233854}, we let the \textit{right dominant dimension} of $\Lambda$ be the dominant dimension of $\Lambda$ viewed as an object in the category $\Mod \Lambda$, and we denote it by $\rdomdim \Lambda$. By \cite[Lemma 1]{MR233854}, for a positive integer $m$, we have $\rdomdim \Lambda \geqslant m$ if and only if there is an exact sequence of right $\Lambda$-modules \[ \begin{tikzcd}
0 \arrow[r] & \Lambda \arrow[r] & Q_{1} \arrow[r] & \cdots \arrow[r] & Q_{m}
\end{tikzcd} \] with each $Q_{i}$ being finitely generated, projective and injective. Similarly, we let the \textit{left dominant dimension} of $\Lambda$ be the dominant dimension of $\Lambda$ viewed as an object of $\Mod \Lambda^{\op}$, and we denote it by $\ldomdim \Lambda$. The above statement for left $\Lambda$-modules characterizes when $\ldomdim \Lambda \geqslant m$.

Observe that every injective $\Lambda$-module is FP-injective. Moreover, we can use Baer's criterion to conclude that if $\Lambda$ is right noetherian, then every FP-injective right $\Lambda$-module is injective. The converse also holds, that is, if every FP-injective right $\Lambda$-module is injective, then $\Lambda$ is right noetherian, see \cite[page 157]{MR224649} or \cite[Theorem 3]{MR294409}. Thus, $\Lambda$ is right noetherian precisely when the classes of injective and of FP-injective right $\Lambda$-modules coincide. Therefore, we can conclude from this and the previous paragraphs and \cite[Lemma 1]{MR224656} that if $\Lambda$ is right noetherian, then $\domdim (\mod \Lambda) = \rdomdim \Lambda$. Of course, similar statements hold for left modules, so that if $\Lambda$ is left noetherian, then $\domdim (\mod \Lambda^{\op}) = \ldomdim \Lambda$.

We now consider some of the results of Section \ref{section.7} in the ring case. We start with the following:

\begin{proposition}\label{proposition.42}
Let $\Lambda$ be a ring.
\begin{enumerate}
    \item[(a)] If $\Lambda$ is right coherent, then $\domdim (\mod \Lambda) \geqslant 1$ if and only if $\Lambda$ is right comprehensive and $\proj \Lambda$ has enough injectives.
    \item[(b)] If $\Lambda$ is left coherent, then $\domdim (\mod \Lambda^{\op}) \geqslant 1$ if and only if $\Lambda$ is left comprehensive and $\proj \Lambda$ has enough projectives.
\end{enumerate}
\end{proposition}

\begin{proof}
Follows from Corollary \ref{corollary.6} and the equivalences of categories $\mod (\proj \Lambda) \approx \mod \Lambda$ and $\mod (\proj \Lambda)^{\op} \approx \mod \Lambda^{\op}$.
\end{proof}

Note that, by our previous remarks, when $\Lambda$ is right noetherian, the condition given in item (a) of Proposition \ref{proposition.42} is equivalent to $\rdomdim \Lambda \geqslant 1$. Also, when $\Lambda$ is right artinian, this condition is equivalent to $\Lambda$ being ``right QF-3'', see \cite[Theorems 3.1 and 3.2]{MR112904} or \mbox{\cite[Theorem 1.3]{MR269686}.} Moreover, the condition in item (b) of Proposition \ref{proposition.42} is equivalent to $\ldomdim \Lambda \geqslant 1$ when $\Lambda$ is left noetherian, and it is also equivalent to $\Lambda$ being ``left QF-3'' when $\Lambda$ is left artinian. Similar remarks apply to the next proposition.

\begin{proposition}\label{proposition.39}
Let $\Lambda$ be a ring, and assume that $\proj \Lambda$ is an $n$-abelian category.
\begin{enumerate}
    \item[(a)] $\domdim (\mod \Lambda) \geqslant 1$ if and only if $\proj \Lambda$ has enough injectives.
    \item[(b)] $\domdim (\mod \Lambda^{\op}) \geqslant 1$ if and only if $\proj \Lambda$ has enough projectives.
\end{enumerate}
\end{proposition}

\begin{proof}
Since $\proj \Lambda$ is $n$-abelian, $\Lambda$ is coherent and comprehensive, by Propositions \ref{proposition.38} and \ref{proposition.27}. Hence the result follows from Proposition \ref{proposition.42}.
\end{proof}

\begin{proposition}\label{proposition.34}
Let $\Lambda$ be a coherent ring. If $\domdim (\mod \Lambda) \geqslant 1$ and \linebreak $\domdim (\mod \Lambda^{\op}) \geqslant 1$, then the following are equivalent:
\begin{enumerate}
    \item[(a)] $\Lambda$ satisfies the axiom \textup{(R2)}.
    \item[(b)] $\domdim (\mod \Lambda) \geqslant n + 1$.
    \item[(c)] $\Lambda$ satisfies the axiom \textup{(R2$^{\op}$)}.
    \item[(d)] $\domdim (\mod \Lambda^{\op}) \geqslant n + 1$.
\end{enumerate}
\end{proposition}

\begin{proof}
Follows from Proposition \ref{proposition.31} and the equivalences of categories $\mod (\proj \Lambda) \approx \mod \Lambda$ and $\mod (\proj \Lambda)^{\op} \approx \mod \Lambda^{\op}$.
\end{proof}

\begin{corollary}\label{corollary.2}
Let $\Lambda$ be a coherent ring. If $\domdim (\mod \Lambda) \geqslant 1$ and \linebreak $\domdim (\mod \Lambda^{\op}) \geqslant 1$, then $\domdim (\mod \Lambda) = \domdim (\mod \Lambda^{\op})$.
\end{corollary}

\begin{proof}
Follows from Proposition \ref{proposition.34}.
\end{proof}

Observe that, by the discussion preceding Proposition \ref{proposition.42}, when $\Lambda$ is right and left noetherian, Corollary \ref{corollary.2} states that if $\rdomdim \Lambda \geqslant 1$ and $\ldomdim \Lambda \geqslant 1$, then $\rdomdim \Lambda = \ldomdim \Lambda$. A similar result was proved by M{\"u}ller in \cite[Theorem 10]{MR233854}, namely, that if $\Lambda$ is semiprimary, $\rdomdim \Lambda \geqslant 1$ and $\ldomdim \Lambda \geqslant 1$, then $\rdomdim \Lambda = \ldomdim \Lambda$, see also \cite[Theorem 7.7]{MR349740}. The intersection of these two results states that if $\Lambda$ is right and left artinian,\footnote{Recall that a ring is right artinian if and only if it is right noetherian and semiprimary, see \cite[Theorem 4.15]{MR1838439}.} $\rdomdim \Lambda \geqslant 1$ and $\ldomdim \Lambda \geqslant 1$, then $\rdomdim \Lambda = \ldomdim \Lambda$, which appears in \cite[Proposition 2.3]{MR340327}. However, it is worth mentioning that Harada proved in \cite[Corollary 1]{MR206049} that if $\Lambda$ is right and left artinian, then $\rdomdim \Lambda \geqslant 1$ if and only if $\ldomdim \Lambda \geqslant 1$. Therefore, we conclude that if $\Lambda$ is right and left artinian, then $\rdomdim \Lambda = \ldomdim \Lambda$. For the record, let us remark that this result was also proved for finite dimensional algebras in \cite[Theorem 4]{MR224656}, in which case it also follows from \cite[Corollary 2.2]{MR314906} or \mbox{\cite[Lemma 2.1]{MR340327}.}

\begin{theorem}\label{theorem.13}
Let $\Lambda$ be a ring, and let $n$ be a positive integer. The category $\proj \Lambda$ is $n$-abelian and has enough injectives if and only if $\Lambda$ is coherent and $\gldim (\mod \Lambda) \leqslant n + 1 \leqslant \domdim (\mod \Lambda)$. Moreover, in this case, $\mod \Lambda$ has enough injectives.
\end{theorem}

\begin{proof}
Follows from Theorem \ref{theorem.12} and the equivalence $\mod (\proj \Lambda) \approx \mod \Lambda$.
\end{proof}

Similarly, by Theorem \ref{theorem.12} and the equivalence $\mod (\proj \Lambda)^{\op} \approx \mod \Lambda^{\op}$, we get that $\proj \Lambda$ is $n$-abelian and has enough projectives if and only if $\Lambda$ is coherent and $\gldim (\mod \Lambda^{\op}) \leqslant n + 1 \leqslant \domdim (\mod \Lambda^{\op})$. Furthermore, in this case, $\mod \Lambda^{\op}$ has enough injectives.

Recall that an \textit{Artin algebra} is an associative $R$-algebra with identity which is finitely generated as an $R$-module, where $R$ is a fixed commutative artinian ring. The next result, due to Iyama and Jasso, gives a refinement of Theorem \ref{theorem.13} in the case of Artin algebras.

\begin{proposition}\label{proposition.35}
Let $\Lambda$ be an Artin algebra, and let $n$ be a positive integer. The category $\proj \Lambda$ is $n$-abelian if and only if $\gldim (\mod \Lambda) \leqslant n + 1 \leqslant \domdim (\mod \Lambda)$.
\end{proposition}

\begin{proof}
Since $\Lambda$ is an Artin algebra, $\proj \Lambda$ is a ``dualizing $R$-variety'', see \cite[Proposition 2.5]{MR342505} or \cite[Proposition 5.1.1]{VitorGulisz}. Therefore, by \cite[Theorem 1.2]{MR3638352}, $\proj \Lambda$ is $n$-abelian if and only if $\gldim (\mod (\proj \Lambda)) \leqslant n + 1 \leqslant \domdim (\mod (\proj \Lambda))$. Thus, the result follows from the equivalence of categories $\mod (\proj \Lambda) \approx \mod \Lambda$.
\end{proof}

Let us point out that if $\Lambda$ is an Artin algebra, then it is both right and left artinian, hence it follows from the paragraph proceeding Corollary \ref{corollary.2} that $\domdim (\mod \Lambda) = \domdim (\mod \Lambda^{\op})$. Thus, by Propositions \ref{proposition.39} and \ref{proposition.35}, if $\Lambda$ is an Artin algebra such that $\proj \Lambda$ is $n$-abelian, then $\proj \Lambda$ has enough injectives and enough projectives. We also recall that an Artin algebra $\Lambda$ is called an \textit{$n$-Auslander algebra} if the inequalities $\gldim (\mod \Lambda) \leqslant n + 1 \leqslant \domdim (\mod \Lambda)$ are satisfied. Therefore, by Proposition \ref{proposition.35}, an Artin algebra $\Lambda$ is an $n$-Auslander algebra if and only if $\proj \Lambda$ is an $n$-abelian category, which is the case if and only if $\Lambda$ satisfies the axioms (R1), (R1$^{\op}$), (R2) and (R2$^{\op}$), by Theorem \ref{theorem.4}. These facts offer an alternative perspective towards $n$-Auslander algebras, and support the idea that the class of rings satisfying the axioms (R1), (R1$^{\op}$), (R2) and (R2$^{\op}$) generalizes the class of $n$-Auslander algebras. Furthermore, observe that the condition on the dominant dimension of an $n$-Auslander algebra can be replaced by the axioms (R2) and (R2$^{\op}$), or by any of the other equivalent axioms covered in this paper, such as the axioms (R2$_{c}$) and (R2$_{c}^{\op}$), or (R2$_{f}$) and (R2$_{f}^{\op}$), as it was pointed out in Proposition \ref{proposition.36}.

Now, let us turn our attention to Lemma \ref{lemma.3}. The reader might have asked why we stated this result as a lemma and not as a proposition or a theorem. The reason is that the correspondence described in Lemma \ref{lemma.3}, in its full generality, is a source for several other correspondences which are obtained by restriction. For instance, by restricting it to $n$-abelian categories, we get the following correspondence:

\begin{theorem}\label{theorem.6}
There is a bijective correspondence between the equivalence classes of $n$-abelian categories with additive generators and the Morita equivalence classes of rings satisfying the axioms \textup{(R1)}, \textup{(R1$^{\op}$)}, \textup{(R2)} and \textup{(R2$^{\op}$)}. The correspondence is given as follows:
\begin{enumerate}
    \item[(a)] If $\C$ is an $n$-abelian category with an additive generator, then send it to $\End(X)$, where $X$ is an additive generator of $\C$.
    \item[(b)] If $\Lambda$ is a ring satisfying \textup{(R1)}, \textup{(R1$^{\op}$)}, \textup{(R2)} and \textup{(R2$^{\op}$)}, then send it to $\proj \Lambda$.
\end{enumerate}
\end{theorem}

\begin{proof}
Follows from Lemma \ref{lemma.3} and Theorem \ref{theorem.4}.
\end{proof}

Similarly, we can conclude from Lemma \ref{lemma.3} and Proposition \ref{proposition.8} that there is a correspondence between pre-$n$-abelian categories with additive generators and rings satisfying the axioms (R1) and (R1$^{\op}$), which are precisely the coherent rings with weak dimension at most $n + 1$. As another (ludic) example, by Lemma \ref{lemma.3}, Proposition \ref{proposition.38} and \cite[Proposition 4.1]{MR3431480}, there is a correspondence between categories with additive generators that have weak kernels and are Krull--Schmidt and rings that are right coherent and semiperfect. Some other examples are also given by Iyama in \cite[Corollary 3.5]{MR2095628} for rings that are Artin algebras. In general, any collection of properties for a category that are invariant under equivalence of categories leads to a correspondence, which is obtained by restricting the correspondence given in Lemma \ref{lemma.3}. Likewise, any collection of Morita invariant properties for a ring leads to a correspondence.

The celebrated Auslander correspondence and Morita--Tachikawa correspondence (and their higher versions) are also restrictions of the correspondence presented in Lemma \ref{lemma.3}. We refer the reader to \cite[Sections 5.2 and 5.3]{VitorGulisz} for details. Even more, the higher Auslander correspondence is actually a restriction of the correspondence described in Theorem \ref{theorem.6}. Indeed, the former is obtained by restricting the latter to ``$n$-cluster tilting subcategories'' (with additive generators) of categories of finitely presented modules over Artin algebras.\footnote{By \cite[Theorem 3.16]{MR3519980}, ``$n$-cluster tilting subcategories'' of abelian categories are $n$-abelian.} Or, from a different point of view, we can say that the higher Auslander correspondence is obtained by restricting the correspondence given in Theorem \ref{theorem.6} to Artin algebras.\footnote{In fact, as it was observed in the paragraph proceeding Proposition \ref{proposition.35}, an Artin algebra is an $n$-Auslander algebra if and only if it satisfies the axioms \textup{(R1)}, \textup{(R1$^{\op}$)}, \textup{(R2)} and \textup{(R2$^{\op}$)}.}

Observe that Theorem \ref{theorem.6} also indicates how to find examples of $n$-abelian categories with additive generators: it suffices to find rings satisfying the axioms (R1), (R1$^{\op}$), (R2) and (R2$^{\op}$). More than that, Theorem \ref{theorem.6} guarantees that all such examples of $n$-abelian categories can be obtained this way. Hence the problem of finding a ring $\Lambda$ satisfying the axioms (R1), (R1$^{\op}$), (R2) and (R2$^{\op}$) is of particular interest since it corresponds to an $n$-abelian category, namely, the category $\proj \Lambda$.

Clearly, von Neumann regular rings are examples of rings with the properties mentioned above, but their corresponding $n$-abelian categories are again von Neumann regular, and hence not interesting. On the other hand, $n$-Auslander algebras (which are not semisimple)\footnote{Recall that an Artin algebra is semisimple if and only if it is von Neumann regular. More generally, a ring is semisimple if and only if it is von Neumann regular and (left or right) noetherian.} give interesting examples of $n$-abelian categories, which have both enough injectives and enough projectives, as remarked in the paragraph subsequent to Proposition \ref{proposition.35}. However, to the best knowledge of the author, there are no known examples of rings satisfying the axioms (R1), (R1$^{\op}$), (R2) and (R2$^{\op}$) which are not von Neumann regular and not Artin algebras. The existence of such rings could possibly result in exotic examples of $n$-abelian categories, and bring new insights to the theory. Also, it would be even more interesting to have an example of such a ring $\Lambda$ with $\domdim (\mod \Lambda) = 0$ and $\domdim (\mod \Lambda^{\op}) = 0$, because then $\proj \Lambda$ would be an $n$-abelian category that does not have enough injectives nor enough projectives, by Proposition \ref{proposition.39}.

As a motivation for the use of Theorem \ref{theorem.6} as a source of possibly interesting examples of $n$-abelian categories, let us recall how Rump has used similar ideas to find a counterexample to Raikov’s conjecture. In \cite{MR2471947}, Rump characterized when a certain type of category $\C$ is semi-abelian or quasi-abelian in terms of $\mod \C$ and $\mod \C^{\op}$. Then, by specializing his results to modules over rings, he was able to find an example of a ring $\Lambda$ for which the category $\proj \Lambda$ is semi-abelian but not quasi-abelian, thereby providing a negative answer to Raikov’s conjecture. This achievement brings hope that the same approach could be used to answer questions concerning $n$-abelian categories.

Nevertheless, let us end this section by showing that commutative rings do not produce interesting examples of $n$-abelian categories with additive generators via Theorem \ref{theorem.6}.

\begin{proposition}\label{proposition.17}
Let $\Lambda$ be a commutative ring, and let $n$ be a positive integer. The category $\proj \Lambda$ is $n$-abelian if and only if $\Lambda$ is von Neumann regular.
\end{proposition}

\begin{proof}
Suppose that $\proj \Lambda$ is $n$-abelian. By Theorem \ref{theorem.4}, $\Lambda$ satisfies the axioms (R1), (R1$^{\op}$), (R2) and (R2$^{\op}$). In particular, $\Lambda$ is coherent and $\gldim (\mod \Lambda) < \infty$. Thus, it follows from \cite[Theorem 4.2.2 and Corollary 4.2.4]{MR0999133} that every principal ideal of $\Lambda$ is projective.

Take $x \in \Lambda$, and let $\langle x \rangle$ be the ideal generated by $x$, which is projective, by our previous remark. Let $f : \Lambda \to \Lambda$ be the multiplication by $x$ and consider $M = \Lambda / \langle x \rangle$, so that we have an exact sequence \[ \begin{tikzcd}
\Lambda \arrow[r, "f"] & \Lambda \arrow[r] & M \arrow[r] & 0
\end{tikzcd} \] in $\mod \Lambda$. By applying $(-)^{\ast}$ to it, we get an exact sequence \[ \begin{tikzcd}
0 \arrow[r] & M^{\ast} \arrow[r] & \Lambda^{\ast} \arrow[r, "f^{\ast}"] & \Lambda^{\ast} \arrow[r] & \Tr M \arrow[r] & 0
\end{tikzcd} \] in $\mod \Lambda$, where $\Tr M$ is a transpose of $M$. Since there is a commutative diagram \[ \begin{tikzcd}
\Lambda^{\ast} \arrow[r, "f^{\ast}"] \arrow[d, "\simeq"'] & \Lambda^{\ast} \arrow[d, "\simeq"] \\
\Lambda \arrow[r, "f"']                                   & \Lambda                           
\end{tikzcd} \] in $\mod \Lambda$ whose vertical arrows are isomorphisms, we deduce that $\Tr M \simeq M$.

Now, consider the short exact sequence \[ \begin{tikzcd}
0 \arrow[r] & \langle x \rangle \arrow[r] & \Lambda \arrow[r] & M \arrow[r] & 0
\end{tikzcd} \] in $\mod \Lambda$. Because $\langle x \rangle$ is projective, we have $\pdim M \leqslant 1$, hence $M$ is $n$-torsion free. In particular, it follows that $\Ext^{1}(M, \langle x \rangle) \simeq \Ext^{1}(\Tr M, \langle x \rangle) = 0$. Therefore, the above short exact sequence splits, so that the inclusion $\langle x \rangle \to \Lambda$ is a split monomorphism. As a left inverse $\Lambda \to \langle x \rangle$ of $\langle x \rangle \to \Lambda$ in $\mod \Lambda$ is given by multiplication by $xy$ for some $y \in \Lambda$, we conclude that $xyx = x$. Thus, $\Lambda$ is von Neumann regular.

The converse follows from Proposition \ref{proposition.6} since $\Lambda$ is von Neumann regular if and only if $\proj \Lambda$ is von Neumann regular.
\end{proof}

As a final observation, note that, in the proof of Proposition \ref{proposition.17}, we only had to use the fact that $M$ was $1$-torsion free, rather than $n$-torsion free. Moreover, as $\Lambda$ was assumed to be commutative, it would have been sufficient to assume that $\Lambda$ satisfies either the axiom (R1) or (R1$^{\op}$). Thus, we can improve Proposition \ref{proposition.17} as follows:

\begin{proposition}
Let $\Lambda$ be a commutative ring, and let $n$ be a positive integer. The category $\proj \Lambda$ has $n$-kernels and satisfies that every monomorphism is a kernel if and only if $\Lambda$ is von Neumann regular.
\end{proposition}

\begin{proof}
Suppose that $\proj \Lambda$ has $n$-kernels and that every monomorphism in $\proj \Lambda$ is a kernel. By Proposition \ref{proposition.8}, $\Lambda$ satisfies the axiom (R1), so that $\Lambda$ is coherent and $\gldim (\mod \Lambda) < \infty$. Furthermore, it follows from Proposition \ref{proposition.12} that $\proj \Lambda$ satisfies the axiom (F2$_{c}$) for $n = 1$. Due to the equivalence of categories $\mod (\proj \Lambda) \approx \mod \Lambda$, we deduce that $\Lambda$ satisfies the axiom (R2$_{c}$) for $n = 1$. Therefore, by Proposition \ref{proposition.36}, every $M \in \mod \Lambda$ with $\pdim M \leqslant 1$ is $1$-torsion free. Thus, the proof of Proposition \ref{proposition.17} applies, and we conclude that $\Lambda$ is von Neumann regular.

Since $\Lambda$ is von Neumann regular if and only if $\proj \Lambda$ is von Neumann regular, the converse follows from Proposition \ref{proposition.6} and Theorem \ref{theorem.8}.
\end{proof}

\appendix

\section{The global dimensions}\label{section.9}

The purpose of this appendix is to prove that if $\C$ is coherent, then the global dimensions of $\mod \C$ and $\mod \C^{\op}$ coincide. We remark that this is a well known result, at least for the case of modules over a coherent ring $\Lambda$, in which case the global dimensions of $\mod \Lambda$ and $\mod \Lambda^{\op}$ coincide with the weak dimension of $\Lambda$, by \cite[Proposition 1.1]{MR306265}. Also, the proof of the general result for modules over a coherent category $\C$ is not very different from the proof for the ring case. The content of this appendix is mainly based on \cite[Section 5]{MR2027559}.

Given an additive and idempotent complete category $\C$, let $\Yoneda : \C \to \mod \C$ be the Yoneda embedding, which is a covariant additive functor. Recall that $\mod \C$ is an additive category that has cokernels. It turns out that this category is universal with respect to these properties, in the following sense:

\begin{proposition}\label{proposition.18}
If $\D$ is an additive category that has cokernels and $\Psi : \C \to \D$ is an additive functor, then there is a unique (up to isomorphism) additive and cokernel preserving functor $\widetilde{\Psi} : \mod \C \to \D$ such that $\widetilde{\Psi} \circ \Yoneda \simeq \Psi$.
\end{proposition}

\begin{proof}
See \cite[Corollary 3.2]{MR2027559} or \cite[Universal Property 2.1]{MR1487973} or \cite[Proposition 2.1]{MR0212070}.
\end{proof}

Note that, by letting $\Yoneda^{\op} : \C^{\op} \to \mod \C^{\op}$ be the Yoneda embedding of $\C^{\op}$, which is a covariant additive functor,\footnote{Although, usually, $\Yoneda^{\op}$ is regarded as a contravariant functor from $\C$ to $\mod \C^{\op}$.} we deduce that the statement of Proposition \ref{proposition.18} also holds for the category $\mod \C^{\op}$, when considering $\Yoneda^{\op}$ instead of $\Yoneda$.

Now, if $G \in \mod \C^{\op}$, then $G : \C \to \Ab$ is an additive functor and it follows from Proposition \ref{proposition.18} that there is a unique (up to isomorphism) additive and cokernel preserving functor $\widetilde{G} : \mod \C \to \Ab$ such that $\widetilde{G} \circ \Yoneda \simeq G$. Similarly, if $F \in \mod \C$, then $F : \C^{\op} \to \Ab$ is an additive functor, so that there is a unique (up to isomorphism) additive and cokernel preserving functor $\widetilde{F} : \mod \C^{\op} \to \Ab$ such that $\widetilde{F} \circ \Yoneda^{\op} \simeq F$.  We denote $\widetilde{G} = - \otimes G$ and $\widetilde{F} = F \otimes -$. These functors define a bifunctor \[ - \otimes - : \mod \C \times \mod \C^{\op} \to \Ab, \] which we call the \textit{tensor product}. In particular, the tensor product of $F \in \mod \C$ and $G \in \mod \C^{\op}$ is given by the abelian group \[ F \otimes G \simeq (- \otimes G)(F) \simeq (F \otimes -)(G), \] which is unique up to isomorphism.

Next, suppose that $\C$ is coherent, so that $\mod \C$ and $\mod \C^{\op}$ are abelian categories (with enough projectives). In this case, given $F \in \mod \C$ and $G \in \mod \C^{\op}$, for each $i \geqslant 0$, we can consider the $i$th left derived functors of $F \otimes -$ and $- \otimes G$, which we denote by $\Tor_{i}(F,-)$ and $\Tor_{i}(-,G)$, respectively. These functors define a bifunctor \[ \Tor_{i}(-,-) : \mod \C \times \mod \C^{\op} \to \Ab, \] for each $i \geqslant 0$, since the tensor product $- \otimes -$ is a left balanced bifunctor.

The next two results can be found in \cite[Lemma 5.2 and Theorem 5.3]{MR2027559}, which are well known for modules over rings, see \cite[Proposition 3.2]{MR480688} and \cite[Proposition 2.2]{MR379599}.

\begin{lemma}\label{lemma.6}
Assume that $\C$ is coherent.\footnote{Actually, the condition that $\C$ is coherent is not needed, but we keep it for expository reasons.} For each $F,H \in \mod \C$, there is a canonical morphism $\varphi^{F}_{H} : H \otimes F^{\ast} \to \Hom(F,H)$ with $\Coker \varphi^{F}_{H} = \underHom (F,H)$. Moreover, $\varphi^{F}_{H}$ is natural both in $F$ and in $H$, and it is an isomorphism if $F$ or $H$ is projective.
\end{lemma}

\begin{proof}
Let \[ \begin{tikzcd}
{\C(-,U)} \arrow[r, "{\C(-,h)}"] &[1.2em] {\C(-,V)} \arrow[r] & H \arrow[r] & 0
\end{tikzcd} \] be a projective presentation of $H$ in $\mod \C$ with $h \in \C(U,V)$. Given that $- \otimes F^{\ast}$ preserves cokernels and $(- \otimes F^{\ast}) \circ \Yoneda \simeq F^{\ast}$, when applying $- \otimes F^{\ast}$ to the above projective presentation of $H$, we obtain an exact sequence \[ \begin{tikzcd}
F^{\ast}(U) \arrow[r, "F^{\ast}(h)"] &[1em] F^{\ast}(V) \arrow[r] & H \otimes F^{\ast} \arrow[r] & 0
\end{tikzcd} \] in $\Ab$. On the other hand, by applying $\Hom(F,-)$ to the same projective presentation of $H$, we get morphisms \[ \begin{tikzcd}
F^{\ast}(U) \arrow[r, "F^{\ast}(h)"] &[1em] F^{\ast}(V) \arrow[r] & {\Hom(F,H)}
\end{tikzcd} \] in $\Ab$ whose composition is zero. Consequently, there is a unique morphism $H \otimes F^{\ast} \to \Hom(F,H)$ in $\Ab$ which makes the diagram \[ \begin{tikzcd}
F^{\ast}(U) \arrow[r, "F^{\ast}(h)"] &[1em] F^{\ast}(V) \arrow[r] \arrow[rd] & H \otimes F^{\ast} \arrow[r] \arrow[d] & 0 \\
                                     &                                  & {\Hom(F,H)}                                               &  
\end{tikzcd} \] commute. It is not difficult to verify that the morphism $H \otimes F^{\ast} \to \Hom(F,H)$ obtained above does not depend on the choice of projective presentation of $H$, and we denote it by $\varphi^{F}_{H}$.

We leave it to the reader to check that $\varphi^{F}_{H}$ is natural both in $F$ and in $H$. Moreover, it is easy to conclude from the previous paragraph that if $F$ or $H$ is projective, then $\varphi^{F}_{H}$ is an isomorphism.

Finally, we show that $\Coker \varphi^{F}_{H} = \underHom (F,H)$. Since $\Coker \varphi^{F}_{H}$ is given by the quotient of abelian groups $\Hom(F,H) / \Image \varphi^{F}_{H}$, it suffices to verify that $\Image \varphi^{F}_{H} = \langle \proj \C \rangle (F,H)$. Well, let $\beta$ be the morphism $\C(-,V) \to H$ from the previous projective presentation of $H$ in $\mod \C$. By naturality, the diagram \[ \begin{tikzcd}
{\C(-,V) \otimes F^{\ast}} \arrow[d, "\beta \otimes F^{\ast}"'] \arrow[r, "{\varphi^{F}_{\C(-,V)}}"] &[1.6em] {\Hom(F,\C(-,V))} \arrow[d, "{\Hom(F,\beta)}"] \\
H \otimes F^{\ast} \arrow[r, "\varphi^{F}_{H}"']                                                      & {\Hom(F,H)}                                    
\end{tikzcd} \] in $\Ab$ is commutative. Furthermore, as $\C(-,V)$ is projective, $\varphi^{F}_{\C(-,V)}$ is an isomorphism, and because $- \otimes F^{\ast}$ preserves cokernels, $\beta \otimes F^{\ast}$ is an epimorphism. Consequently, $\Image \varphi^{F}_{H} = \Image \Hom(F,\beta)$, and it is easy to see that $\Image \Hom(F,\beta) = \langle \proj \C \rangle (F,H)$.
\end{proof}

Observe that if we had considered $\C^{\op}$-modules in Lemma \ref{lemma.6} instead of $\C$-modules, we would conclude that, for each $G,H \in \mod \C^{\op}$, there is a canonical morphism $G^{\ast} \otimes H \to \Hom(G,H)$ with the same properties that were claimed in Lemma \ref{lemma.6}.

\begin{proposition}\label{proposition.22}
Assume that $\C$ is coherent. For each $F,H \in \mod \C$, there is an isomorphism $\Tor_{1}(H, \Tr F) \simeq \underHom (F,H)$, which is natural both in $F$ and in $H$.
\end{proposition}

\begin{proof}
Let \[ \begin{tikzcd}
{\C(-,X)} \arrow[r, "{\C(-,f)}"] &[1.2em] {\C(-,Y)} \arrow[r, "\alpha"] & F \arrow[r] & 0
\end{tikzcd} \] be a projective presentation of $F$ in $\mod \C$ with $f \in \C(X,Y)$. By applying $(-)^{\ast}$ to it, we obtain an exact sequence \[ \begin{tikzcd}
0 \arrow[r] & F^{\ast} \arrow[r, "\alpha^{\ast}"] & {\C(-,Y)^{\ast}} \arrow[r, "{\C(-,f)^{\ast}}"] &[1.6em] {\C(-,X)^{\ast}} \arrow[r] & \Tr F \arrow[r] & 0
\end{tikzcd} \] in $\mod \C^{\op}$, where $\Tr F$ is a transpose of $F$ given by the cokernel of $\C(-,f)^{\ast}$. Since both $\C(-,Y)^{\ast}$ and $\C(-,X)^{\ast}$ are projective\footnote{As we have seen in Section \ref{section.3}, they are isomorphic to $\C(Y,-)$ and $\C(X,-)$, respectively.}, we can extend the above exact sequence to a projective resolution of $\Tr F$ in $\mod \C^{\op}$, and use it to compute $\Tor_{1}(H,\Tr F)$. However, because $H \otimes -$ preserves cokernels, we can easily see that \[ \Tor_{1}(H,\Tr F) = \frac{\Ker H \otimes \C(-,f)^{\ast}}{\Image H \otimes \alpha^{\ast}}. \]

Now, by Lemma \ref{lemma.6}, there is a commutative diagram \[ \begin{tikzcd}
            & H \otimes F^{\ast} \arrow[r, "H \otimes \alpha^{\ast}"] \arrow[d, "\varphi^{F}_{H}"] &[2.5em] {H \otimes \C(-,Y)^{\ast}} \arrow[r, "{H \otimes \C(-,f)^{\ast}}"] \arrow[d, "{\varphi^{\C(-,Y)}_{H}}"] &[4.2em] {H \otimes \C(-,X)^{\ast}} \arrow[d, "{\varphi^{\C(-,X)}_{H}}"] \\
0 \arrow[r] & {\Hom(F,H)} \arrow[r, "{\Hom(\alpha,H)}"']                                           & {\Hom(\C(-,Y),H)} \arrow[r, "{\Hom(\C(-,f),H)}"']                                                       & {\Hom(\C(-,X),H)}                                              
\end{tikzcd} \] in $\Ab$, where both $\varphi^{\C(-,Y)}_{H}$ and $\varphi^{\C(-,X)}_{H}$ are isomorphisms. Moreover, note that the bottom row of this diagram is exact. In this case, it is not difficult to deduce that \[ \frac{\Ker H \otimes \C(-,f)^{\ast}}{\Image H \otimes \alpha^{\ast}} \simeq \frac{\Hom(F,H)}{\Image \varphi^{F}_{H}} = \Coker \varphi^{F}_{H}. \] Therefore, it follows from Lemma \ref{lemma.6} that $\Tor_{1}(H, \Tr F) \simeq \underHom (F,H)$, and we leave it to the reader to verify that this isomorphism is natural both in $F$ and in $H$.
\end{proof}

We remark that if we had chosen to work with $\C^{\op}$-modules in Proposition \ref{proposition.22} instead of $\C$-modules, we would conclude that, for each $G,H \in \mod \C^{\op}$, there is an isomorphism $\Tor_{1}(\Tr G, H) \simeq \underHom (G,H)$, which is natural both in $G$ and in $H$.

\begin{corollary}\label{corollary.4}
Assume that $\C$ is coherent, let $F \in \mod \C$, and let $d$ be a nonnegative integer. Then $\Tor_{d+1}(F,-) = 0$ if and only if $\pdim F \leqslant d$.
\end{corollary}

\begin{proof}
If $\Tor_{d+1}(F,-) = 0$, then $\Tor_{1}(\Omega^{d} F,-) = 0$, where $\Omega^{d} F$ is $d$\textsuperscript{th} syzygy of $F$. Thus, it follows from Proposition \ref{proposition.22} that $\underHom(\Omega^{d} F, \Omega^{d} F) = 0$, hence $1_{\Omega^{d} F} \in \langle \proj \C \rangle$. By Proposition \ref{proposition.23}, we obtain that $\Omega^{d} F \in \proj \C$, which implies that $\pdim F \leqslant d$. The converse is straightforward.
\end{proof}

Note that the corresponding result for $\C^{\op}$-modules in Corollary \ref{corollary.4} is that if $G \in \mod \C^{\op}$, then $\Tor_{d+1}(-,G) = 0$ if and only if $\pdim G \leqslant d$.

We can now prove the main result of this appendix, which appears in \cite[Corollary 5.6]{MR2027559} and is well known, at least for the case of modules over a coherent ring, as we mentioned earlier.

\begin{theorem}\label{theorem.5}
Let $\C$ be an additive and idempotent complete category. If $\C$ is coherent, then $\gldim (\mod \C) = \gldim (\mod \C^{\op})$.
\end{theorem}

\begin{proof}
Observe that if $d$ is a nonnegative integer, then $\Tor_{d+1}(F,-) = 0$ for every $F \in \mod \C$ if and only if $\Tor_{d+1}(-,G) = 0$ for every $G \in \mod \C^{\op}$. Therefore, by Corollary \ref{corollary.4}, we have $\pdim F \leqslant d$ for every $F \in \mod \C$ if and only if $\pdim G \leqslant d$ for every $G \in \mod \C^{\op}$. Consequently, $\gldim (\mod \C) = \gldim (\mod \C^{\op})$.
\end{proof}

\section{More on the second axioms}\label{section.10}

In this short appendix, we point out some other equivalent ways of stating the functorial axioms (F2) and (F2$^{\op}$) of an $n$-abelian category, which could be interesting for further investigation.

For a left and right coherent category $\C$, consider the following axioms, respectively:

\begin{enumerate}
    \item[(F2$_{h}$)] Every $F \in \mod \C^{\op}$ with $F^{\ast} = 0$ satisfies $\Ext^{i}(F,\C(X,-)) = 0$ for all $1 \leqslant i \leqslant n$ and all $X \in \C$.
    \item[(F2$_{h}^{\op}$)] Every $F \in \mod \C$ with $F^{\ast} = 0$ satisfies $\Ext^{i}(F,\C(-,X)) = 0$ for all $1 \leqslant i \leqslant n$ and all $X \in \C$.
\end{enumerate}

\begin{proposition}\label{proposition.37}
If $\C$ is coherent, then the axioms \textup{(F2)} and \textup{(F2$^{\op}$)} are equivalent to \textup{(F2$_{h}$)} and \textup{(F2$_{h}^{\op}$)}, respectively.
\end{proposition}

\begin{proof}
Follows easily from Lemma \ref{lemma.4} and the fact that if $F \in \mod \C$ satisfies $\pdim F \leqslant 1$, then there is a transpose $\Tr F$ of $F$ such that $(\Tr F)^{\ast} = 0$.
\end{proof}

The axioms (F2$_{h}$) and (F2$_{h}^{\op}$) are considered, for example, in \cite[Theorem 1.2]{MR3638352} and \cite[Remark 2.25]{Klapproth}, while analogous conditions also appear in \cite[Definition 4.1]{MR4392222} and \cite[Definition 4.1]{EbrahimiNasr-Isfahani}. Although we have not explored these axioms in this paper, they were actually used indirectly in the proofs of Propositions \ref{proposition.20} and \ref{proposition.30}.

Another way of stating the axioms (F2$_{h}$) and (F2$_{h}^{\op}$) is in terms of the ``grade'' of objects in $\mod \C^{\op}$ and $\mod \C$, respectively. Let us recall this classical definition, which can be considered in a general abelian category.

Let $\A$ be an abelian category. Given $X \in \A$, the \textit{grade} of $X$, denoted by $\grade X$, is the supremum of the set of nonnegative integers $j$ for which $\Ext_{\A}^{i}(X,P) = 0$ for every $0 \leqslant i < j$ and every projective object $P \in \A$, where we agree that $\Ext_{\A}^{0}(X,P) = \A(X,P)$. If the supremum does not exist, then we say that $\grade X = \infty$.

Observe that if $\C$ is left coherent, then $F \in \mod \C^{\op}$ satisfies $F^{\ast} = 0$ if and only if $\grade F \geqslant 1$. Furthermore, the axiom (F2$_{h}$) for $\C$ can be rephrased as ``every $F \in \mod \C^{\op}$ with $\grade F \geqslant 1$ satisfies $\grade F \geqslant n + 1$''. Dually, if $\C$ is right coherent, then the axiom (F2$_{h}^{\op}$) for $\C$ can be rephrased as ``every $F \in \mod \C$ with $\grade F \geqslant 1$ satisfies $\grade F \geqslant n + 1$''. Thus, motivated by \cite[Theorem 5.10]{MR4392222}, for a left and right coherent category $\C$, consider the following axioms, respectively:

\begin{enumerate}
    \item[(F2$_{i}$)] The category $\mod \C^{\op}$ has no objects of grade $1, \ldots, n$.
    \item[(F2$_{i}^{\op}$)] The category $\mod \C$ has no objects of grade $1, \ldots, n$.
\end{enumerate}

\begin{proposition}\label{proposition.44}
If $\C$ is left coherent, then the axioms \textup{(F2$_{h}$)} and \textup{(F2$_{i}$)} are equivalent. Dually, if $\C$ is right coherent, then the axioms \textup{(F2$_{h}^{\op}$)} and \textup{(F2$_{i}^{\op}$)} are equivalent.
\end{proposition}

\begin{proof}
Follows from the above discussion.
\end{proof}

While we do not explore the axioms (F2$_{i}$) and (F2$_{i}^{\op}$) in this paper, let us state the following characterization of $n$-abelian categories:

\begin{theorem}\label{theorem.14}
Let $\C$ be an additive and idempotent complete category. Then $\C$ is $n$-abelian if and only if $\C$ is coherent, $\gldim (\mod \C) \leqslant n + 1$, and every nonzero object in $\mod \C$ and in $\mod \C^{\op}$ has grade either $0$ or $n + 1$.
\end{theorem}

\begin{proof}
Follows directly from Theorems \ref{theorem.2} and \ref{theorem.5}, Propositions \ref{proposition.37} and \ref{proposition.44}, and Lemma \ref{lemma.1}.
\end{proof}

We can, of course, restrict the above results to the ring case, as it was done with our previous results in Section \ref{section.8}. For instance, Theorem \ref{theorem.14} becomes the following:

\begin{theorem}
Let $\Lambda$ be a ring, and let $n$ be a positive integer. The category $\proj \Lambda$ is $n$-abelian if and only if $\Lambda$ is coherent, has weak dimension at most $n + 1$, and every nonzero object in $\mod \Lambda$ and in $\mod \Lambda^{\op}$ has grade either $0$ or $n + 1$.
\end{theorem}

\begin{proof}
Follows from Theorem \ref{theorem.14}, and from the equivalences $\mod (\proj \Lambda) \approx \mod \Lambda$ and $\mod (\proj \Lambda)^{\op} \approx \mod \Lambda^{\op}$.
\end{proof}

\section{Some applications}\label{section.11}

In this appendix, we use the functorial perspective to prove a few known results on $n$-abelian categories. Its purpose is to further expose the reader to the functorial approach, and also to offer an alternative proof of Proposition \ref{proposition.6}.

First, recall that a category is called \textit{balanced} when all of its morphisms which are both monomorphisms and epimorphisms are isomorphisms.

\begin{proposition}\label{proposition.41}
Every $n$-abelian category is balanced.
\end{proposition}

\begin{proof}
Let $\C$ be an $n$-abelian category, so that $\C$ satisfies the axioms (F1), (F1$^{\op}$), (F2) and (F2$^{\op}$), by Theorem \ref{theorem.2}. Suppose that $f \in \C(X,Y)$ is a morphism in $\C$ which is both a monomorphism and an epimorphism. Since $f$ is a monomorphism, the sequence \[ \begin{tikzcd}
0 \arrow[r] & {\C(-,X)} \arrow[r, "{\C(-,f)}"] &[1.2em] {\C(-,Y)} \arrow[r] & \mr(f) \arrow[r] & 0
\end{tikzcd} \] is exact in $\mod \C$, and then $\pdim \mr(f) \leqslant 1$. Therefore, $\mr(f)$ is $n$-torsion free, and hence $1$-torsion free. Now, as $f$ is an epimorphism, the sequence \[ \begin{tikzcd}
0 \arrow[r] & {\C(Y,-)} \arrow[r, "{\C(f,-)}"] &[1.2em] {\C(X,-)}
\end{tikzcd} \] is exact in $\mod \C^{\op}$. Thus, it follows from Proposition \ref{proposition.14} that \[ \begin{tikzcd}
{\C(-,X)} \arrow[r, "{\C(-,f)}"] &[1.2em] {\C(-,Y)} \arrow[r] & 0
\end{tikzcd} \] is exact in $\mod \C$, which implies that $f$ is a split epimorphism. By duality, we deduce that $f$ is a split monomorphism. Consequently, $f$ is an isomorphism.
\end{proof}

Let $m$ be a positive integer. We say that an $m$-exact sequence \[ \begin{tikzcd}
Z_{m+1} \arrow[r, "h_{m+1}"] &[0.85em] Z_{m} \arrow[r, "h_{m}"] &[0.2em] \cdots \arrow[r, "h_{2}"] & Z_{1} \arrow[r, "h_{1}"] & Z_{0}
\end{tikzcd} \] in $\C$ \textit{splits} if $h_{m+1}$ is a split monomorphism and $h_{1}$ is a split epimorphism.

The next result is from \cite[Proposition 2.6]{MR3519980}.

\begin{proposition}\label{proposition.40}
Let $\C$ be an additive and idempotent complete category,\footnote{Actually, the proof we present only requires $\C$ to be a preadditive category.} let $m$ be a positive integer, and let \[ \begin{tikzcd}
Z_{m+1} \arrow[r, "h_{m+1}"] &[0.85em] Z_{m} \arrow[r, "h_{m}"] &[0.2em] \cdots \arrow[r, "h_{2}"] & Z_{1} \arrow[r, "h_{1}"] & Z_{0}
\end{tikzcd} \] be an $m$-exact sequence in $\C$. The following are equivalent:
\begin{enumerate}
    \item[(a)] The above sequence splits.
    \item[(b)] $h_{1}$ is a split epimorphism.
    \item[(c)] $h_{m+1}$ is a split monomorphism.
\end{enumerate}
\end{proposition}

\begin{proof}
Suppose that $h_{1}$ is a split epimorphism. Then $\C(-,h_{1})$ is a split epimorphism and the sequence \[ \begin{tikzcd}
0 \arrow[r] & {\C(-,Z_{m+1})} \arrow[r, "{\C(-,h_{m+1})}"] &[2.65em] \cdots \arrow[r, "{\C(-,h_{2})}"] &[1.5em] {\C(-,Z_{1})} \arrow[r, "{\C(-,h_{1})}"] &[1.5em] {\C(-,Z_{0})} \arrow[r] & 0
\end{tikzcd} \] is exact in $\Mod \C$. Since each $\C(-,Z_{i})$ is projective in $\Mod \C$, it is easy to conclude that $\C(-,h_{m+1})$ is a split monomorphism, which implies that $h_{m+1}$ is a split monomorphism.

Thus, (b) implies (c). By duality, (c) implies (b). Hence (a) is equivalent to both (b) and (c).
\end{proof}

Observe that if we extend the notion of an $m$-exact sequence to include the case $m = 0$, then we obtain from Proposition \ref{proposition.41} that $0$-exact sequences split in $n$-abelian categories. In addition to this fact, we have the following result, which is well known:

\begin{proposition}\label{proposition.32}
Let $\C$ be an $n$-abelian category, and let $m$ be a positive integer. If $m \neq n$, then every $m$-exact sequence in $\C$ splits.
\end{proposition}

\begin{proof}
To begin with, recall that $\C$ satisfies the axioms (F1), (F1$^{\op}$), (F2) and (F2$^{\op}$), by Theorem \ref{theorem.2}. Next, let \[ \begin{tikzcd}
Z_{m+1} \arrow[r, "h_{m+1}"] &[0.85em] Z_{m} \arrow[r, "h_{m}"] &[0.2em] \cdots \arrow[r, "h_{2}"] & Z_{1} \arrow[r, "h_{1}"] & Z_{0}
\end{tikzcd} \] be an $m$-exact sequence in $\C$. By Proposition \ref{proposition.11}, $\mr(h_{1})$ is $(m+1)$-spherical.

If $m < n$, then we deduce from Theorem \ref{theorem.7} that $\mr(h_{1})$ is $(n-m)$-torsion free, and hence $1$-torsion free. Thus, as \[ \begin{tikzcd}
0 \arrow[r] & {\C(Z_{0},-)} \arrow[r, "{\C(h_{1},-)}"] &[1.5em] {\C(Z_{1},-)}
\end{tikzcd} \] is exact in $\mod \C^{\op}$, we conclude from Proposition \ref{proposition.14} that \[ \begin{tikzcd}
{\C(-,Z_{1})} \arrow[r, "{\C(-,h_{1})}"] &[1.5em] {\C(-,Z_{0})} \arrow[r] & 0
\end{tikzcd} \] is exact in $\mod \C$. Hence $h_{1}$ is a split epimorphism, which implies that the above $m$-exact sequence splits, by Proposition \ref{proposition.40}.

Now, assume that $n < m$. Since $\mr(h_{1})$ is $(m+1)$-spherical, it follows from Lemma \ref{lemma.1} that either $\mr(h_{1})$ is projective or $\pdim \mr(h_{1}) = m + 1$. But as $n + 1 < m + 1$ and $\gldim (\mod \C) \leqslant n + 1$, it is not possible that $\pdim \mr(h_{1}) = m + 1$, hence $\mr(h_{1})$ is projective. In this case, given that {\relsize{-0.5} \[ \begin{tikzcd}
0 \arrow[r] &[-0.1em] {\C(-,Z_{m+1})} \arrow[r, "{\C(-,h_{m+1})}"] &[2.85em] {\C(-,Z_{m})} \arrow[r, "{\C(-,h_{m})}"] &[2em] \cdots \arrow[r, "{\C(-,h_{1})}"] &[1.8em] {\C(-,Z_{0})} \arrow[r] &[-0.1em] \mr(h_{1}) \arrow[r] &[-0.1em] 0
\end{tikzcd} \]} \hspace{-0.667em} is a projective resolution of $\mr(h_{1})$ in $\mod \C$, we can easily conclude that $\C(-,h_{m+1})$ is a split monomorphism, which implies that $h_{m+1}$ is a split monomorphism. Therefore, the previous $m$-exact sequence splits, by Proposition \ref{proposition.40}.
\end{proof}

The next result, which is essentially \cite[Theorem 3.9]{MR3519980}, finalizes the discussion concerning the splitting of $m$-exact sequences in $n$-abelian categories.

\begin{proposition}\label{proposition.33}
Let $\C$ be an additive and idempotent complete category. Then $\C$ is $n$-abelian and every $n$-exact sequence in $\C$ splits if and only if $\C$ is von Neumann regular.
\end{proposition}

\begin{proof}
Suppose that $\C$ is $n$-abelian and every $n$-exact sequence in $\C$ splits. Let $f \in \C(X,Y)$ be a monomorphism in $\C$. Since $\C$ satisfies the axiom (A2), by taking an $n$-cokernel of $f$, we get an $n$-exact sequence \[ \begin{tikzcd}
X \arrow[r, "f"] & Y \arrow[r, "g_{1}"] & Y_{1} \arrow[r, "g_{2}"] & \cdots \arrow[r, "g_{n}"] & Y_{n}
\end{tikzcd} \] in $\C$, which splits. Therefore, every monomorphism in $\C$ is a split monomorphism. Consequently, every monomorphism in $\mod \C$ whose domain and codomain are projective is a split monomorphism. Since we know from Proposition \ref{proposition.2} that $\mod \C$ is an abelian category (with enough projectives) and $\gldim (\mod \C) < \infty$, we can then easily deduce that $\gldim (\mod \C) = 0$. Hence $\C$ is von Neumann regular, by Proposition \ref{proposition.4}.

Conversely, assume that $\C$ is von Neumann regular. By Proposition \ref{proposition.4}, $\C$ is coherent with $\gldim (\mod \C) = \gldim (\mod \C^{\op}) = 0$. Thus, it follows from Theorem \ref{theorem.2} that $\C$ is $n$-abelian for every positive integer $n$. In particular, $\C$ is $(n+1)$-abelian, which implies that every $n$-exact sequence in $\C$ splits, by Proposition \ref{proposition.32}.
\end{proof}

Finally, note that Propositions \ref{proposition.32} and \ref{proposition.33} give another proof of Proposition \ref{proposition.6}.

\section*{Acknowledgments}

The author would like to thank his Ph.D. advisor Alex Martsinkovsky for his constant encouragement and discussions on the content of this paper. The author would also like to thank Sondre Kvamme and Carlo Klapproth for making him aware of the papers \cite{EbrahimiNasr-Isfahani}, \cite{MR4392222} and \cite{MR3659323}, \cite{MR3836680}, \cite{Klapproth}, respectively. The author is also grateful to the referee for the thoughtful comments, which resulted in a significant improvement of this text.


\begin{thebibliography}{10}

\bibitem{MR0212070}
Maurice Auslander.
\newblock Coherent functors.
\newblock In {\em Proc. {C}onf. {C}ategorical {A}lgebra ({L}a {J}olla, {C}alif., 1965)}, pages 189--231. Springer-Verlag New York, Inc., New York, 1966.

\bibitem{auslander1971representation}
Maurice Auslander.
\newblock Representation dimension of artin algebras.
\newblock {\em Lecture notes, Queen Mary College, London}, 1971.
\newblock Reprinted in \cite[pages 505--574]{MR1674397}.

\bibitem{MR349747}
Maurice Auslander.
\newblock Representation theory of {A}rtin algebras {I}.
\newblock {\em Comm. Algebra}, 1:177--268, 1974.

\bibitem{artin.algebras.ii}
Maurice Auslander.
\newblock Representation theory of {A}rtin algebras {II}.
\newblock {\em Comm. Algebra}, 1:269--310, 1974.

\bibitem{MR480688}
Maurice Auslander.
\newblock Functors and morphisms determined by objects.
\newblock In {\em Representation theory of algebras ({P}roc. {C}onf., {T}emple {U}niv., {P}hiladelphia, {P}a., 1976)}, volume Vol. 37 of {\em Lect. Notes Pure Appl. Math.}, pages 1--244. Dekker, New York-Basel, 1978.

\bibitem{MR1674397}
Maurice Auslander.
\newblock {\em Selected works of {M}aurice {A}uslander. {P}art 1}.
\newblock American Mathematical Society, Providence, RI, 1999.
\newblock Edited and with a foreword by Idun Reiten, Sverre O. Smal\o , and \O yvind Solberg.

\bibitem{MR0269685}
Maurice Auslander and Mark Bridger.
\newblock {\em Stable module theory}.
\newblock Memoirs of the American Mathematical Society, No. 94. American Mathematical Society, Providence, R.I., 1969.

\bibitem{MR342505}
Maurice Auslander and Idun Reiten.
\newblock Stable equivalence of dualizing {$R$}-varieties.
\newblock {\em Advances in Math.}, 12:306--366, 1974.

\bibitem{MR379599}
Maurice Auslander and Idun Reiten.
\newblock Representation theory of {A}rtin algebras. {III}. {A}lmost split sequences.
\newblock {\em Comm. Algebra}, 3:239--294, 1975.

\bibitem{MR2027559}
Apostolos Beligiannis.
\newblock On the {F}reyd categories of an additive category.
\newblock {\em Homology Homotopy Appl.}, 2:147--185, 2000.

\bibitem{MR3406183}
Apostolos Beligiannis.
\newblock Relative homology, higher cluster-tilting theory and categorified {A}uslander-{I}yama correspondence.
\newblock {\em J. Algebra}, 444:367--503, 2015.

\bibitem{MR2909639}
George~M. Bergman.
\newblock On diagram-chasing in double complexes.
\newblock {\em Theory Appl. Categ.}, 26:No. 3, 60--96, 2012.

\bibitem{MR2116320}
Francis Borceux and Ji\v{r}\'{i} Rosick\'{y}.
\newblock On von {N}eumann varieties.
\newblock {\em Theory Appl. Categ.}, 13:No. 1, 5--26, 2004.

\bibitem{MR269686}
R.~R. Colby and E.~A. Rutter, Jr.
\newblock Generalizations of {${\rm QF}-3$} algebras.
\newblock {\em Trans. Amer. Math. Soc.}, 153:371--386, 1971.

\bibitem{EbrahimiNasr-Isfahani}
Ramin Ebrahimi and Alireza Nasr-Isfahani.
\newblock Higher auslander correspondence for exact categories, 2021.
\newblock \href{https://arxiv.org/abs/2108.13645}{arXiv:2108.13645}.

\bibitem{MR4514466}
Ramin Ebrahimi and Alireza Nasr-Isfahani.
\newblock Higher {A}uslander's formula.
\newblock {\em Int. Math. Res. Not. IMRN}, (22):18186--18203, 2022.

\bibitem{MR314906}
Gary~L. Eerkes.
\newblock Codominant dimension of rings and modules.
\newblock {\em Trans. Amer. Math. Soc.}, 176:125--139, 1973.

\bibitem{MR3659323}
Haruhisa Enomoto.
\newblock Classifying exact categories via {W}akamatsu tilting.
\newblock {\em J. Algebra}, 485:1--44, 2017.

\bibitem{MR3836680}
Haruhisa Enomoto.
\newblock Classifications of exact structures and {C}ohen-{M}acaulay-finite algebras.
\newblock {\em Adv. Math.}, 335:838--877, 2018.

\bibitem{MR166240}
Peter Freyd.
\newblock {\em Abelian categories. {A}n introduction to the theory of functors}.
\newblock Harper's Series in Modern Mathematics. Harper \& Row, Publishers, New York, 1964.

\bibitem{MR0209333}
Peter Freyd.
\newblock Representations in abelian categories.
\newblock In {\em Proc. {C}onf. {C}ategorical {A}lgebra ({L}a {J}olla, {C}alif., 1965)}, pages 95--120. Springer-Verlag New York, Inc., New York, 1966.

\bibitem{MR3021448}
Christof Geiss, Bernhard Keller, and Steffen Oppermann.
\newblock {$n$}-angulated categories.
\newblock {\em J. Reine Angew. Math.}, 675:101--120, 2013.

\bibitem{MR0999133}
Sarah Glaz.
\newblock {\em Commutative coherent rings}, volume 1371 of {\em Lecture Notes in Mathematics}.
\newblock Springer-Verlag, Berlin, 1989.

\bibitem{VitorGulisz}
Vitor Gulisz.
\newblock First steps in higher {A}uslander--{R}eiten theory.
\newblock Master's thesis, Universidade Federal do Paran{\'a}, 2021.
\newblock \url{https://hdl.handle.net/1884/71956}.

\bibitem{MR206049}
Manabu Harada.
\newblock {${\rm QF}-3$} and semi-primary {${\rm PP}$}-rings {II}.
\newblock {\em Osaka J. Math.}, 3:21--27, 1966.

\bibitem{MR116045}
Alex Heller.
\newblock The loop-space functor in homological algebra.
\newblock {\em Trans. Amer. Math. Soc.}, 96:382--394, 1960.

\bibitem{MR4392222}
Ruben Henrard, Sondre Kvamme, and Adam-Christiaan van Roosmalen.
\newblock Auslander's formula and correspondence for exact categories.
\newblock {\em Adv. Math.}, 401:Paper No. 108296, 65, 2022.

\bibitem{MR4188310}
Martin Herschend, Yu~Liu, and Hiroyuki Nakaoka.
\newblock {$n$}-exangulated categories ({I}): {D}efinitions and fundamental properties.
\newblock {\em J. Algebra}, 570:531--586, 2021.

\bibitem{MR340327}
David~A. Hill.
\newblock On dominant and codominant dimension of {${\rm QF}-3$} rings.
\newblock {\em Pacific J. Math.}, 49:93--99, 1973.

\bibitem{MR2095628}
Osamu Iyama.
\newblock The relationship between homological properties and representation theoretic realization of {A}rtin algebras.
\newblock {\em Trans. Amer. Math. Soc.}, 357(2):709--734, 2005.

\bibitem{MR2298819}
Osamu Iyama.
\newblock Higher-dimensional {A}uslander-{R}eiten theory on maximal orthogonal subcategories.
\newblock {\em Adv. Math.}, 210(1):22--50, 2007.

\bibitem{MR3638352}
Osamu Iyama and Gustavo Jasso.
\newblock Higher {A}uslander correspondence for dualizing {$R$}-varieties.
\newblock {\em Algebr. Represent. Theory}, 20(2):335--354, 2017.

\bibitem{MR112904}
J.~P. Jans.
\newblock Projective injective modules.
\newblock {\em Pacific J. Math.}, 9:1103--1108, 1959.

\bibitem{MR3519980}
Gustavo Jasso.
\newblock {$n$}-{A}belian and {$n$}-exact categories.
\newblock {\em Math. Z.}, 283(3-4):703--759, 2016.

\bibitem{Klapproth}
Carlo Klapproth.
\newblock On $n$-exact categories {I}: The existence and uniqueness of maximal $n$-exact structures, 2024.
\newblock \href{https://arxiv.org/abs/2410.05242}{arXiv:2410.05242}.

\bibitem{MR1487973}
Henning Krause.
\newblock Functors on locally finitely presented additive categories.
\newblock {\em Colloq. Math.}, 75(1):105--132, 1998.

\bibitem{MR3431480}
Henning Krause.
\newblock Krull-{S}chmidt categories and projective covers.
\newblock {\em Expo. Math.}, 33(4):535--549, 2015.

\bibitem{MR4327095}
Henning Krause.
\newblock {\em Homological theory of representations}, volume 195 of {\em Cambridge Studies in Advanced Mathematics}.
\newblock Cambridge University Press, Cambridge, 2022.

\bibitem{MR4301013}
Sondre Kvamme.
\newblock Axiomatizing subcategories of {A}belian categories.
\newblock {\em J. Pure Appl. Algebra}, 226(4):Paper No. 106862, 27, 2022.

\bibitem{MR1653294}
T.~Y. Lam.
\newblock {\em Lectures on modules and rings}, volume 189 of {\em Graduate Texts in Mathematics}.
\newblock Springer-Verlag, New York, 1999.

\bibitem{MR1838439}
T.~Y. Lam.
\newblock {\em A first course in noncommutative rings}, volume 131 of {\em Graduate Texts in Mathematics}.
\newblock Springer-Verlag, New York, second edition, 2001.

\bibitem{MR224649}
B.~H. Maddox.
\newblock Absolutely pure modules.
\newblock {\em Proc. Amer. Math. Soc.}, 18:155--158, 1967.

\bibitem{MR306265}
D.~George McRae.
\newblock Homological dimensions of finitely presented modules.
\newblock {\em Math. Scand.}, 28:70--76, 1971.

\bibitem{MR294409}
Charles Megibben.
\newblock Absolutely pure modules.
\newblock {\em Proc. Amer. Math. Soc.}, 26:561--566, 1970.

\bibitem{MR0294454}
Barry Mitchell.
\newblock Rings with several objects.
\newblock {\em Advances in Math.}, 8:1--161, 1972.

\bibitem{MR224656}
Bruno~J. M\"uller.
\newblock The classification of algebras by dominant dimension.
\newblock {\em Canadian J. Math.}, 20:398--409, 1968.

\bibitem{MR233854}
Bruno~J. M\"uller.
\newblock Dominant dimension of semi-primary rings.
\newblock {\em J. Reine Angew. Math.}, 232:173--179, 1968.

\bibitem{MR2471947}
Wolfgang Rump.
\newblock A counterexample to {R}aikov's conjecture.
\newblock {\em Bull. Lond. Math. Soc.}, 40(6):985--994, 2008.

\bibitem{MR4093080}
Wolfgang Rump.
\newblock The abelian closure of an exact category.
\newblock {\em J. Pure Appl. Algebra}, 224(10):106395, 29, 2020.

\bibitem{MR3537819}
Jeremy Russell.
\newblock Applications of the defect of a finitely presented functor.
\newblock {\em J. Algebra}, 465:137--169, 2016.

\bibitem{MR258888}
Bo~Stenstr\"om.
\newblock Coherent rings and {$FP$}-injective modules.
\newblock {\em J. London Math. Soc. (2)}, 2:323--329, 1970.

\bibitem{MR349740}
Hiroyuki Tachikawa.
\newblock {\em Quasi-{F}robenius rings and generalizations. {${\rm QF}-3$} and {${\rm QF}-1$} rings}, volume Vol. 351 of {\em Lecture Notes in Mathematics}.
\newblock Springer-Verlag, Berlin-New York, 1973.
\newblock Notes by Claus Michael Ringel.

\end{thebibliography}
\end{document}